\documentclass[11pt]{amsart}

\usepackage[T1]{fontenc}
\usepackage[utf8]{inputenc}
\usepackage[american]{babel}
\usepackage[babel]{microtype}
\usepackage{amsfonts} % Standards
\usepackage{amsthm,amssymb,mathrsfs,setspace,pstricks}%amsmath, latexsym,footmisc
\usepackage{dynkin-diagrams}
\usepackage{amsmath}
\usepackage{booktabs}
\usepackage{array}
\usepackage[autostyle, english = american]{csquotes}
\MakeOuterQuote{"}
\usepackage{tikz}
\usepackage{tikz-cd}
\usetikzlibrary{positioning,arrows.meta}
\usetikzlibrary{calc}
\usetikzlibrary{decorations.pathreplacing}

\usepackage{mathrsfs}

\usepackage{mathtools}

\usepackage{dynkin-diagrams}

\usepackage{enumerate}

\usepackage{tikz}
\usetikzlibrary{calc}
\usetikzlibrary{decorations.pathreplacing}

\usepackage{xspace}

\usepackage{xcolor}
\usepackage{array}
\usepackage{colortbl}
\usepackage{booktabs}
\usepackage{multirow} 
\usepackage{tikz-cd}
\usepackage[noadjust]{cite}
\usepackage[hidelinks,bookmarksnumbered]{hyperref}
\usepackage[noabbrev,capitalise]{cleveref}

\usepackage{orcidlink}

\newtheorem{theorem}{Theorem}[section]
\newtheorem{lemma}[theorem]{Lemma}
\newtheorem{remark}[theorem]{Remark}
\newtheorem{proposition}[theorem]{Proposition}
\newtheorem{corollary}[theorem]{Corollary}

\DeclareMathOperator{\End}{End}
\DeclareMathOperator{\Der}{Der}
\DeclareMathOperator{\ad}{ad}
\DeclareMathOperator{\Ker}{Ker}

\DeclareMathOperator{\Hom}{Hom}

\newcommand{\ord}{\mathrm{ord}}
\newcommand{\rk}{\operatorname{rank}}
\newcommand{\cl}{\colon}
\newcommand{\bb}[1]{\mathbb{#1}}
\newcommand{\ft}{\mathfrak{t}}
\newcommand{\op}[1]{\operatorname{#1}}

\newcommand{\CC}{\mathbb{C}}
\newcommand{\Z}{\mathbb{Z}}

\newcommand{\frk}[1]{\mathfrak{#1}}
\newcommand{\g}{\mathfrak{g}}
\newcommand{\tg}{\tilde{\mathfrak{g}}}

\newcommand{\tbG}{\widetilde{\mathbf{G}}}

\newcommand{\bmu}{\boldsymbol{\mu}}

\newcommand{\h}{\mathfrak{h}}
\newcommand{\Aut}{\operatorname{Aut}}
\newcommand{\bAut}{\operatorname{\mathbf{Aut}}}
\newcommand{\Out}{\operatorname{Out}}
\newcommand{\Int}{\operatorname{Int}}
\newcommand{\Ad}{\operatorname{Ad}}
\newcommand{\Lie}{\mathscr{L}}
\newcommand{\bG}{\mathbf{G}}

\newcommand{\thistheoremname}{}
\newtheorem*{genericthm*}{\thistheoremname}
\newenvironment{namedthm*}[1]
{\renewcommand{\thistheoremname}{#1}%
\begin{genericthm*}}
{\end{genericthm*}}

\DeclareFontFamily{U}{mathx}{}
\DeclareFontShape{U}{mathx}{m}{n}{<-> mathx10}{}
\DeclareSymbolFont{mathx}{U}{mathx}{m}{n}
\DeclareMathAccent{\widehat}{0}{mathx}{"70}
\DeclareMathAccent{\widecheck}{0}{mathx}{"71}

\theoremstyle{definition}
\newtheorem{definition}[theorem]{Definition}

\newtheorem{example}[theorem]{Example}

\AddToHook{env/lemma/begin}{\crefalias{theorem}{lemma}}
\AddToHook{env/corollary/begin}{\crefalias{theorem}{corollary}}
\AddToHook{env/proposition/begin}{\crefalias{theorem}{proposition}}
\AddToHook{env/definition/begin}{\crefalias{theorem}{definition}}
\AddToHook{env/remark/begin}{\crefalias{theorem}{remark}}
\AddToHook{env/example/begin}{\crefalias{theorem}{example}}
\AddToHook{env/problem/begin}{\crefalias{theorem}{problem}}

\usepackage{tikz}
\usetikzlibrary{arrows.meta}
\tikzset{
  dynnode/.style={circle, draw, fill=white, inner sep=0pt, minimum size=7.5pt},
  dynedge/.style={line width=0.75pt},
  dynDoubleRight/.style={double distance=1.4pt, line width=0.3pt, -{Latex[length=1.7mm,width=1.3mm]}},
  dynDoubleLeft/.style={double distance=1.4pt, line width=0.3pt, {Latex[length=1.7mm,width=1.3mm]}-}
}

\title[Gradings by cyclic groups on classical simple Lie algebras]{Gradings by cyclic groups on classical simple Lie algebras in prime characteristics}
\author{Mikhail Kochetov
\orcidlink{https://orcid.org/0000-0002-0427-7926}
}
\email{mikhail@mun.ca}
\address{Department of Mathematics and Statistics, Memorial University, St. John's, NL, A1C5S7, Canada}
\thanks{We acknowledge the support of the Natural Sciences and Engineering Research Council of Canada (NSERC), RGPIN-2018-04883. Cette recherche a été financée par le Conseil de recherches en sciences naturelles et en génie du Canada (CRSNG), RGPIN-2018-04883.}

\author{Vishal Yadav
\orcidlink{https://orcid.org/0000-0003-0038-7949}
}
\email{vkyadav@mun.ca}
\address{Department of Mathematics and Statistics, Memorial University, St. John's, NL, A1C5S7, Canada}
\thanks{This paper is based on the second author's Ph.D. thesis, written under the supervision of of the first author. The second author acknowledges support by a graduate fellowship from Memorial University, Canada.}



\begin{document}

\begin{abstract} \footnotesize
We classify, up to isomorphism, gradings by finite cyclic groups on classical simple Lie algebras \(\g\) over an algebraically closed field of arbitrary characteristic. Using the smoothness of the automorphism group scheme \(\op{\mathbf{Aut}}\g\) and correspondence between \(\mathbb Z_m\)-gradings and morphisms \(\boldsymbol{\mu}_m\to\operatorname{\mathbf{Aut}}\g\), we express the classification as an orbit problem for certain Weyl-type groups. More generally, for the affine group scheme \(\mathbf G\) associated to a semisimple algebraic group \(G\) and the constant group scheme \(\mathbf \Gamma_0\) associated to a subgroup \(\Gamma_0\) of the automorphism group of the based root datum of \(G\), we consider the classification of morphisms \(\boldsymbol{\mu}_m\to \mathbf G \rtimes \mathbf\Gamma_0\) up to conjugation by \(G \rtimes \Gamma_0\). We show that the classification in characteristic \(p\) is the same as in characteristic \(0\) except that, in characteristic \(p\), only elements of \(\Gamma_0\) whose order is prime to \(p\) can occur. For \(\Z_m\)-gradings on \(\g\), this extends the classification by Kac coordinates to arbitrary characteristic, with the caveat that only diagram automorphisms of order prime to \(p\) are allowed and, if \(p=2\) or \(3\), the type of \(\Aut\g\) is not always the same as the type of \(\g\).
\end{abstract}

\maketitle

\section{Introduction}

Let \(\mathcal A\) be an algebra, not necessarily associative, and let \(H\) be an abstract group, written multiplicatively.
A \emph{grading by} \(H\), or an \emph{\(H\)-grading}, on \(\mathcal A\) is a decomposition
\(\Gamma\cl \mathcal A=\bigoplus_{g\in H}\mathcal A_g\)
such that \(\mathcal A_g\mathcal A_h\subseteq\mathcal A_{gh}\) for all
\(g,h\in H\). The subspace \(\mathcal A_g\) is called the
\emph{homogeneous component} of degree \(g\). 
Two \(H\)-gradings,
\(\Gamma\cl\mathcal A=\bigoplus_{g\in H}\mathcal A_g\) and
\(\Gamma'\cl\mathcal A=\bigoplus_{g\in H}\mathcal A'_g\),
are called \emph{isomorphic} if there exists an algebra automorphism 
\(\psi\cl\mathcal A\to\mathcal A\) such that
\(\psi(\mathcal A_g)=\mathcal A'_g\) for every \(g\in H\).
We say that \(\Gamma'\cl\mathcal A=\bigoplus_{g'\in H'}\mathcal A'_{g'}\) is a \emph{coarsening} of \(\Gamma\) (or \(\Gamma\) is a \emph{refinement} of \(\Gamma'\)) if for any \(g\in H\) there exists \(g'\in H'\) such that \(\mathcal A_g\subseteq\mathcal A'_{g'}\).

In this paper, we consider gradings by  \(\mathbb Z_m:=\mathbb Z/m\mathbb Z\) on a classical simple Lie algebra \(\g\) over an algebraically closed field \(K\).

If \(K\) contains a primitive \(m\)-th root of unity \(\zeta\), then a
\(\mathbb Z_m\)-grading
\(\Gamma\cl\mathcal{A}=\bigoplus_{\bar k\in\mathbb Z_m}\mathcal A_{\bar k}\)
determines an automorphism
\(\theta_\Gamma\in\operatorname{Aut}\mathcal A\) by
\(\theta_\Gamma(x)=\zeta^k x\) for \(x\in\mathcal A_{\bar k}\).
Conversely, any automorphism \(\theta\) satisfying
\(\theta^m=\operatorname{id}\) is semisimple, and its eigenspace decomposition
defines a \(\mathbb Z_m\)-grading. Under this correspondence, isomorphism
classes of gradings correspond to conjugacy classes of such automorphisms.

Kac~\cite{Kac} classified finite-order automorphisms of a finite-dimensional simple Lie algebra \(\g\) over \(\bb{C}\), and
hence gradings by finite cyclic groups, in terms of affine Dynkin diagrams and
the coordinates that now bear his name (see also~\cite{vic}). An alternative proof, not using infinite-dimensional Lie algebras and following a method of Gantmakher~\cite{gan}, is presented by Onishchik and Vinberg~\cite[\S\S3.8--3.11]{Vin}. 
It shows that, up to conjugation, \(\theta\) is contained in a certain quasitorus \(S^{(\sigma)}\subset\Aut\g\), where \(\sigma\) is a diagram automorphism of \(\g\) of order \(q\mid m\), and recasts the classification as an orbit problem for a Weyl-type group, \(W^{(\sigma)}\), acting on \(S^{(\sigma)}\). 

In characteristic \(p>0\), the reformulation of gradings in terms of finite-order automorphisms extends only to \(\Z_m\)-gradings with \(p\nmid m\). If \(p\mid m\), however, finite-order automorphisms detect only the prime-to-\(p\) part of a \(\mathbb Z_m\)-grading.
Serre~\cite{ser} observed that, if we replace elements of order \(m\) in a simply connected simple algebraic group \(G\) by embeddings of the affine group scheme of $m$-th roots of unity, $\boldsymbol{\mu}_m$, into the 
affine group scheme $\mathbf G$ associated to \(G\), then the classification up to conjugation becomes characteristic-independent and can be expressed in terms of Kac coordinates. 
For an adjoint simple algebraic group \(G\) and $A=G \rtimes \Gamma$, where \(\Gamma\) is the group of diagram automorphisms of \(G\), it was asserted by Reeder et al.~\cite[\S1.3]{reeder-levy-yu} that the classification of embeddings \(\boldsymbol{\mu}_m \to \mathbf{A}\) in terms of Kac coordinates holds in characteristic \(p\) assuming that the projection of \(\boldsymbol{\mu}_m\) to \(\mathbf\Gamma\) has order not divisible by \(p\). A proof of this was sketched by Levy in~\cite[\S2]{levy}, but only in the case $p\nmid m$. 

The purpose of this paper is to classify, up to isomorphism, cyclic gradings on a classical simple Lie algebra \(\g\) in arbitrary 
characteristic. This is equivalent to the classification of morphisms \(\boldsymbol{\mu}_m \to \mathbf{A}\) where \(\mathbf{A}=\op{\mathbf{Aut}}\g\), the automorphism group scheme of \(\g\), up to conjugation by the automorphism group \(A=\mathbf{A}(K)\). This latter was computed by Steinberg~\cite{Rob} and is isomorphic to \(G\rtimes\Gamma\) for an adjoint simple algebraic group \(G\), which has the same type as \(\g\) except in some cases in characteristic \(2\) and \(3\), see~\Cref{tab:exceptional-aut-groups}.
We show that in all cases \(\mathbf{A}\)
% , the affine group scheme associated to \(A=\Aut\g\), is smooth,
is smooth, i.e., it is the affine group scheme associated to \(A=\Aut\g\), 
and consider, more generally, the affine group schemes associated to (possibly disconnected) reductive algebraic groups \(A\).
%such that the identity component \(G=A^\circ\) is semisimple and contains its centralizer in \(A\), which implies that the connected components of \(A\) are indexed by the elements of a subgroup \(\Gamma_0\) of the automorphism group \(\Gamma\) of the Dynkin diagram of \(G\).

We first use a version of Gantmakher's approach to show that the conjugacy problem for semisimple elements contained in a connected component \(A^{(\sigma)}\) of \(A\) is equivalent to the orbit problem for the action of a Weyl-type group, \(W^{(\sigma)}\), on  a certain quasitorus \(S^{(\sigma)}\) associated to any \(\sigma\) in the component group \(\Gamma_0\) of \(A\).
The same result for the case \(A=G\rtimes\Gamma_0\) with semisimple \(G\) was obtained by Mohrdieck \cite{HamburgFoldingThesis}, using a different method, under the assumptions that \(\operatorname{char}K\neq 2\), and that \(K\) is not an algebraic closure of a finite field. Then we show that, over any algebraically closed field \(K\), any morphism \(\boldsymbol{\mu}_m \to \mathbf{A}\) is \(G\)-conjugate to one whose image is contained in \(\mathbf{S}^{(\sigma)}\), for some \(\sigma\in\Gamma_0\) of order \(q\) dividing \(m\) and not divisible by \(\operatorname{char}K\), and reduce the conjugacy problem for such morphisms to the classification of \(W^{(\sigma)}\)-orbits in \(\Hom_\Z(\frk X(S^{(\sigma)}),\Z_m)\), where \(\frk{X}\) denotes the character group. In the case \(A=G\rtimes\Gamma_0\), we determine the structure of \(W^{(\sigma)}\) and its action on \(\frk X(S^{(\sigma)})\), which shows that they are independent of \(K\). Taking \(G\) simple of adjoint type and \(\Gamma_0=\Gamma\), this recovers the classification of \(\Z_m\)-gradings on classical simple Lie algebras in terms of Kac coordinates. 

The paper is organized as follows. Section~2 recalls the correspondence
between gradings by abelian groups and actions of their Cartier duals, reviews the construction of classical simple Lie algebras, and discusses the fixed-point subgroups of quasitori under an automorphism.
In~\Cref{sec:outer-components-general}, we study semisimple conjugacy classes in the
connected components \(A^{(\sigma)}\) of  \(A\) as above. We attach to each \(\sigma\in\Gamma_0\) of 
order not divisible by \(\operatorname{char}K\) a quasitorus \(S^{(\sigma)}\), its component \(M^{(\sigma)}\), and a finite group \(W^{(\sigma)}\), and prove that 
every semisimple element of \(\sigma G\) is \(G\)-conjugate to an element of \(M^{(\sigma)}\), with conjugacy inside \(M^{(\sigma)}\) governed by \(W^{(\sigma)}\) (see~\Cref{thm:sem-conj-outer-reorganized}).
Then, in~\Cref{sec:gen-root-decomp}, we consider the grading of the tangent Lie algebra \(\mathscr{L}(A)\) determined by the quasitorus \(S^{(\sigma)}\), which shares several properties with the root space decomposition (corresponding to \(\sigma=\mathrm{id}\)) and is called the \emph{\(\sigma\)-root decomposition} by Vinberg and Onishchik (see~\Cref{thm:sigma-root-properties-rewritten}).
We also construct a characteristic-independent  model for \(W^{(\sigma)}\) and its action  (see~\Cref{prop:full-lattice-sigma-weyl-action}) in the case \(A=G\rtimes\Gamma_0\). The scheme-theoretic reductions needed to handle the \(p\)-primary part of \(\boldsymbol{\mu}_m\) are established
in~\Cref{sec:lem-on-group-schemes} and then combined
in~\Cref{sec:mu-embeddings} with the results of~\Cref{sec:outer-components-general} to classify morphisms
\(\boldsymbol{\mu}_m\to\mathbf A\) up to \(G\)-conjugacy (see~\Cref{thm:block-reduction-Weyl-orbits,thm:final-classification-mu-embeddings}) and \(A\)-conjugacy (see~\Cref{prop:A-vs-G-fixed-outer-block}).~\Cref{sec:classical-simple-lie-algebras} is dedicated to classical simple Lie algebras. First we determine their automorphism group schemes (see \Cref{thm:aut-g-smooth,thm:aut-group-scheme-classical}), which 
allows us to apply the classification of morphisms \(\boldsymbol{\mu}_m\to\mathbf A\) in the case \(A=\Aut\g\) to classify \(\Z_m\)-gradings on a classical simple Lie algebra \(\g\) (see~\Cref{thm:cyclic-gradings-final-theorem}) and present them as coarsenings of the \(\sigma\)-root decompositions of \(\mathscr{L}(A)\) restricted to its ideal \(\g\).

Throughout the paper, the ground field \(K\) is assumed to be algebraically closed but of arbitrary characteristic, unless stated otherwise.

\section{Preliminaries}

\subsection{Gradings and actions}\label{sec:duality}
Given a \(G\)-grading \(\Gamma:\,\mathcal A=\bigoplus_{g\in G}\mathcal A_g\), every character \(\chi\in\widehat G:=\operatorname{Hom}(G,K^\times)\) acts on $\mathcal A\) by \(\chi\cdot a := \chi(g)a\) for  \(a\in\mathcal A_g\), extended \(K\)–linearly. This yields an automorphism of \(\mathcal A\) as a \emph{graded} algebra, since each homogeneous component $\mathcal A_g$ is preserved and $\chi$ restricts to the scalar operator $\chi(g)\operatorname{id}_{\mathcal A_g}$. Consequently, the grading \(\Gamma\) determines a homomorphism $\eta_\Gamma:\ \widehat G \rightarrow \Aut\mathcal A$ which sends $\chi \longmapsto [a\mapsto \chi\cdot a].$
When \(G\) is abelian and \(K\) is algebraically closed of characteristic $0\), characters separate points of \(G\), so \(\Gamma\) can be recovered as the simultaneous eigenspace decomposition for the commuting family $\eta_\Gamma(\widehat G)$:
\begin{equation}\label{homog-comp-deg-g}
    \mathcal A_g
=\bigl\{\,a\in\mathcal A\ \big|\ \eta_\Gamma(\chi)(a)=\chi(g)\,a\ \text{for all } \chi\in\widehat G\,\bigr\}.
\end{equation}
For example, for the $\mathbb Z^n$-grading on $M_n(K)$ defined by assigning the matrix unit \(E_{ij}\) degree \(\varepsilon_i-\varepsilon_j\), where \(\{\varepsilon_1,\ldots,\varepsilon_n\}\) is the standard basis of \(\Z^n\), the associated homomorphism is given by $(K^\times)^n \rightarrow \Aut\bigl(M_n(K)\bigr),
$ sending $(\lambda_1,\dots,\lambda_n)$ to the inner automorphism \(\operatorname{Int}\bigl(\operatorname{diag}(\lambda_1,\dots,\lambda_n)\bigr)\).

If \(\mathcal A\) is finite-dimensional over an algebraically closed field \(K\), then \(\Aut\mathcal A\) is an algebraic group (see e.g.~\cite[\S7.6]{Wat}). 
For \(G\) a finitely generated abelian group, the character group $\widehat G$ is a diagonalizable algebraic group ~\cite[\S16]{Hum}, isomorphic to a product of a torus with a finite abelian group of order not divisible by $\operatorname{char}K$. We will refer to algebraic groups of this form as \emph{quasitori}.
For any \(G\)-grading, $\eta_\Gamma$ is a morphism of algebraic groups. Conversely, let $\eta:\widehat G\to \Aut\mathcal A$ be a morphism of algebraic groups. Its image consists of commuting diagonalizable automorphisms, hence determines a simultaneous eigenspace decomposition of \(\mathcal A\) indexed by the group \(\frk X(\widehat{G})\) of algebraic group homomorphisms $\widehat G\to K^\times$. When $\operatorname{char}K=0\) (or $\operatorname{char}K=p$ and \(G\) has no $p$-torsion), \(\frk X(\widehat{G})\) is canonically isomorphic to \(G\); defining
$\mathcal A_g$ by Equation \eqref{homog-comp-deg-g} produces a \(G\)-grading \(\Gamma\) on \(\mathcal A\) with $\eta=\eta_\Gamma$.

If the base field \(K\) is not assumed to be algebraically closed, or if 
$\operatorname{char} K$ is positive, the preceding one-to-one correspondence is formulated in the language of affine group schemes over \(K\), which can be defined as representable functors from the category of commutative associative unital \(K\)-algebras to the category of groups (for general background, we refer the reader to \cite{Wat}). For each such algebra $\mathcal R$, the Cartier dual $G^D$ of an abelian group \(G\) is defined by $G^D(\mathcal R) := \Hom(G,\mathcal R^{\times})$, and the automorphism group scheme $\op{\mathbf{Aut}}\mathcal A$ 
 is defined by \(\op{\mathbf{Aut}}\mathcal A(\mathcal R):= \Aut_{\mathcal R}(\mathcal A \otimes_{K} \mathcal R)\).
Then a \(G\)-grading \(\Gamma\) on \(\mathcal A\) yields a  homomorphism of groups \(\left(\eta_{\Gamma}\right)_{\mathcal{R}}: \Hom(G,\mathcal R^{\times}) \rightarrow \operatorname{Aut}_{\mathcal{R}}(\mathcal{A} \otimes \mathcal{R})\) defined by
\[\left(\eta_{\Gamma}\right)_{\mathcal{R}}(f)(x \otimes r)=x \otimes f(g) r \text{ for all } x \in \mathcal{A}_g, g \in G, r \in \mathcal{R},f\in\Hom(G,\mathcal R^{\times}).\]
Thus, we obtain a morphism of affine group schemes \(\eta_{\Gamma}:G^{D}\rightarrow \operatorname{\mathbf{Aut}}\mathcal{A}\).

Conversely, a morphism $\eta: G^D \rightarrow \operatorname{\mathbf{Aut}}\mathcal{A}$ gives rise to a \(G\)-grading \(\Gamma\) on the algebra \(\mathcal{A}\) such that $\eta_{\Gamma}=\eta$.

The following proposition summarises the preceding correspondence (see e.g. \cite[\S 1.4]{Alb}.
\begin{proposition} Let \(G\) be an abelian group and \(\mathcal{A}\) a finite-dimensional algebra over an arbitrary field. Then \(G\)-gradings on \(\mathcal{A}\) are in one-to-one correspondence with morphisms of affine group schemes \(G^D \rightarrow \op{\mathbf{Aut}}\mathcal{A}\). Two \(G\)-gradings are isomorphic if and only if the corresponding morphisms are conjugate by an element of \(\Aut\mathcal{A}\). 
\end{proposition}

Since we assume that \(K\) is algebraically closed, we can identify a smooth algebraic group 
scheme with the algebraic group of its \(K\)-points. For a finite-dimensional algebra 
\(\mathcal A\), the automorphism group scheme \(\op{\mathbf{Aut}}\mathcal A\) 
is algebraic, while the Cartier dual \(G^D\) of an abelian group \(G\) is algebraic 
if and only if \(G\) is finitely generated. Smoothness is automatic in characteristic
\(0\), but may fail when \(\operatorname{char}K=p>0\). In particular, \(G^D\) is smooth if and only if \(G\) has no \(p\)-torsion. Moreover, \(\operatorname{\mathbf{Aut}}\mathcal A\) is smooth if and only if the Lie algebra of the algebraic group \(\operatorname{Aut}\mathcal A=\operatorname{\mathbf{Aut}}\mathcal A(K)\) 
coincides with \(\operatorname{Der}\mathcal A\), the Lie algebra of the group scheme \(\operatorname{\mathbf{Aut}}\mathcal A\); in general, the former is contained in the latter. Regardless of smoothness, \(G^D\) is
diagonalizable, and centralizers of diagonalizable subgroupschemes in smooth
group schemes are smooth; see, e.g.,~\cite[Exp.~XI,~Cor.~2.4]{sga3}, cf.~\cite[\S18.4]{Hum}.

\begin{remark}\label{rem:reduced-homomorphism-to-group-scheme}
If \(K\) is algebraically closed, then any affine group scheme \(\mathbf{A}\) has the largest smooth subgroupscheme, namely the subgroupscheme defined by the ideal \(I\) of the representing Hopf algebra \(K[\mathbf{A}]\) consisting of all nilpotent elements.  
If \(\mathbf A\) is algebraic, then this is the affine group scheme associated to the algebraic group \(A:=\mathbf{A}(K)\). If \(B\) is another affine algebraic group and \(\mathbf B\) is the associated smooth affine group scheme (represented by the coordinate algebra \(K[B]\)), then any morphism of affine algebraic groups \(f\cl B\to A\) induces a morphism of affine group schemes
\(\bar f \cl \mathbf{B}\to\mathbf{A}\) such that \(\bar f_K=f\). Indeed, precomposing the comorphism \(f^*\cl K[A]\to K[B]=K[\mathbf{B}]\) with the quotient map \(K[\mathbf{A}] \to K[A]=K[\mathbf{A}]/I\) gives a Hopf algebra homomorphism 
\(K[\mathbf{A}]\to K[\mathbf{B}]\), and hence a morphism \(\mathbf{B}\to\mathbf{A}\).
\end{remark}

We will need the following general fact: given a grading \(\Gamma\cl\mathcal A=\bigoplus_{g\in G}\mathcal A_g\) by a group \(G\) on an algebra \(\mathcal A\), any group homomorphism \(f\cl G\to G'\) induces a \(G'\)-grading \({}^f\Gamma\) on \(\mathcal A\) with homogeneous components \(\mathcal A'_{g'}:=\bigoplus_{g\in f^{-1}(g')}\mathcal A_g\) for any \(g'\in G'\). Clearly, \({}^f\Gamma\) is a coarsening of \(\Gamma\).

\begin{example}\label{ex:Cartan_grading}
For any finite-dimensional algebra $\mathcal{A}$,
a maximal torus $T$ of $\Aut\mathcal{A}$ gives a subgroupscheme $\mathbf{T}$ of $\mathbf{Aut}\mathcal{A}$. It also defines a grading on $\mathcal{A}$ by the free abelian group $\Lambda=\frk{X}(T)$. Since maximal tori are conjugate, this grading is unique up to isomorphism and called the \emph{Cartan grading} of $\mathcal{A}$. If $\Gamma$ is any grading on $\mathcal{A}$ by an abelian group $G$ such that the image of $\eta_\Gamma\cl G^D\to\mathbf{Aut}\mathcal{A}$ is contained in a torus (which is always the case if $G$ is torsion-free), then $\Gamma$ is isomorphic to a coarsening of the Cartan grading, induced by a group homomorphism $\Lambda\to G$. Two such coarsenings are isomorphic if and only if the homomorphisms $\Lambda\to G$ are conjugate by the action of the normalizer of $T$ in $\Aut\mathcal{A}$ (see \cite[Prop.~4.22]{Alb}).   
\end{example}

For instance, if $\g$ is a finite-dimensional simple Lie algebra over $\mathbb{C}$ with root system $\Phi$, then $\Z$-gradings on $\g$ are classified up to isomorphism by the orbits in $\Hom(\Z\Phi,\Z)$ under the action of the automorphism group of $\Phi$. If $\Pi$ is a base of $\Phi$, then each orbit has a representative that sends the elements of $\Pi$ to nonnegative integers, and two such assignments of integers yield isomorphic $\Z$-gradings if and only if one can move from one to the other by a symmetry of the Dynkin diagram. We will follow a similar approach for $\Z_m$-gradings on classical simple Lie algebras over any algebraically closed field, but we will have to replace the maximal torus of $\Aut\g$ by certain quasitori.

\subsection{Classical simple Lie algebras}\label{sec:classical}

We recall the construction of the classical simple Lie algebras over an
arbitrary field \(K\), fixing notations that will be used later. 

Let \(\Phi\) be an irreducible root system of rank \(\ell\), with base
\(\Pi=\{\alpha_1,\dots,\alpha_\ell\}\). Choose a Chevalley basis
\(\{h_1,\dots,h_\ell\}\cup\{x_\alpha\mid\alpha\in\Phi\}\) of the corresponding simple complex Lie algebra \(\g_\CC\). Let \(\mathfrak g_{\mathbb Z}\) be its \(\mathbb Z\)-span, and \(\mathfrak{h}_{\mathbb{Z}}\) be the \(\mathbb{Z}\)-span of the basis elements 
\(\{h_1, \ldots, h_{\ell}\}\). 
The resulting \(\mathbb Z\)-form is independent, up to isomorphism, of the chosen Chevalley basis; see~\cite[\S25.4]{James}. Set
\(\tilde{\g}_K:=\mathfrak g_{\mathbb Z}\otimes_{\mathbb Z}K\) and \(\tilde{\h}_K:=\h_{\Z}\otimes_{\Z} K\). When no confusion can arise, we omit the subscript \(K\).

For \(\alpha\in\Phi\), let
\(\bar\alpha\in\tilde{\h}^{\,*}\) be defined by
\(\bar\alpha(h_i)=\langle\alpha,\alpha_i^\vee\rangle\), where the Cartan
integer is regarded as an element of \(K\). 

The algebra \(\tilde{\mathfrak{g}}_{K}\) is not necessarily simple. Its simplicity depends on the characteristic of the field \(K\). Specifically, the quotient of
\(\tilde{\g}\) by its center is simple under the following conditions:
\(\operatorname{char} K \neq 2\) if \(\Phi\) contains roots of unequal lengths or is of type \(A_1\), and \(\operatorname{char} K \neq 3\) if \(\Phi\) is of type \(G_2\) (see~\cite[\S2.6]{Rob}). Under these assumptions, \(Z(\tilde{\g})=\{h\in\tilde{\h}\mid
\bar\alpha(h)=0\text{ for all }\alpha\in\Phi\}\), and we set  
\(\mathfrak g:=\tilde{\g}/Z(\tilde{\g})\) and \(\mathfrak h:=\tilde{\h}/Z(\tilde{\g})\), and continue to denote by the same symbols the images of \(h_i\) and \(x_\alpha\) as well as the functionals on \(\h\) induced by \(\bar\alpha\).

The Chevalley basis gives the usual root space decomposition of \(\g\), which is a grading by the root lattice \(\Z\Phi\):
\[
\mathfrak g=\mathfrak h\oplus\bigoplus_{\alpha\in\Phi}\g_\alpha
\] 
where \(\g_\alpha=Kx_\alpha\), \(\deg x_\alpha=\alpha\) and \(\mathfrak h\) has degree \(0\). 
In positive characteristic, distinct roots may induce the same functional on \(\mathfrak h\), so the eigenspace decomposition relative to \(\mathfrak{h}\) may be distinct from the above (a proper coarsening):
\begin{equation*}\label{4}
\mathfrak{g} = \mathfrak{h} \oplus \bigoplus_{\bar{\alpha} \in \bar{\Phi}} \mathfrak{g}_{\bar{\alpha}} , \quad \text{where} \quad \mathfrak{g}_{\bar{\alpha}} = \{ x \in \mathfrak{g} \mid [h, x] = \bar{\alpha}(h)x \text{ for all } h \in \mathfrak{h} \}.
\end{equation*}

%In characteristic \(0\), one has \(Z(\tilde{\mathfrak g})=0\), and this decomposition coincides with the root-lattice grading. 
%If \(\operatorname{char}K=p>0\), then \(\langle\bar{\Phi}\rangle\subseteq\mathfrak h^*\) is an elementary abelian \(p\)-group of rank at most \(\ell\), and the eigenspace decomposition may be a proper coarsening of the root-lattice grading.

The simple Lie algebras constructed in this way are called \emph{classical simple Lie algebras}. Under the stated restrictions on the
characteristic, their type is well defined: two such algebras are isomorphic if and only if their underlying irreducible root systems have the same type (see~\cite[\S8.1]{Rob}).

\subsection{Fixed point subgroups and their duals}

We next record a few consequences of the duality between quasitori over an algebraically closed field \(K\) and their character groups. For a quasitorus \(Q\) with character group \(X=\frk X(Q)\), each subgroup \(M\leq X\) defines the closed subgroup of \(Q\) on which all characters in \(M\) are trivial: 
\[
M^\perp:=\{q\in Q\mid \chi(q)=1\text{ for all }\chi\in M\}.
\] 
%equivalently \(M^\perp=\bigcap_{\chi\in M}\Ker(\chi)\subseteq Q\). 
If \(M\subseteq N\), then \(N^\perp\subseteq M^\perp\). Since \(K\) is algebraically closed, \(K^\times\) is a divisible abelian group, hence an injective \(\mathbb Z\)-module. Therefore, the short exact sequence of abelian groups $0\to M \xrightarrow{\iota} X \to X/M \to 0,$ where $\iota$ is the inclusion map, yields
\[0\to \Hom(X/M,K^\times) \to \Hom(X,K^\times) \xrightarrow{\iota^*} \Hom(M,K^\times)\to 0,\]
where $\iota^*$ is the restriction map. 
Since \(Q\) is a quasitorus, evaluation gives a canonical isomorphism \(Q\cong\Hom(X,K^\times)\), hence we obtain 
\begin{equation}\label{eq:A-perp-isom}
    M^\perp \cong \Ker\iota^* \cong \Hom(X/M,K^\times) \quad \text{and} \quad Q/M^\perp \cong \Hom(M,K^\times).
\end{equation} 
Define the \emph{radical} of \(M\) in \(X\) by
\(\sqrt M:=\{\chi\in X\mid n\chi\in M\ \text{for some}\ n\geq 1\}\).

\begin{proposition}\label{prop:perp-formalism}
Let $M,N$ be subgroups of $X.$ Then:
\begin{enumerate}
\item[(1)]\label{prop:perp1}
\((\sqrt{M})^\perp = (M^\perp)^\circ.\)
\item[(2)]\label{prop:perp2}
\((M+N)^\perp = M^\perp\cap N^\perp.\)
\end{enumerate}
\end{proposition}

\begin{proof} 
(1) Set \(L:=X/M\). Then the torsion subgroup of \(L\) is \(\mathrm t(L)=\sqrt M/M\), and restriction from \(L\) to \(\mathrm t(L)\) gives the exact sequence 
\[ 1\longrightarrow(\sqrt M)^\perp \longrightarrow M^\perp \longrightarrow\Hom_{\mathbb Z}(\mathrm t(L),K^\times) \longrightarrow1,\]
where we have used Equation \eqref{eq:A-perp-isom} to identify \(M^\perp \cong \Hom(X/M,K^\times)\) and \((\sqrt{M})^\perp \cong \Hom(X/\sqrt{M},K^\times)\). The group \(\Hom(\mathrm t(L),K^\times)\) is finite, hence \((M^\perp)^\circ\subseteq(\sqrt M)^\perp\). 
Since \(X/\sqrt M\cong L/\mathrm t(L)\) is free, \((\sqrt M)^\perp\) is a torus, and hence \((\sqrt M)^\perp\subseteq M^\perp\) implies \((\sqrt M)^\perp \subseteq (M^\perp)^\circ\). 

(2) An element \(q\in Q\) belongs to \((M+N)^\perp\) if and only if every character in both \(M\) and \(N\) is trivial on \(q\). Hence \((M+N)^\perp=M^\perp\cap N^\perp\).
\end{proof}

Let \(\sigma\) be an automorphism of $Q$. We write the action of \(\sigma\) on \(Q\) on the
right, so that \(q^\sigma\) denotes the image of \(q\in Q\) under \(\sigma\).
The induced action on the character group \(X\) is written on the left and is defined by
\[(\sigma\chi)(q):=\chi(q^\sigma),
\qquad q\in Q,\ \chi\in X.\]

\begin{proposition}\label{prop:fixed-points-quotient-quasitorus}
Let \(\sigma\) be an automorphism of a quasitorus \(Q\), and set \(X=\frk X(Q)\) and
\(X_\sigma=(\sigma-1)X\). Then \begin{enumerate}
    \item[(1)] \(Q^\sigma\cong\Hom_{\mathbb Z}(X/X_\sigma,K^\times)\), and 
    \item[(2)] \(Q/Q^\sigma\cong\Hom_{\mathbb Z}(X_\sigma,K^\times)\).
\end{enumerate}
\end{proposition}

\begin{proof} 
Since \(Q\cong\Hom(X,K^\times)\), an element \(q\in Q\) is fixed by \(\sigma\) if and only if \(\chi(q^\sigma)=\chi(q)\) for every \(\chi\in  X\). 
By the definition of the induced action on \(X\), this is equivalent to \((\sigma\chi-\chi)(q)=1\) for every \(\chi\in X\), or equivalently \((X_\sigma)^\perp=Q^\sigma\). 
The result now follows from Equation \eqref{eq:A-perp-isom}.
\end{proof}

\begin{corollary}\label{cor:kernel-translation-lattice}
The restriction map \(X \to \frk X((Q^\sigma)^\circ)\) is surjective, has kernel \(\sqrt{X_\sigma}\), and induces an isomorphism \(\frk X\bigl((Q^\sigma)^\circ\bigr)
 \cong X/\sqrt{X_\sigma} \cong M_\sigma/\mathrm t(M_\sigma)\), where \(M_\sigma:=X/X_\sigma\). 
\end{corollary}

\begin{proof}
The surjectivity is a general fact: for any closed subgroup $R$ of $Q$, restriction is a surjection of the coordinate algebras $K[Q]\to K[R]$, which in this case are the group algebras of $X$ and $\frk{X}(R)$.
Since \(Q^\sigma= (X_\sigma)^\perp\), we have  \((Q^\sigma)^\circ=(\sqrt{X_\sigma})^\perp\) by~\Cref{prop:perp-formalism}(1). By Equation~\eqref{eq:A-perp-isom}, we get \((Q^\sigma)^\circ\cong \Hom(X/\sqrt{X_\sigma},K^\times)\).
Since \(X/\sqrt{X_\sigma}\cong M_\sigma/\mathrm t(M_\sigma)\) is a free abelian group, its elements are separated by homomorphisms to \(K^\times\), which means that only the elements of 
\(\sqrt{X_\sigma}\) restrict have trivial restriction to \((Q^\sigma)^\circ\).
% it is a projective \(\mathbb Z\)-module. Hence the exact sequence 
% \(0 \to \sqrt{X_\sigma}\to X \to M_\sigma/\mathrm{t}(M_\sigma) \to 0\) splits.
% Choose a complement \(X'\) of \(\sqrt{X_\sigma}\) in \(X\). Under the identification 
% \(Q\cong\Hom_(X,K^\times)\), the decomposition \(X=\sqrt{X_\sigma}\oplus X'\) yields
% \(Q=(X')^\perp\times(\sqrt{X_\sigma})^\perp
% =(X')^\perp\times(Q^\sigma)^\circ\).
% It follows that every character of \((Q^\sigma)^\circ\) extends to a
% character of \(Q\), by declaring it to be trivial on the factor
% \((X')^\perp\). Therefore the restriction map
% \(\varphi \colon X\rightarrow
% \frk X\bigl((Q^\sigma)^\circ\bigr)\)
% is surjective.
% Clearly, \(\sqrt{X_\sigma}\subseteq\Ker \varphi\), since
% \((Q^\sigma)^\circ=(\sqrt{X_\sigma})^\perp\). On the other hand, if \(\chi'\in X'\) is in 
% \(\Ker\varphi\), then it will be trivial on \((X')^\perp\times(Q^\sigma)^\circ=Q\).  Therefore
% \(\Ker \varphi=\sqrt{X_\sigma}\).
\end{proof}

\section{Quasitori attached to connected components}\label{sec:outer-components-general}

\subsection{The general setup}\label{subsec:the-general-setup}

%This subsection establishes some general facts about semisimple conjugacy classes 
%in a (possibly disconnected) reductive algebraic group \(A\). 
%under the assumptions that \(G:=A^\circ\) is semisimple and that \(C_A(G)\subseteq G\). 
% After identifying the connected components of \(A\) with their induced outer automorphisms of
% \(G\), we associate with each component containing semisimple elements a
% quasitorus, a distinguished coset, and a generalized Weyl group. These objects
% reduce the classification of semisimple \(G\)-conjugacy classes in the component to a finite-group orbit problem.
%
% Fix a maximal torus \(T\subset G\). Let \(\Pi\) be a base of the root system \(\Phi(G,T)\), and let \(\Gamma:=\Aut\Pi\) be the group of diagram automorphisms. 
% Since \(G=A^\circ\), conjugation induces a morphism \(\rho\cl A\longrightarrow\Aut G\),
% \(a\longmapsto\Int(a)|_G\), and hence a homomorphism 
% \(\bar\rho\cl A/G\longrightarrow\Out G\).
% This homomorphism is injective. Indeed, if
% \(\Int(a)|_G=\Int(g)\) for some \(g\in G\), then
% \(ag^{-1}\in C_A(G)\subseteq G\), and therefore \(a\in G\).
% Via the injective map \(\bar \rho\), we identify \(A/G\) with \(\Gamma_0:=\op{Im}(\bar\rho)\subseteq \op{Out}G \subseteq \Aut \Pi\) (see e.g.~\cite[\S27.4]{Hum}) and write 
% \(\pi\cl A\to \Gamma_0\) for the resulting quotient map. 
Let \(A\) be a (possibly disconnected) reductive algebraic group. Denote \(G=A^\circ\) and \(\Gamma_0=A/G\), so we have a short exact sequence of algebraic groups

\begin{equation*}\label{eq:A-extension}
1\longrightarrow G\longrightarrow A\xrightarrow{\;\pi\;}\Gamma_0\longrightarrow 1,
\end{equation*}
with \(G\) connected and \(\Gamma_0\) finite.
For \(\sigma\in\Gamma_0\), set \(A^{(\sigma)}:=\pi^{-1}(\sigma)\).
Thus \(A^{(\sigma)}\) is the connected component corresponding to \(\sigma\). 
%and \(A^{(\mathrm{id})}=G\).

We want to describe \(G\)-conjugacy classes of semisimple elements in \(A^{(\sigma)}\).
Note that \(A^{(\sigma)}\) contains semisimple elements if and only if the order of \(\sigma\) is not divisible by the characteristic of \(K\). 
Indeed, if \(s\in A^{(\sigma)}\) is semisimple, then so is \(\sigma=\pi(s)\), hence its order cannot be divisible by \(\operatorname{char}K\). Conversely, if \(\operatorname{char}K\nmid\operatorname{ord}(\sigma)\) and \(a=a_sa_u\) is the Jordan decomposition of an element \(a\in A^{(\sigma)}\), then \(\sigma=\pi(a_s)\pi(a_u)\) is the Jordan decomposition of \(\sigma\). Since \(\sigma\) is semisimple, we get \(\pi(a_u)=1\) and \(\pi(a_s)=\sigma\), so \(a_s\in A^{(\sigma)}\).

% Let \(\g:=\mathscr L(G)=\mathscr L(A)\).
% Conjugation by \(A\) on \(G\) defines a morphism
% \(\rho\colon A\rightarrow\operatorname{Aut}G\) that sends \(a\mapsto\Int(a)|_G\).
% For \(a\in A\), let \(\Ad(a):=d\bigl(\rho(a)\bigr)\in\operatorname{Aut}\g\), and write \(\g^a:=\{x\in\g\mid\Ad(a)(x)=x\}\).
Let $\g:=\mathscr L(G) = \mathscr{L}(A)$.  Each $a\in A$ acts on \(\g\) by the Lie algebra automorphism \(\mathrm{Ad}(a)\);
we write \(\g^a:=\{x\in\g\mid \mathrm{Ad}(a)(x)=x\}\). For \(\sigma\in\Gamma_0\), we define
\begin{equation}\label{df:l(sigma)}
\ell(\sigma)\;:=\;\min\bigl\{\dim\g^a\bigm|a\in A^{(\sigma)}\bigr\},
\end{equation}
and call an element  \(a\in A^{(\sigma)}\) \emph{regular} if \(\dim\g^a=\ell(\sigma)\).

Fix a semisimple element \(s\in A^{(\sigma)}\). Let \(T^{(\sigma)}\) be a maximal torus of the fixed-point subgroup \(G^s\) and consider the Zariski closure of the subgroup \(S^{(\sigma)}\) of \(A\) generated by \(T^{(\sigma)}\) and \(s\). Since \(s\) is semisimple, the Zariski closure \(\overline{\langle s\rangle}\) is a quasitorus. 
Since \(T^{(\sigma)}\) centralizes \(s\), the group \(S^{(\sigma)}\) is a homomorphic image of \(T^{(\sigma)}\times \overline{\langle s\rangle}\), hence also a quasitorus. 
Since \(T^{(\sigma)}\leq(S^{(\sigma)})^\circ\leq G^s\), the maximality of \(T^{(\sigma)}\) gives
\((S^{(\sigma)})^\circ=T^{(\sigma)}\). In particular, \(T^{(\sigma)}\) has finite index in \(S^{(\sigma)}\), so the subgroup \(\langle T^{(\sigma)},s\rangle\) is already closed, and hence equals \(S^{(\sigma)}\).

Consider the coset \(M^{(\sigma)}:=sT^{(\sigma)}\) and its normalizer \(N^{(\sigma)}:=N_G\bigl(M^{(\sigma)}\bigr)\). Since
\(s\in M^{(\sigma)}\) and \(s^{-1}M^{(\sigma)}=T^{(\sigma)}\), the coset \(M^{(\sigma)}\) generates
\(S^{(\sigma)}\). The group \(N^{(\sigma)}\) acts on \(M^{(\sigma)}\) by conjugation, and this action factors through the quotient group \(W^{(\sigma)}:=N^{(\sigma)}/T^{(\sigma)}\), which we call the
\emph{\(\sigma\)-Weyl group}.
The construction depends on the choice of the element \(s\) and maximal torus \(T^{(\sigma)}\),
but we will prove in \Cref{prop:independence-of-M}
%{thm:sem-conj-outer-reorganized} 
that the \(G\)-conjugacy classes of
\(T^{(\sigma)}\), \(S^{(\sigma)}\), and \(M^{(\sigma)}\) depend only on the component \(A^{(\sigma)}\).

\begin{lemma}\label{lem:zero-weight-space-outer}
Consider the $S^{(\sigma)}$-weight decomposition of \(\g\) 
\[\g=\g_0\oplus \bigoplus_{0\neq \chi\in  \frk X(S^{(\sigma)})}\g_\chi, \qquad
\g_\chi=\{x\in\g\mid (\operatorname{Ad} t)(x)=\chi(t)x\ \forall\,t\in S^{(\sigma)}\}.\]
Then \(\g_0=\mathscr{L}\bigl(T^{(\sigma)}\bigr).\)
\end{lemma}

\begin{proof}
By definition,
\(\g_0\) is the centralizer \(\mathfrak c_{\g}\bigl(S^{(\sigma)}\bigr)\) of \(S^{(\sigma)}\) in \(\g\). 
Since $S^{(\sigma)}$ is a quasitorus, we have, by~\cite[\S18.4, Prop.~A]{Hum}, that \(\mathfrak c_{\g}\!\bigl(S^{(\sigma)}\bigr)=\mathscr{L}\!\bigl(C_G(S^{(\sigma)})\bigr)\). As
\(S^{(\sigma)}=\langle T^{(\sigma)},s\rangle\), we have
\(C_G\bigl(S^{(\sigma)}\bigr)
=C_{G^s}\bigl(T^{(\sigma)}\bigr)\).
Since \(s\) is semisimple, \((G^s)^\circ\) is reductive (see,
e.g.,~\cite[\S3.6.2]{mcninch}). Moreover, \(T^{(\sigma)}\) is a maximal torus
of \((G^s)^\circ\), and therefore \(C_{(G^s)^\circ}\bigl(T^{(\sigma)}\bigr)=T^{(\sigma)}\)
(see, e.g.,~\cite[\S26.2, Cor.~A]{Hum}). Hence  
\(\g_0=\mathscr L\bigl(C_{G^s}(T^{(\sigma)})\bigr)=\mathscr L\bigl(C_{(G^s)^\circ}(T^{(\sigma)})\bigr)
=\mathscr L\bigl(T^{(\sigma)}\bigr).\)
\end{proof}

We next isolate a nonempty open subset of \(M^{(\sigma)}\) for elements of which the fixed-point algebra is exactly this zero-weight space.

\begin{lemma}\label{lem:regular-open-slice}
With notation as above, let
\(\Phi^{(\sigma)}:=\{\chi\in\frk X(S^{(\sigma)})\setminus\{0\}\mid \g_\chi\neq0\}\), and define
\(M_0^{(\sigma)}:=\{x\in M^{(\sigma)}\mid \chi(x)\neq1\text{ for every }\chi\in\Phi^{(\sigma)}\}\).
Then \(M_0^{(\sigma)}\) is a nonempty open subset of \(M^{(\sigma)}\). Moreover, \(x\in M_0^{(\sigma)}\) if and only if \(\g^x=\g_0=\mathscr L(T^{(\sigma)})\). Equivalently, \(x\in M_0^{(\sigma)}\) if and only if \(\dim\g^x=\dim T^{(\sigma)}\).
\end{lemma}
\begin{proof}
Write \(\Phi^{(\sigma)}=\{\chi_1,\dots,\chi_m\}\) and set \(T_i:=\{t\in T^{(\sigma)}\mid \chi_i(st)\neq 1\}\). 
Then \(M_0^{(\sigma)}=sT_0\), where \(T_0:=\bigcap_iT_i\). We claim that each \(T_i\) is a nonempty open subset of \(T^{(\sigma)}\). If \(\chi_i|_{T^{(\sigma)}}\) is nontrivial, then the function \(t\mapsto\chi_i(st)=\chi_i(s)\chi_i(t)\) is a nonconstant regular function on \(T^{(\sigma)}\), so \(T_i\) is the complement of a proper closed subset and is therefore nonempty open. If \(\chi_i|_{T^{(\sigma)}}=1\), then \(\chi_i(st)=\chi_i(s)\) is constant. In this case \(\chi_i(s)\neq1\), since otherwise \(\chi_i\) would be trivial on \(S^{(\sigma)}=\langle T^{(\sigma)},s\rangle\). 

Since \(T^{(\sigma)}\) is irreducible, the finite intersection \(T_0=\bigcap_iT_i\) is nonempty open, hence so is \(M_0^{(\sigma)}=sT_0\). Now let \(x\in M^{(\sigma)}\). Since \(x\in S^{(\sigma)}\), the \(S^{(\sigma)}\)-weight decomposition gives 
\[
\g^x=\g_0\oplus\bigoplus_{\chi\in\Phi^{(\sigma)},\,\chi(x)=1}\g_\chi.
\] 
By~\Cref{lem:zero-weight-space-outer}, \(\g_0=\mathscr L(T^{(\sigma)})\). Therefore \(\g^x=\g_0\) if and only if no nonzero \(S^{(\sigma)}\)-weight occurring in the decomposition of \(\g\) takes the value \(1\) on \(x\), which is exactly the condition \(x\in M_0^{(\sigma)}\). This is also equivalent to \(\dim\g^x=\dim\g_0=\dim T^{(\sigma)}\).
\end{proof}

We next show that the
\(G\)-conjugates of \(M_0^{(\sigma)}\) form a nonempty open subset of \(A^{(\sigma)}\). Recall that $\mathscr{T}(X)_x$ denotes the Zariski tangent space of an algebraic variety $X$ at a point $x$.
\begin{lemma}\label{lem:orbit-map-open}
Consider the morphism
\(\varphi \colon G\times M_0^{(\sigma)}\rightarrow A^{(\sigma)},\ (g,r)\mapsto grg^{-1}\).
Then $\varphi$ is open. In particular, its image is a nonempty open subset of $A^{(\sigma)}$.
\end{lemma}
\begin{proof}
By~\Cref{lem:regular-open-slice} $M_0^{(\sigma)}\neq\varnothing$, and hence the image is nonempty.
The varieties $G\times M_0^{(\sigma)}$ and $A^{(\sigma)}$ are smooth, so if $\varphi$ has surjective
differential everywhere then it follows from~\cite[III,~\S10, Prop.~10.4]{Har} that $\varphi$ is smooth everywhere. Since smooth morphisms are open (see e.g.~\cite[III,~\S9,~Exer.~9.1]{Har}), we have that $\varphi(G\times M_0^{(\sigma)})$ is open in $A^{(\sigma)}$. We now prove that the differential of $\varphi$ is surjective everywhere. 

For $(g,r)\in G\times M_0^{(\sigma)}$, the differential of $\varphi$ is a \(K\)-linear map
\[d\varphi_{(g,r)}\cl \mathscr{T}(G\times M_0^{(\sigma)})_{(g,r)}\;=\;\mathscr{T}(G)_g\oplus \mathscr{T}(M_0^{(\sigma)})_r \longrightarrow \mathscr{T}(A^{(\sigma)})_{grg^{-1}}.\]
Fix $g\in G$. Then left translation $L_g\cl G\to G$ and conjugation $\Int(g)\cl A^{(\sigma)}\to A^{(\sigma)}$ are isomorphisms of varieties, and \(\varphi\circ (L_g\times \mathrm{id})=\Int(g)\circ \varphi\).
Differentiating at $(1,r)$ shows that \(d\varphi_{(g,r)}\) is obtained from \(d\varphi_{(1,r)}$ by composing with the differentials of these isomorphisms. In particular, \(d\varphi_{(g,r)}\) is surjective if and only if \(d\varphi_{(1,r)}\) is surjective. Hence it suffices to prove surjectivity at  $(1,r)$ with $r\in M_0^{(\sigma)}$.

For $x,r\in A^{(\sigma)}$ we have $xr^{-1}\in G$, so the right translation 
\(R_{r^{-1}}\cl \ A^{(\sigma)}\longrightarrow G,\ x\longmapsto xr^{-1}\) is an isomorphism of varieties.
Since $R_{r^{-1}}(r)=1$, we have 
\(d(R_{r^{-1}})_r\cl \mathscr{T}(A^{(\sigma)})_r\xrightarrow{\sim}\mathscr{T}(G)_1=\g\). 
Moreover, since $r\in M^{(\sigma)}=sT^{(\sigma)}$ and $T^{(\sigma)}\subset G^s$, writing $r=sh$ with $h\in T^{(\sigma)}$ we get
\(M^{(\sigma)}r^{-1} = sT^{(\sigma)}h^{-1}s^{-1}=T^{(\sigma)}\).
Hence $R_{r^{-1}}$ restricts to an isomorphism $M^{(\sigma)}\to T^{(\sigma)}$; since
$M_0^{(\sigma)}$ is open in $M^{(\sigma)}$, it follows that
\(d(R_{r^{-1}})_r\cl \mathscr{T}(M_0^{(\sigma)})_r
\xrightarrow{\sim}
\mathscr{T}(T^{(\sigma)})_1
=\mathscr{L}\bigl(T^{(\sigma)}\bigr)\).

We now consider the restrictions of $\varphi$ to $G \times \{r\}$ and $\{1\} \times M^{(\sigma)}_{0}$. For the first one, we will use the commutator map
\(\gamma_r\cl G\to G, \gamma_r(g):=grg^{-1}r^{-1}.\)
Then $\varphi(g,r)=grg^{-1}=\gamma_r(g)r$, i.e.,
\(\varphi(\,\cdot\,,r)=R_r\circ \gamma_r\).
Since \(d(R_{r^{-1}})_r\cl \mathscr T(A^{(\sigma)})_r\to\g\) is an isomorphism, surjectivity of \(d\varphi_{(1,r)}\) is equivalent to surjectivity of
\[
d(R_{r^{-1}})_r\circ d\varphi_{(1,r)}\cl  \g\oplus \mathscr{T}(M_0^{(\sigma)})_r \longrightarrow \g.
\]
Using $\varphi(\,\cdot\,,r)=R_r\circ\gamma_r$, we have
\(d(R_{r^{-1}})_r\circ d\bigl(\varphi(\,\cdot\,,r)\bigr)_1 = d(\gamma_r)_1\).
It is well known (see e.g.~\cite[Prop.~10.1(c)]{Hum}) that
\(d(\gamma_r)_1 = 1-\operatorname{Ad}(r)\) on \(\g\). On the other hand, the restriction of $\varphi$ to $\{1\}\times M_0^{(\sigma)}$ is the inclusion
$M_0^{(\sigma)}\hookrightarrow A^{(\sigma)}$, hence its differential at $r$ is the natural inclusion
\(\mathscr{T}(M_0^{(\sigma)})_r \hookrightarrow \mathscr{T}(A^{(\sigma)})_r\).
After composing with \(d(R_{r^{-1}})_r\), this becomes the inclusion
\(\mathscr{L}\bigl(T^{(\sigma)}\bigr)\ \hookrightarrow\ \g\). Combining the differentials of the two restrictions, we obtain the formula
\[\bigl(d(R_{r^{-1}})_r\circ d\varphi_{(1,r)}\bigr)(\xi,\eta)
\;=\;(1-\operatorname{Ad}(r))\xi+\eta,\qquad
\xi\in\g,\ \eta\in \mathscr{L}\bigl(T^{(\sigma)}\bigr).\]
In particular, \(\op{Im}\bigl(d(R_{r^{-1}})_r\circ d\varphi_{(1,r)}\bigr)
=(1-\operatorname{Ad}(r))\g+\mathscr{L}\bigl(T^{(\sigma)}\bigr)\).
By definition $\Ker(1-\operatorname{Ad}(r))=\g^r$, so \Cref{lem:regular-open-slice} gives 
\(\Ker(1-\operatorname{Ad}(r))=\mathscr{L}\bigl(T^{(\sigma)}\bigr)\) for $r\in M_0^{(\sigma)}$. Since \(r \in A\) is semisimple, \(\Ad(r)\) is a semisimple endomorphism of
\(\g\), so
\(\g=\Ker(1-\Ad(r))\oplus \op{Im}(1-\Ad(r))=\mathscr{L}\bigl(T^{(\sigma)}\bigr)+(1-\operatorname{Ad}(r))\g\), so \(d(R_{r^{-1}})_r\circ d\varphi_{(1,r)}\) is surjective, as required. 
\end{proof}

\begin{proposition}\label{prop:independence-of-M}
Let $\tilde s\in A^{(\sigma)}$ be semisimple, and let $\widetilde T^{(\sigma)}$,
$\widetilde S^{(\sigma)}$ and $\widetilde M^{(\sigma)}$ be constructed from $\tilde s$ instead of $s$. 
Then there exists $g\in G$ such that
\[g\,\widetilde M^{(\sigma)}\,g^{-1}=M^{(\sigma)}, \quad g\,\widetilde S^{(\sigma)}\,g^{-1}=S^{(\sigma)},
\quad\text{and}\quad
g\,\widetilde T^{(\sigma)}\,g^{-1}=T^{(\sigma)},
\]
i.e., $M^{(\sigma)}$, $T^{(\sigma)}$ and $S^{(\sigma)}$ depend, up to \(G\)-conjugacy,
only on $\sigma$. Moreover, every semisimple element of $A^{(\sigma)}$ is \(G\)-conjugate to an
element of $M^{(\sigma)}$.
\end{proposition}
\begin{proof}
Applying~\Cref{lem:orbit-map-open} to the semisimple elements \(s\) and \(\tilde s\), we
obtain two nonempty open subsets of \(A^{(\sigma)}\), namely the images of
\[G\times M_0^{(\sigma)}\longrightarrow A^{(\sigma)}
\qquad\text{and}\qquad
G\times \widetilde M_0^{(\sigma)}\longrightarrow A^{(\sigma)}.\]
Hence these images intersect. Choose a point \(x\) in the intersection. Then there
exist elements \(a,b\in G\), \(s_0\in M_0^{(\sigma)}\), and
\(\tilde s_0\in \widetilde M_0^{(\sigma)}\) such that
\(x=as_0a^{-1}=b\tilde s_0 b^{-1}\). Set \(g:=a^{-1}b\), then
\(s_0=g\tilde s_0 g^{-1}\in M_0^{(\sigma)}\cap g\widetilde M_0^{(\sigma)}g^{-1}\). 

We claim that
\(T^{(\sigma)}=(G^{s_0})^\circ\).
Indeed, writing \(s_0=st_0\) with \(t_0\in T^{(\sigma)}\), we see that
\(T^{(\sigma)}\) centralizes \(s_0\) and hence \(T^{(\sigma)}\subseteq(G^{s_0})^\circ\).
Moreover, \(s_0\) is semisimple, so \(\mathscr L((G^{s_0})^\circ)=\g^{s_0}\). 
On the other hand, by \Cref{lem:regular-open-slice}, \(\g^{s_0}=\mathscr L(T^{(\sigma)})\).
Thus \(T^{(\sigma)}\) and \((G^{s_0})^\circ\) have the same dimension, which forces 
\(T^{(\sigma)}=(G^{s_0})^\circ\).
Applying the same argument to
\(s_0\in g\widetilde M_0^{(\sigma)}g^{-1}\) gives
\(g\widetilde T^{(\sigma)}g^{-1}=(G^{s_0})^\circ\).
Therefore,
\(g\widetilde T^{(\sigma)}g^{-1}=T^{(\sigma)}\).

It follows that \(M^{(\sigma)}\) and \(g\widetilde M^{(\sigma)}g^{-1}\) are cosets of
the same torus \(T^{(\sigma)}\), and both contain the element \(s_0\). Hence
\(g\widetilde M^{(\sigma)}g^{-1}=s_0T^{(\sigma)}=M^{(\sigma)}\). 
Since \(M^{(\sigma)}\) generates \(S^{(\sigma)}\) and \(\widetilde M^{(\sigma)}\) generates 
\(\widetilde S^{(\sigma)}\), the first assertion is proved.
Finally, if \(\tilde s\in A^{(\sigma)}\) is semisimple, then by construction
\(\tilde s\in \widetilde M^{(\sigma)}\), and the conjugacy just established shows that
\(\tilde s\) is \(G\)-conjugate to an element of \(M^{(\sigma)}\).
\end{proof}

Therefore \(M^{(\sigma)}\) (hence also \(T^{(\sigma)}\) and \(S^{(\sigma)}\)) 
are defined up to \(G\)-conjugacy.
Moreover, \Cref{prop:independence-of-M} reduces the classification of
semisimple elements in \(A^{(\sigma)}\) up to \(G\)-conjugacy to the
classification of elements of \(M^{(\sigma)}\) up to \(G\)-conjugacy.

\begin{theorem}\label{thm:sem-conj-outer-reorganized}
Let \(A\) be a reductive algebraic group with identity component \(G\) and component group \(\Gamma_0\).
%such that \(G:=A^\circ\) is semisimple and \(C_A(G)\subseteq G\). 
%Let \(\Gamma_0\) be the image of the monomorphism \(A/G\to\Out G\) induced by the action of \(A\) on \(G\). 
Denote by \(\pi\cl A\to\Gamma_0\) the quotient map and set \(A^{(\sigma)}:=\pi^{-1}(\sigma)\) for \(\sigma\in\Gamma_0\). If \(\op{char}K=p\), assume that \(p\nmid\ord(\sigma)\). For \(s\in A^{(\sigma)}\) semisimple, let \(T^{(\sigma)}\) be a maximal torus of \(G^s\), \(S^{(\sigma)}=\langle T^{(\sigma)},s \rangle\), \(M^{(\sigma)}=sT^{(\sigma)}\), and \(W^{(\sigma)}=N_G(M^{(\sigma)})/T^{(\sigma)}\). 
\begin{enumerate}
\item[(1)] The subset \(M^{(\sigma)}\subset A^{(\sigma)}\), and hence also \(T^{(\sigma)}\) and \(S^{(\sigma)}\), depend up to \(G\)-conjugacy only on \(\sigma\), and not on the choice of semisimple \(s\in A^{(\sigma)}\).

\item[(2)] \(\dim S^{(\sigma)}=\dim T^{(\sigma)}=\ell(\sigma)\), where \(\ell(\sigma)\) is defined by Equation \eqref{df:l(sigma)}.
%:=\min\{\dim\g^a\mid a\in A^{(\sigma)}\}\).

\item[(3)] The set of regular semisimple elements of \(A^{(\sigma)}\) is open and nonempty.

\item[(4)] Any semisimple element of \(A^{(\sigma)}\) is \(G\)-conjugate to an element of \(M^{(\sigma)}\).

\item[(5)] Two elements of \(M^{(\sigma)}\) are \(G\)-conjugate if and only if they lie in the same \(W^{(\sigma)}\)-orbit.
\end{enumerate}
\end{theorem}

\begin{proof}
Assertions \((1)\) and \((4)\) are precisely~\Cref{prop:independence-of-M}. It remains to prove \((2)\), \((3)\), and \((5)\).

We first prove \((2)\) and \((3)\). Since \((S^{(\sigma)})^\circ=T^{(\sigma)}\), we have \(\dim S^{(\sigma)}=\dim T^{(\sigma)}\). Consider the subset 
\[
U=\{s\in A^{(\sigma)}\mid s\text{ is semisimple and }\dim\g^s=\dim T^{(\sigma)}\}
\] 
of \(A^{(\sigma)}\). We claim that \(U\) is nonempty and open. 
Indeed, by~\Cref{lem:regular-open-slice}, we have \(U\cap M^{(\sigma)}=M_0^{(\sigma)}\). By \((4)\), any semisimple \(s\in A^{(\sigma)}\) is \(G\)-conjugate to some \(x\in M^{(\sigma)}\). Since \(\dim\g^s=\dim\g^x\), this implies that \(U\) is the union of all \(G\)-conjugates of \(M_0^{(\sigma)}\), so the claim follows from \Cref{lem:orbit-map-open}. 
Now, the set \(\{a\in A^{(\sigma)}\mid\operatorname{rank}(1-\Ad(a))\le k\}\) is Zariski closed in \(A^{(\sigma)}\) for any integer \(k\), hence the set \(\{a\in A^{(\sigma)}\mid\dim\g^a<\dim T^{(\sigma)}\}\) is open and, having empty intersection with \(U\), it must be empty. It follows that \(\ell(\sigma)=\dim T^{(\sigma)}\) and that \(U\) coincides with the set of regular semisimple elements of \(A^{(\sigma)}\).

Finally, we consider \((5)\). The ``if'' part is clear.
Conversely, let $x,y\in M^{(\sigma)}$ and assume that $y=gxg^{-1}$ for some $g\in G$.
Write
\(x=s t_x,\, y=s t_y\), where
\(t_x,t_y\in T^{(\sigma)}\). Since $T^{(\sigma)}$ centralizes both $s$ and $t_x$, we have $T^{(\sigma)}\subseteq G^x$. If
$R\subseteq G^x$ is a torus containing $T^{(\sigma)}$, then every element of $R$ centralizes both
$x=st_x$ and $t_x\in T^{(\sigma)}$, hence centralizes
\(s=x\,t_x^{-1}\).
Thus $R\subseteq G^s$, and hence 
$R=T^{(\sigma)}$. So $T^{(\sigma)}$ is maximal in $G^x$; similarly it is maximal in $G^y$. Now $g\,T^{(\sigma)}\,g^{-1}$ is also a maximal torus of $G^y$. Since maximal tori of \(G^y\) are conjugate, there exists
$u\in G^y$ such that
\(u\,g\,T^{(\sigma)}\,g^{-1}u^{-1}=T^{(\sigma)}
\), and so $ug\in N_G(T^{(\sigma)})$. 
Since $u$ centralizes $y$, we have
\((ug)x(ug)^{-1}=y\). Thus 
\((ug)\,M^{(\sigma)}\,(ug)^{-1}=
M^{(\sigma)}\), and hence $ug\in N_G(M^{(\sigma)})$. Therefore the elements $x$ and $y$ lie in the
same $W^{(\sigma)}$-orbit.
\end{proof}

\begin{remark}\label{rem:OV-recovery}
The proof of~\Cref{thm:sem-conj-outer-reorganized} follows the outline of the proof of~\cite[Thm.~3.12]{Vin}     
(attributed to Gantmakher), but with characteristic-independent methods. That theorem discusses the semisimple automorphisms of a complex semisimple Lie algebra, which corresponds to the case $A=G \rtimes \Gamma$ with \(G\) semisimple of adjoint type, and $K=\bb{C}$. We will return to this setting in~\Cref{subsec-cyclic-gradings-on-simple}, but for arbitrary characteristic.  
\end{remark}

\begin{remark}\label{rem:Mohrdieck-Cartan-comparison}
Items (4) and (5) of \Cref{thm:sem-conj-outer-reorganized} reproduce the semisimple part of Mohrdieck's description of the fibers of the affine quotient \(\sigma G/\!/G\) for \(A=G\rtimes\Gamma_0\) with \(G\) semisimple, but we do not have restrictions on \(K\) 
present in \cite{HamburgFoldingThesis}. Mohrdieck proves
that each fiber contains exactly one semisimple \(G\)-conjugacy class, and that this class is the unique closed orbit in the fiber (see \cite[Cor.~3.1(iii)]{HamburgFoldingThesis}).
\end{remark}

%\end{document}

\subsection{The pinned semidirect-product case}
\label{subsec:pinned-semidirect-product-case}

We now specialize the general setup of Section~\ref{subsec:the-general-setup}
to a semidirect product defined by pinned automorphisms of a semisimple group, 
which we now recall.

Let \(G\) be a (connected) semisimple algebraic group, and fix a pinning
\(\bigl(B,T,(\epsilon_\alpha)_{\alpha\in\Pi}\bigr)\) of \(G\), i.e., a maximal torus \(T\), 
a Borel subgroup \(B\) containing \(T\), and the isomorphism \(\epsilon_\alpha\) 
from the additive algebraic group \(K\) onto the root subgroup of \(G\) for each \(\alpha\in\Pi\), where 
\(\Pi\) is the base of the root system \(\Phi=\Phi(G,T)\) corresponding to \(B\). 
Let \((\mathfrak X(T),\mathfrak Y(T),\Phi,\Phi^\vee;\Pi,\Pi^\vee)\) be the based root datum 
of \(G\) relative to \((B,T)\). 
Denote by \(\Gamma_G\) the automorphism group of the based root datum. 
Restriction to \(\Pi\) embeds \(\Gamma_G\) into \(\operatorname{Aut}\Pi\). 
If \(G\) is of adjoint or simply connected type, then this embedding is an isomorphism.

By the isomorphism theorem for pinned reductive groups, the pinning determines, for
every \(\gamma\in\Gamma_G\), a unique automorphism
\(\dot\gamma\in\operatorname{Aut}G\) satisfying \(\dot\gamma(B)=B\),
\(\dot\gamma(T)=T\), and
\(\dot\gamma\circ\epsilon_\alpha=\epsilon_{\gamma(\alpha)}\) for every
\(\alpha\in\Pi\) (see e.g.~\cite[Exp.~XXIV, Thm.~1.3]{sga3} and
\cite[Prop.~7.1.6, Thm.~7.1.9]{ConradRGS}). Thus
\(\gamma\mapsto\dot\gamma\) defines a splitting of the exact sequence
\begin{equation}\label{eq:aut-split-root-datum}
1\longrightarrow\operatorname{Int}G\longrightarrow\operatorname{Aut}G
\longrightarrow\Gamma_G\longrightarrow1.
\end{equation}
Hence, for any subgroup \(\Gamma_0\le\Gamma_G\), we can define the semidirect product \(G\rtimes\Gamma_0\) where the action of \(\Gamma_0\) on \(G\) is via the above splitting.

\begin{remark}  
Changing the pinning replaces the splitting \(\Gamma_G\to\operatorname{Aut}G\) 
of the sequence \eqref{eq:aut-split-root-datum} by a conjugate one. Hence the corresponding semidirect products \(G\rtimes\Gamma_0\) are isomorphic.
\end{remark}

These semidirect products can be characterized in more abstract terms:

\begin{proposition}
Let \(A\) be an affine algebraic group with semisimple identity component \(G\). 
Define a homomorphism \(\rho\cl A\rightarrow\Aut G\) by \(a\mapsto\Int(a)|_G\), and let \(\bar\rho\) be the induced homomorphism \(A/G\rightarrow\Out G\cong\Gamma_G\). 
Then \(A\cong G\rtimes\Gamma_0\) for some subgroup \(\Gamma_0\le\Gamma_G\) if and only if (i) \(C_A(G)\subset G\) and (ii) there is a splitting \(f\) of the quotient map \(\pi\colon A\to A/G\)
such that, for any \(\sigma\in A/G\), the action of \(f(\sigma)\) on \(G\) is the pinned automorphism corresponding to \(\bar\rho(\sigma)\). Moreover, if \(G\) is of adjoint type, condition (ii) holds automatically.
\end{proposition}

\begin{proof}
Condition (i) is equivalent to saying that \(\bar\rho\) is injective, so it holds in the case \(A\cong G\rtimes\Gamma_0\). Under condition (i), \(\bar\rho\) maps \(A/G\) isomorphically onto a subgroup \(\Gamma_0\le\Gamma_G\), hence \(A\) is isomorphic to \(G\rtimes\Gamma_0\) with pinned action if and only if condition (ii) holds.

Since \(\operatorname{Int}G\cong G/Z(G)\) is isomorphic to the adjoint group \(G_{\ad}\) of the same type as \(G\), we can use this isomorphism to make \(\Int G\), and hence \(\Aut G\), an algebraic group. Then \(\rho\colon A\to\Aut G\) is a homomorphism of algebraic groups, which in the case \(G=G_{\ad}\) becomes a closed embedding as long as (i) holds. It then follows that \(A\cong G\rtimes \Gamma_0\) with pinned action.
\end{proof}

Using the isomorphism \(A\cong\ G\rtimes\Gamma_0\), we will identify \(\Gamma_0\) 
with its image in \(A\). Thus, for \(\sigma\in\Gamma_0\), the component denoted \(A^{(\sigma)}\)
in Section~\ref{subsec:the-general-setup} is the coset \(\sigma G\). 
We will denote \(q:=\operatorname{ord}(\sigma)\) and assume throughout this subsection that
\(\operatorname{char}K\nmid q\) so that the element \(\sigma\in A\) is semisimple. 
We will also denote $\Lambda=\frk X(T)$, $\Lambda_\sigma=(\sigma-1)\Lambda$, and $\sqrt{\Lambda_\sigma}=\{\chi\in\Lambda \mid n \chi \in \Lambda_\sigma  \text{ for some } n \geq 1\}$.

The constructions of~\Cref{subsec:the-general-setup} produce the torus \(T^{(\sigma)}\) only up to \(G\)-conjugacy. In the current setup, we will see that there is a canonical choice of $T^{(\sigma)}$, namely, $(T^{\sigma})^\circ$. 

\begin{lemma}\label{lem:fixed-torus-equals-centralizer}
Let \(S:=(T^\sigma)^\circ\). Then \(C_G(S)=T\). Moreover, if $s=\sigma h \in \sigma S$, then \(S\) is a maximal torus of \(G^s\).
\end{lemma}

\begin{proof}
For a subtorus \(S\subseteq T\), the centralizer \(C_G(S)\) is connected reductive, contains \(T\), and its roots with respect to \(T\) are exactly the
roots \(\alpha\in\Phi(G,T)\) whose restriction to \(S\) is trivial (see e.g.~\cite[\S26.2,~\S22.3]{Hum}). It is therefore enough to prove that no root of
\(G\) restricts trivially to \(S\).
Let \(V:=\Lambda\otimes_{\mathbb Z}\mathbb Q\). Since \(\sigma\) has finite order \(q\),
\(V=V^\sigma\oplus(\sigma-1)V\), and the projection onto \(V^\sigma\) is
\(p_\sigma(v):=q^{-1}\sum_{i=0}^{q-1}\sigma^i(v)\). 
By~\Cref{cor:kernel-translation-lattice}, the kernel of the restriction map \(\Lambda\rightarrow \frk X(S)\) is \(\sqrt{\Lambda_\sigma} = \Lambda \cap \bb Q\Lambda_\sigma = \Lambda\cap(\sigma-1)V = \Ker(p_\sigma|_\Lambda)\). Therefore, for \(\lambda\in\Lambda\), the restriction \(\lambda|_S\) is trivial if and only if \(p_\sigma(\lambda)=0\).
Let \(\alpha\in \Phi=\Phi(G,T)\). Since \(\sigma\) stabilizes \(\Pi\), it preserves the  sets of positive and negative roots. For \(\alpha\in\Phi^+\), all roots \(\sigma^i(\alpha)\) are positive, so their sum is nonzero in
\(V\), and similarly for \(\alpha\in\Phi^-\). Thus no root vanishes on \(S\), and hence
\(C_G(S)=T\).

Finally, let \(R\) be a torus of \(G^s\)
containing \(S\). Since \(R\) is commutative, it centralizes \(S\), and hence
\(R\subseteq C_G(S)=T\). For \(r\in R\), the equality \(rs=sr\) gives
\(r\sigma h=\sigma h r\). Since \(r,h\in T\), this is equivalent to
\(\sigma^{-1}r\sigma=r\), i.e., \(r\in T^\sigma\). Thus
\(R\subseteq T^\sigma\), and hence \(R\subseteq (T^\sigma)^\circ=S\), which completes the proof. 
\end{proof}

\begin{remark}\label{rem:folded}
There is an isomorphism \(\frk X(S)\to p_\sigma(\Lambda)\) that sends \(\lambda|_S\mapsto p_\sigma(\lambda)\) for any \(\lambda\in\Lambda\).
\end{remark}

By~\Cref{lem:fixed-torus-equals-centralizer}, we may take \(s=\sigma\) and \(T^{(\sigma)}=S\). 
Hence, from this point on, we put \(T^{(\sigma)}=S\), \(M^{(\sigma)}=\sigma S\), and \(S^{(\sigma)}=\langle S,\sigma\rangle=S\times\langle\sigma\rangle\). 

\begin{corollary}\label{cor:normalizer-M-equals-normalizer-S}
The subgroup \(N^{(\sigma)}=N_G(M^{(\sigma)})\) coincides with \(N_G(S^{(\sigma)})\).
\end{corollary}

\begin{proof}
From the discussion above, \(M^{(\sigma)}=S^{(\sigma)}\cap\sigma G\). Since \(G\) is normal in \(A\), for any \(g\in N_G(S^{(\sigma)})\), we have
\(gM^{(\sigma)}g^{-1}
=\bigl(gS^{(\sigma)}g^{-1}\bigr)\cap\bigl(g(\sigma G)g^{-1}\bigr)
=S^{(\sigma)}\cap\sigma G
=M^{(\sigma)}\), so \(g\in N_G(M^{(\sigma)})\). This proves \(N_G(S^{(\sigma)})\subseteq N_G(M^{(\sigma)})\).

The reverse inclusion follows from \(S^{(\sigma)}=\langle M^{(\sigma)}\rangle\) .
\end{proof}

Consequently, \(W^{(\sigma)}=N_G(S^{(\sigma)})/S\) is analogous to the classical Weyl group \(W=N_G(T)/T\), which corresponds to the case \(\sigma=1\). To describe the structure of \(W^{(\sigma)}\) in the next section, we will need the following:

\begin{proposition}\label{prop:kernel-translation-lattice}
Let \(L_\sigma=\sqrt{\Lambda_\sigma}/(1-\sigma)
\sqrt{\Lambda_\sigma}\). There is a canonical isomorphism
\[(T/S)^\sigma\cong
\Hom_{\mathbb Z}(L_\sigma,K^\times).\]
Moreover, \(L_\sigma\) is finite and \(q\)-periodic.
\end{proposition}

\begin{proof}
By~\Cref{cor:kernel-translation-lattice}, one has
\(\mathfrak X(T/S)=\sqrt{\Lambda_\sigma}\). Applying
\Cref{prop:fixed-points-quotient-quasitorus}(1) to the quasitorus \(Q:=T/S\),
whose character group is \(\sqrt{\Lambda_\sigma}\), gives
\((T/S)^\sigma\cong
\Hom_{\mathbb Z}(\sqrt{\Lambda_\sigma}/(1-\sigma)\sqrt{\Lambda_\sigma},
K^\times)
=\Hom_{\mathbb Z}(L_\sigma,K^\times)\).

Now put \(V:=\Lambda\otimes_{\mathbb Z}\mathbb Q\). Then 
\(\sqrt{\Lambda_\sigma}=\Lambda\cap (1-\sigma)V\). 
The endomorphism \(1-\sigma\) of \(V\) is semisimple and hence injective on \((1-\sigma)V\).  Therefore \((1-\sigma)\sqrt{\Lambda_\sigma}\) has the same rank as
\(\sqrt{\Lambda_\sigma}\). It follows that the quotient
\(L_\sigma=\sqrt{\Lambda_\sigma}/(1-\sigma)\sqrt{\Lambda_\sigma}\) is finite.

For any \(m\in\sqrt{\Lambda_\sigma}\), we have
\(\sigma m-m\in(1-\sigma)\sqrt{\Lambda_\sigma}\), so \(\sigma\) acts
trivially on \(L_\sigma\). On the other hand, since \(p_\sigma\) annihilates 
\(\sqrt{\Lambda_\sigma}\subset (1-\sigma)V\), we have
\(0=(1+\sigma+\cdots+\sigma^{q-1})m\equiv qm\pmod{(1-\sigma)\sqrt{\Lambda_\sigma}}\). 
Therefore \(qL_\sigma=0\).
\end{proof}

\section{The \texorpdfstring{$\sigma$}{sigma}-root space decomposition and the \texorpdfstring{$\sigma$}{sigma}-Weyl group}\label{sec:gen-root-decomp}

Throughout this section, we retain the standing assumptions of~\Cref{subsec:pinned-semidirect-product-case}: \(A=G \rtimes \Gamma_0\) where \(G=A^\circ\) is a semisimple algebraic group of arbitrary isogeny type and \(\Gamma_0\le\Gamma_G\) acts on \(G\) via a fixed pinning \(\bigl(B,T,(\epsilon_\alpha)_{\alpha\in\Pi}\bigr)\). Also fix \(\sigma\in\Gamma_0\) of order \(q\) not divisible by \(\operatorname{char}K\) and put \(S=(T^\sigma)^\circ\) and \(S^{(\sigma)}=S\times\langle\sigma\rangle\).

\subsection{The \texorpdfstring{$\sigma$}{sigma}-root space decomposition}

We begin by comparing the three gradings that arise naturally in our setup: the root space decomposition of \(\g=\Lie(G)\) relative to $T$ and the weight decompositions of \(\g\) relative to \(S\) and \(S^{(\sigma)}\). We already encountered the latter in \Cref{lem:zero-weight-space-outer}; it is called the \emph{\(\sigma\)-root space decomposition} because of its similarity with the usual root space decomposition, which we will see in \Cref{thm:sigma-root-properties-rewritten}.

Denote \(\Lambda=\frk X(T)\) and \(X=\frk X(S)\). 
The inclusions \(S\hookrightarrow T\) and \(S\hookrightarrow S^{(\sigma)}\) induce restriction homomorphisms \(\varphi\cl\Lambda\rightarrow X\) and \(\pi_X:\frk X(S^{(\sigma)})\rightarrow X\) respectively.
Coarsening the root space decomposition $\g=\frk h\oplus \bigoplus_{\alpha\in\Phi}\g_\alpha
$ along \(\varphi\) yields the \(S\)-weight decomposition: 
\begin{equation*}\label{eq:T-sigma-weight-dec}
\g=\h\oplus
\bigoplus_{0\ne \bar\alpha\in\Phi_\sigma}\g_{\bar\alpha},
\qquad \g_{\bar\alpha}
:=\bigoplus_{\alpha\in\Phi:\,\alpha|_{S}=\bar\alpha}\g_\alpha,
\end{equation*}
where we have used \Cref{lem:fixed-torus-equals-centralizer} for the \(0\)-component and denoted the set of restrictions of roots to \(S\) by \(\Phi_\sigma\). By \Cref{rem:folded}, \(\Phi_\sigma\) can be identified with the so-called \emph{folded root system} \(p_\sigma(\Phi)\).
The same \(S\)-weight decomposition is similarly obtained by coarsening the \(\sigma\)-root space decomposition along \(\pi_X\).

We now fix coordinates on \(\frk X(S^{(\sigma)})\). Since \(S\) is the identity component of \(S^{(\sigma)}\), \(\frk X(\langle\sigma\rangle)\) is canonically identified with the torsion  subgroup of \(\frk X(S^{(\sigma)})\). Moreover, in our setup \(S^{(\sigma)}= S\times\langle\sigma\rangle\), so we can also identify \(X\) with the subgroup \(\langle\sigma\rangle^\perp\) of \(\frk X(S^{(\sigma)})\). 
Fixing a primitive \(q\)-th root of unity \(\zeta_q\in K^\times\) allows us to identify 
\(\frk X(\langle\sigma\rangle)\) with \(\bb{Z}_q\) such that \(\bar k\in\bb{Z}_q\) corresponds to the character defined by \(\sigma\mapsto\zeta_q^k\). Thus we obtain a (noncanonical) identification of \(\frk X(S^{(\sigma)})\) with \(X\oplus\bb{Z}_q\). Explicitly,
\((\lambda,\bar k)\) corresponds to the character \(\chi_{(\lambda,\bar k)}\)
given by \(\chi_{(\lambda,\bar k)}(h\sigma^u)=\lambda(h)\zeta_q^{ku}\), for 
\(h\in S\) and \(u\in\mathbb Z\). 
With these coordinates, the \(S^{(\sigma)}\)-weight decomposition can be written as 
\[
\g=\g_0\oplus
\bigoplus_{0\ne (\bar\alpha,\bar k)\in X\oplus \bb{Z}_q}\g_{(\bar\alpha,\bar k)},
\]
where \(\g_0=\mathscr L(S)\) by \Cref{lem:zero-weight-space-outer} (with \(T^{(\sigma)}=S\)). As already mentioned, coarsening along the projection \(\pi_X\) gives the \(S\)-weight decomposition, while the projection to \(\Z_q\) gives the eigenspace decomposition of \(\sigma\) acting on \(\g\):
%shows that the \(S^{(\sigma)}\)-grading refines the \(S\)-grading:
%$\g_{\bar\alpha}=\bigoplus_{\bar k\in \bb{Z}_q}\g_{(\bar\alpha,\bar k)}\(\bar\alpha\in X)$.
%Finally, projection to the second factor coarsens the \(S^{(\sigma)}\)-grading to the $\bb{Z}_q$-grading given by the eigenspace decomposition for $\sigma$, $$, where
\begin{equation}\label{eq:sigma-eigenspac-decom}
%\g_{\bar k}=\bigoplus_{\bar\alpha\in X}\g_{(\bar\alpha,\bar k)},
\g=
\bigoplus_{\bar k\in \bb{Z}_q}\g_{\bar k},
\quad \g_{\bar k}:=\{x\in \g\mid \Ad(\sigma)(x)=\zeta_q^k x\}.
%\quad \bar k\in \bb{Z}_q.
\end{equation}
Thus the \(\sigma\)-root decomposition simultaneously refines the \(S\)-weight decomposition and the \(\sigma\)-eigenspace decomposition: \(\g_{(\bar\alpha,\bar k)}=\g_{\bar\alpha}\cap\g_{\bar k}\). 

\begin{definition}\label{def:sigma-roots-real}
The elements of
\(\Phi^{(\sigma)}
=\left\{\alpha\in\frk X(S^{(\sigma)})\setminus\{0\}
\mid\g_\alpha\neq0
\right\}\) 
%defined in~\Cref{lem:regular-open-slice} 
are called the \emph{\(\sigma\)-roots} of \(\g\) with respect to
\(S^{(\sigma)}\).  Writing
$\alpha=(\bar\alpha,\bar k)\in X\oplus \bb{Z}_q$, we say that $\alpha$ is
\emph{real} if $\bar\alpha\neq 0\) and \emph{imaginary} otherwise. We write  \(\Phi_{\mathrm{re}}^{(\sigma)}\) for the subset of real \(\sigma\)-roots, and 
\[\bar\Phi^{(\sigma)}:=\{\bar\alpha\in X\mid \alpha\in \Phi_{\mathrm{re}}^{(\sigma)}\}
\subset X\]
for the set of their restrictions to $S$. 
\end{definition}

Note that \(\bar\Phi^{(\sigma)}=\Phi_\sigma\), the folded root system attached to \((\Phi,\sigma)\). The imaginary \(\sigma\)-roots are the elements of
\(\Phi^{(\sigma)}\cap (\{0\}\oplus\mathbb Z_q)\). 

The relations among these decompositions are summarized by the diagram in \Cref{fig:coarsening-refinement-diagram}, in which an arrow points from a finer decomposition to a coarser one.

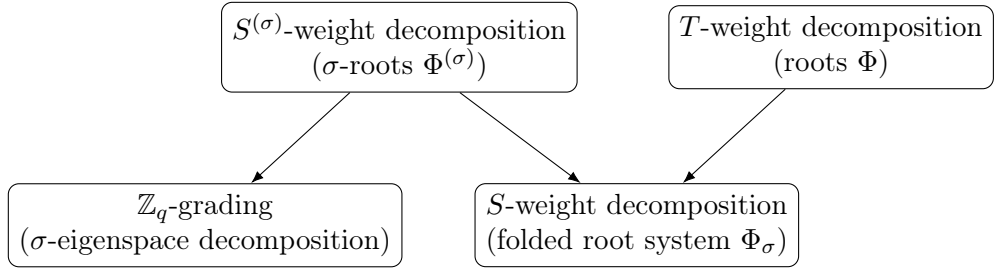
\begin{figure}[htbp]
\centering
\begin{tikzpicture}[
  node distance=14mm and 18mm,
  >=Latex,
  every node/.style={
    draw,
    rectangle,
    rounded corners,
    align=center,
    inner sep=4pt
  }
]
\node (S) {\(S^{(\sigma)}\)-weight decomposition\\(\(\sigma\)-roots $\Phi^{(\sigma)}$)};
\node (T) [right=12mm of S] {\(T\)-weight decomposition \\(roots $\Phi$)};
\node (Ts) [below=18mm of $(S)!0.55!(T)$] {\(S\)-weight decomposition \\ (folded root system $\Phi_\sigma$)};
\node (sig) [below left=12mm and -24mm of S] {\(\bb{Z}_q\)-grading\\(\(\sigma\)-eigenspace decomposition)};

\draw[->] (S) -- (Ts);
\draw[->] (S) -- (sig);
\draw[->] (T) -- (Ts);
\end{tikzpicture}
\caption{Coarsening relations among the decompositions attached to \(T\), \(S\), \(\sigma\), and \(S^{(\sigma)}\).}
\label{fig:coarsening-refinement-diagram}
\end{figure}

We next record the structure of a connected reductive group attached to a real \(\sigma\)-root. This will be used to prove the one-dimensionality of real \(\sigma\)-root spaces and to rule out nontrivial multiples of real \(\sigma\)-roots in \(\frk X(S^{(\sigma)})\).

\begin{lemma}\label{lem:rank-one-centralizer-real-root}
Let \(\alpha=(\bar\alpha,\bar k)\in \Phi^{(\sigma)}_{\mathrm{re}}\), and set
\( Z_\alpha:=(C_G(\Ker\alpha))^\circ\).
Then \(Z_\alpha\) is a connected reductive subgroup of \(G\), the torus \(S\) is a maximal torus of
\(Z_\alpha\), and \(\rk_{\mathrm{ss}}(Z_\alpha)=1\).
\end{lemma}
\begin{proof}
Since \(S\) centralizes \(S^{(\sigma)}\), we have \(S\subseteq Z_\alpha\). Moreover, \(\Ker\alpha\) is a quasitorus, so
\(C_G(\Ker\alpha)\) is reductive and
\(Z_\alpha=(C_G(\Ker\alpha))^\circ\) is connected reductive.

We first prove that \(S\) is a maximal torus of \(Z_\alpha\). 
Let \(R\subseteq Z_\alpha\) be a torus containing \(S\). Let \(k\in\{0,\dots,q-1\}\) be the integer representing \(\bar k\in\mathbb Z_q\). Since \(\bar\alpha\cl S\to K^\times\) is a nontrivial character of a
torus, it is surjective. Choose \(t_\alpha\in S\) such that
\(\bar\alpha(t_\alpha)=\zeta_q^{-k}\), and set
\(x:=\sigma t_\alpha\). Then \(\alpha(x)=1\), so
\(x\in\Ker\alpha\).
Since \(R\) centralizes \(S\), the equality
\(C_G(S)=T\) gives \(R\subseteq T\). On the other hand, \(R\subseteq Z_\alpha\), so every \(r\in R\) centralizes \(x=\sigma t_\alpha\). 
Since \(r,t_\alpha\in T\), this implies \(r\in T^\sigma\). Therefore \(R=(T^\sigma)^\circ=S\).

Set \(T_{\bar\alpha}:=(\Ker\bar\alpha)^\circ\subseteq S\).
Since \(\bar\alpha\neq0\), the torus \(T_{\bar\alpha}\) has codimension
\(1\) in \(S\). Moreover, \(T_{\bar\alpha}\subseteq\Ker\alpha\), so every element of \(Z_\alpha\) centralizes \(T_{\bar\alpha}\). Hence
\(T_{\bar\alpha}\subseteq Z(Z_\alpha)^\circ\).
Let \(\pi\cl Z_\alpha\rightarrow Z_\alpha/Z(Z_\alpha)^\circ\) be the semisimple quotient. 
Because \(S\) is a maximal torus of
\(Z_\alpha\), its image \(\pi(S)\) is a maximal torus of the quotient. Furthermore,
\(\Ker(\pi|_{S})\supseteq T_{\bar\alpha}\), and therefore
\(\rk_{\mathrm{ss}}(Z_\alpha)=\dim\pi(S)
\leq\dim(T^\sigma/T_{\bar\alpha})=1\).
Finally, since \(\alpha\in\Phi_{\mathrm{re}}^{(\sigma)}\), one has \(\g_\alpha\neq0\). The quasitorus \(\Ker\alpha\) acts trivially on
\(\g_\alpha\), so \(0\neq\g_\alpha\subseteq
\mathfrak c_\g(\Ker\alpha)=\mathscr L(Z_\alpha)\), where the last equality follows from~\cite[\S18.4,~Prop.~A]{Hum}.
Thus \(Z_\alpha\) has a nonzero \(S\)-weight \(\bar\alpha\) and is not a torus.
Hence \(\rk_{\mathrm{ss}}(Z_\alpha)=1\). 
\end{proof}

\begin{theorem}\label{thm:sigma-root-properties-rewritten}
For the \(\sigma\)-root decomposition of \(\g\), the following hold:
\begin{enumerate}
\item[(1)] \(\dim \g_\alpha=1\) for every \(\alpha\in \Phi^{(\sigma)}_{\mathrm{re}}\);
\item[(2)] if \(\alpha\in \Phi^{(\sigma)}_{\mathrm{re}}\), then \(m\alpha\notin \Phi^{(\sigma)}_{\mathrm{re}}\) for
every \(m\neq \pm1\);
\item[(3)] the set \(\bar\Phi^{(\sigma)}\subset E:=\frk X(S)\otimes_\mathbb Z \mathbb R\) is a root system (possibly nonreduced).
\end{enumerate}
\end{theorem}

\begin{proof}
(1) By \Cref{lem:rank-one-centralizer-real-root}, \(Z_\alpha\) is connected reductive of semisimple rank \(1\), with maximal torus \(S\). Hence the root system \(\Phi(Z_\alpha,S)\) has type \(A_1\), and the corresponding weight spaces in \(\frk z_\alpha:=\mathscr L(Z_\alpha)\) are one-dimensional (see e.g.~\cite[\S26.2,~Cor.~B]{Hum}). 
%Since \(\Ker\alpha\) acts trivially on \(\g_\alpha\), we have
%\[\g_\alpha\subseteq \frk c_{\g}(\Ker\alpha)=\mathscr L(C_G(\Ker\alpha))=\frk z_\alpha.\] 
We saw in the proof of \Cref{lem:rank-one-centralizer-real-root} that \(\bar\alpha\) occurs as a nonzero \(S\)-weight in \(\frk z_\alpha\).  Hence the nonzero weights are precisely \(\pm\bar\alpha\). In particular, \((\frk z_\alpha)_{\bar\alpha}\) is one-dimensional. Since \(\g_\alpha\subseteq(\frk z_\alpha)_{\bar\alpha}\) and \(\g_\alpha\neq0\), it follows that \(\dim\g_\alpha=1\).

(2) Suppose \(m\alpha\in\Phi^{(\sigma)}_{\mathrm{re}}\). 
Then \(\Ker\alpha\) acts trivially on \(\g_{m\alpha}\), so \(0\neq\g_{m\alpha}\subseteq\frk z_\alpha\). 
Thus \(m\bar\alpha\) is a nonzero \(S\)-weight of \(\frk z_\alpha\). 
Since the only nonzero \(S\)-weights in \(\frk z_\alpha\) are \(\pm\bar\alpha\), we have \(m\bar\alpha=\pm\bar\alpha\). As \(\frk X(S)\) is torsion-free and \(\bar\alpha\neq0\), this gives \(m=\pm1\).

(3) Recall that \(\bar\Phi^{(\sigma)}=\Phi_\sigma\) can be identified with the folded root system attached to \((\Phi,\sigma)\). The latter is known to be a possibly nonreduced root system (see~\cite[Cor.~13.1.4,~Prop.~13.2.2]{carter}). 
\end{proof}

\begin{remark}\label{rem:nonreduced-no-contradiction}
There is no contradiction between (2) in \Cref{thm:sigma-root-properties-rewritten} and the possible nonreducedness of \(\bar\Phi^{(\sigma)}\). The theorem rules out nontrivial multiples of a real \(\sigma\)-root inside \(\frk X(S^{(\sigma)})\), but it may happen that two real
$\sigma$-roots restrict to $\bar\alpha$ and $2\bar\alpha$. This occurs in type \(A_{2n}\) with \(\sigma\) of order \(2\), where \(\bar\Phi^{(\sigma)}\) has type \(BC_n\).
\end{remark}

\subsection{The structure of \texorpdfstring{$W^{(\sigma)}$}{Wsigma}}\label{subsec:red-to-irreducible-case}

Since \(\sigma\) normalizes \(T\), conjugation by \(\sigma\) induces
actions on \(W=N_G(T)/T\) and \(T/S\). Hence we can consider the groups of fixed points 
\(W^\sigma\) and \(W_0^{(\sigma)}:=(T/S)^\sigma\).

\begin{proposition}\label{prop:W-sigma-to-W}
Assume we are in the setup of~\Cref{subsec:pinned-semidirect-product-case}.
Let \(\sigma\in\Gamma_0\) have order \(q\), with \(\op{char}K\nmid q\). 
%Let \(S^{(\sigma)}:=\langle S,\sigma\rangle\), \(W^{(\sigma)}:=N_G(S^{(\sigma)})/S\), and \(W:=N_G(T)/T\). 
Then the homomorphism \(\phi\colon W^{(\sigma)}\to W^\sigma\), defined by
\(\phi(nS)=nT\), fits into a split exact sequence
\[1\longrightarrow W_0^{(\sigma)}\longrightarrow W^{(\sigma)}
\xrightarrow{\ \phi\ }W^\sigma\longrightarrow1.\] 
Consequently, \(W^{(\sigma)}\cong W_0^{(\sigma)}\rtimes W^\sigma\), where the action of
\(W^\sigma\) on \(W_0^{(\sigma)}\) is induced by its natural action on \(T/S\).
\end{proposition}

\begin{proof}
We first verify that \(\phi\) is well defined and identify its kernel. Let
\(n\in N_G(S^{(\sigma)})\). Since \(S=(S^{(\sigma)})^\circ\), the element
\(n\) normalizes \(S\), and therefore also \(C_G(S)=T\) by
\Cref{lem:fixed-torus-equals-centralizer}. Hence \(n\in N_G(T)\). Moreover,
\(n\sigma n^{-1}\in\sigma S\) by \Cref{cor:normalizer-M-equals-normalizer-S}, so
\(n^\sigma n^{-1}\in S\subseteq T\). Consequently, the class \(nT\in W\) is
fixed by \(\sigma\), and the rule \(\phi(nS):=nT\) defines a homomorphism
\(\phi\colon W^{(\sigma)}\to W^\sigma\).

If \(tS\in(T/S)^\sigma\), then \(t^\sigma t^{-1}\in S\). Since \(t\in T\) centralizes \(S\), 
the identity \(t\sigma t^{-1}=\sigma \cdot t^\sigma t^{-1}\) shows that \(t\) normalizes
\(S^{(\sigma)}=\langle S,\sigma\rangle\). Therefore \(W_0^{(\sigma)}=(T/S)^\sigma\) is a subgroup of \(W^{(\sigma)}\); let \(\iota\) be the inclusion map. The kernel of \(\phi\) consists of all elements of the form \(tS\) where \(t\in T\) normalizes \(S^{(\sigma)}\).
The above identity shows that \(t\in N_G(S^{(\sigma)})\) if and only if
\(t^\sigma t^{-1}\in S\), equivalently \(tS\in(T/S)^\sigma\). Hence
\(\operatorname{Ker}\phi=\operatorname{Im}\iota\). 
%Thus the sequence is exact at \(W_0^{(\sigma)}\) and \(W^{(\sigma)}\).

It remains to construct a section of \(\phi\). Put \(H:=(G^\sigma)^\circ\).
By~\Cref{lem:fixed-torus-equals-centralizer}, \(S\) is a maximal torus of
\(H\). If \(n\in N_H(S)\), then \(n\) normalizes \(S^{(\sigma)}=\langle S,\sigma\rangle\), and hence \(nS\in W^{(\sigma)}\), so \(W(H,S)=N_H(S)/S\) is a subgroup of \(W^{(\sigma)}\).
By~\cite[Prop.~5.4(8)]{Achar.Weyl.grp.fix.pt}, the restriction of \(\phi\) to \(W(H,S)\) is an isomorphism onto \(W^{\sigma}\). 
%Injectivity: if \(nS\in\operatorname{Ker}\phi\), then \(n\in T\). 
%Since \(n\in H\), we have \(n\in C_H(S)=S\). 
The inverse of this isomorphism is the desired section.
\end{proof}

\subsection{A lattice model for \texorpdfstring{$W^{(\sigma)}$}{Wsigma} and its action}
\label{subsec:lattice-translation-action}
The purpose of this subsection is to replace the $\sigma$-Weyl group \(W^{(\sigma)}\) by an isomorphic group that is defined in terms of the root datum of \(G\). We will refer to this group as a ``lattice model'' to emphasize its independence of \(K\).

Recall \(\Lambda=\frk{X}(T)\) and \(\Lambda_\sigma:=(1-\sigma)\Lambda\). By \Cref{prop:kernel-translation-lattice}, we have 
\(W_0^{(\sigma)}\cong\operatorname{Hom}(L_\sigma,K^\times)\) where 
\(L_\sigma=\sqrt{\Lambda_\sigma}/(1-\sigma)\sqrt{\Lambda_\sigma}\) is finite and \(q\)-periodic.
Hence every homomorphism \(L_\sigma\to K^\times\) has image
in \(\boldsymbol{\mu}_q(K)\).
After fixing a primitive \(q\)-th root of unity \(\zeta_q\in K^\times\),
we have an isomorphism \(W_0^{(\sigma)}\cong W^{(\sigma)}_{0,\mathrm{lat}}\) where
\begin{equation*}\label{eq:Wtr-lattice-Zq}
W^{(\sigma)}_{0,\mathrm{lat}}:=\Hom(L_\sigma,\bb{Z}_q).
\end{equation*}
The action of \(W^\sigma\) on \(\Lambda\) commutes with \(\sigma\), 
hence \(W^\sigma\) acts naturally on \(L_\sigma\). This gives a
left action of \(W^\sigma\) on \(W^{(\sigma)}_{0,\mathrm{lat}}\): 
\((w\cdot f)(\ell)=f(w^{-1}\ell)\), for \(w\in W^\sigma\),
\(f\in W^{(\sigma)}_{0,\mathrm{lat}}\), and \(\ell\in L_\sigma\). With
respect to this action, we define the lattice model of the \(\sigma\)-Weyl group by
\begin{equation}\label{eq:sigma-weyl-group-lattice}
    W^{(\sigma)}_{\mathrm{lat}}:=W^{(\sigma)}_{0,\mathrm{lat}}\rtimes W^\sigma.
\end{equation}
The next results compute its action on \(\mathfrak X(S^{(\sigma)})\cong X\oplus \bb{Z}_q\).

\begin{lemma}\label{lem:deltaq-onestep}
Let \(\varphi:\Lambda\to X\) be the restriction map and denote by \([m]\) the image of any element \(m\in\sqrt{\Lambda_\sigma}\) in \(L_\sigma\). Then the assignment 
\(\delta(\lambda):=[(1-\sigma)\widetilde\lambda]\),
where \(\widetilde\lambda\in\Lambda\) is any preimage of \(\lambda\in\Lambda\) under \(\varphi\), defines a \(W^\sigma\)-equivariant homomorphism \(\delta:X\to L_\sigma\), which factors through \(\delta_q:X/qX\to L_\sigma\). 
\end{lemma}

\begin{proof}
Let \(\widetilde\lambda,\widetilde\lambda'\in\Lambda\) be two preimages of
\(\lambda\in X\). By \Cref{cor:kernel-translation-lattice},
\(\operatorname{Ker}\varphi=\sqrt{\Lambda_\sigma}\). Hence
\(\widetilde\lambda'-\widetilde\lambda\in\sqrt{\Lambda_\sigma}\), and
therefore
\((1-\sigma)(\widetilde\lambda'-\widetilde\lambda)
\in(1-\sigma)\sqrt{\Lambda_\sigma}\). Thus \(\delta\) is a well-defined homomorphism. 
Since \(qL_\sigma=0\), it factors through \(X/qX\).
The \(W^\sigma\)-equivariance of \(\delta\) follows from the fact that \(\varphi\) is \(W^\sigma\)-equivariant.
\end{proof}

The natural pairing \(\langle\ ,\ \rangle:X\times Y\to \mathbb Z\) between \(X=\frk{X}(S)\) and \(Y=\frk{Y}(S)\cong\Hom_\Z(X,\Z)\) induces, by reduction modulo \(q\), a pairing
\[\langle\ ,\ \rangle_q:(X/qX)\times (Y/qY)\longrightarrow \bb{Z}_q, \qquad \langle \bar\lambda,\bar\nu\rangle_q:=\overline{\langle \lambda,\nu\rangle},\]
which identifies $Y/qY$ with $\Hom(X/qX,\bb{Z}_q)$, via the map $\bar\nu\mapsto(\bar\lambda\mapsto\langle\bar\lambda,\bar\nu\rangle_q)$.

\begin{definition}\label{def:nubar-lat-onestep}
For $f\in W^{(\sigma)}_{0,\mathrm{lat}}$ define $\bar\nu(f)$ to be the unique element in $Y/qY$ corresponding to $f\circ\delta_q \in \Hom_{\mathbb Z}(X/qX,\bb{Z}_q)$ under the identification $Y/qY\cong \Hom(X/qX,\bb{Z}_q)$, i.e., the element \(\bar\nu(f)\in Y/qY\) is characterized by
\begin{equation}\label{eq:def-of-nu}
    f(\delta_q(\bar\lambda))=
\langle \bar\lambda,\bar\nu(f)\rangle_q  \quad \text{for all} \quad \bar\lambda\in X/qX.
\end{equation}
\end{definition}

Since \(\delta_q\) is \(W^\sigma\)-equivariant and the pairing \(\langle\ ,\ \rangle\) is \(W^\sigma\)-invariant, it follows that 
\(\bar{\nu}\cl W^{(\sigma)}_{0,\mathrm{lat}}\to Y/qY\) is a \(W^\sigma\)-equivariant homomorphism.
% for every \(\bar\lambda\in X/qX\) one has
% \[\langle\bar\lambda,\bar\nu(w\cdot f)\rangle_q
% =(w\cdot f)(\delta_q(\bar\lambda))
% =f(w^{-1}\delta_q(\bar\lambda))
% =\langle w^{-1}\bar\lambda,\bar\nu(f)\rangle_q
% =\langle\bar\lambda,w\bar\nu(f)\rangle_q.\]
% By nondegeneracy of the pairing, \(\bar\nu(w\cdot f)=w\bar\nu(f)\).

\begin{lemma}\label{lem:translation-action-onestep}
The subgroup \(W^{(\sigma)}_{0,\mathrm{lat}}\) acts on
\(\frk X(S^{(\sigma)})\cong X\oplus \bb{Z}_q\) by
\begin{equation}\label{eq:translation-action-onestep}
f\cdot(\lambda,\bar k)
=
\bigl(\lambda,\ \bar k+\langle \bar\lambda,\bar\nu(f)\rangle_q\bigr),\
f\in W^{(\sigma)}_{0,\mathrm{lat}},
\end{equation}
which corresponds to the action of $W_0^{(\sigma)}$ on $S^{(\sigma)}$.
\end{lemma}

\begin{proof}
Let \([t]:=tS\in W_0^{(\sigma)}\), and let \(\chi_{[t]}:L_\sigma\to K^\times\) be the corresponding homomorphism under \Cref{prop:kernel-translation-lattice}.  Let 
\(f\in W^{(\sigma)}_{0,\mathrm{lat}}=\Hom(L_\sigma,\bb{Z}_q)\)
correspond to \(\chi_{[t]}\) via the fixed primitive root \(\zeta_q\), i.e., $\chi_{[t]}(\ell)=\zeta_q^{f(\ell)}$ for all $\ell \in L_\sigma$.
Set \(c_t:=t^\sigma t^{-1}\in T\). Since \([t]\) is fixed by \(\sigma\) in
\(T/S\), we have \(c_t\in S\). In particular,
\(c_t\) is fixed by \(\sigma\). Now let \(h\in S\) and \(u\in\mathbb Z\). 
Then \(\Int(t^{-1})(h\sigma^u)
=h\,c_t^{-u}\sigma^u\), so we obtain
\begin{equation}\label{eq:chi-of-int}
\bigl([t]\cdot\chi_{(\lambda,\bar k)}\bigr)(h\sigma^u)=\bigl(\chi_{(\lambda,\bar k)}
\circ\Int(t^{-1})\bigr)(h\sigma^u)=
\lambda(h)\lambda(c_t)^{-u}\zeta_q^{ku}.
\end{equation}
It remains to identify the scalar \(\lambda(c_t)^{-1}\). Choose a preimage
\(\widetilde\lambda\in\Lambda\) of \(\lambda\in X\). Since
\(c_t\in S\), we have
\[\lambda(c_t)^{-1}
=\widetilde\lambda(t^\sigma t^{-1})^{-1}
=((1-\sigma)\widetilde\lambda)(t)=\chi_{[t]}(\delta(\lambda)).\]
By Definition~\ref{def:nubar-lat-onestep}, the class
\(\bar\nu(f)\in Y/qY\) is characterized by Equation~\eqref{eq:def-of-nu}.
Equivalently,
\(\chi_{[t]}(\delta(\lambda))=
\zeta_q^{\langle \bar\lambda,\bar\nu(f)\rangle_q}\).
Substituting this into Equation~\eqref{eq:chi-of-int} gives
\[\bigl([t]\cdot\chi_{(\lambda,\bar k)}\bigr)(h\sigma^u)=\lambda(h)
\zeta_q^{\left(k+\langle\bar\lambda,\bar\nu(f)\rangle_q\right)u}.\]
Therefore
\([t]\cdot\chi_{(\lambda,\bar k)}=\chi_{
\left(\lambda,\bar k+\langle\bar\lambda,\bar\nu(f)\rangle_q
\right)}\), which is exactly Equation~\eqref{eq:translation-action-onestep}.
\end{proof}

\begin{theorem}\label{prop:full-lattice-sigma-weyl-action}
Let \(\iota\cl W^{(\sigma)}_{0,\mathrm{lat}}\to W^{(\sigma)}_0\) be the isomorphism of \Cref{prop:kernel-translation-lattice} and let \(s\cl W^\sigma\to W^{(\sigma)}\) be the section constructed in \Cref{prop:W-sigma-to-W}. Then the isomorphism
\(W^{(\sigma)}_{0,\mathrm{lat}}\rtimes W^\sigma \xrightarrow{\sim} W^{(\sigma)}\),
\((f,w)\mapsto \iota(f)\,s(w)\), represents the action of \(W^{(\sigma)}\) on
\(\frk X(S^{(\sigma)})\cong X\oplus \bb{Z}_q\) as follows:
\begin{equation}\label{eq:full-gen-weyl-group-action}
(f,w)\cdot(\lambda,\bar k)=
\bigl(w\lambda,\bar k+\langle\overline{w\lambda},\bar\nu(f)\rangle_q\bigr).
\end{equation}
\end{theorem}

\begin{proof}
The action of the subgroup $W_{0,\mathrm{lat}}^{(\sigma)}$ was computed in
\Cref{lem:translation-action-onestep}: if
\(f\in W^{(\sigma)}_{0,\mathrm{lat}}\), then its action on
\(\frk X(S^{(\sigma)})\cong X\oplus\bb{Z}_q\) via the isomorphism \(\iota\) is
\(f \cdot (\lambda,\bar k)=
\bigl(\lambda,\bar k+\langle\bar\lambda,\bar\nu(f)\rangle_q\bigr)\).
We will now identify the action of the chosen representatives of
\(W^\sigma\).
By the construction of the section in
\Cref{prop:W-sigma-to-W}, every
\(w\in W^\sigma\) may be represented by an element
\(n\in N_{(G^\sigma)^\circ}(S)\), and the corresponding element of \(W^{(\sigma)}\) is
\(s(w)=nS\). Then \(n\) centralizes \(\sigma\) and normalizes \(S\). Hence, for
\(h\in S\) and \(u\in\mathbb Z\), the conjugation action on
\(S^{(\sigma)}\) gives
\((h\sigma^u)^n=n^{-1}h\sigma^u n=n^{-1}hn\,\sigma^u\). 
For any \(\chi_{(\lambda,\bar k)}\in\frk X(S^{(\sigma)})\), we then have
\[(n\chi_{(\lambda,\bar k)})(h\sigma^u)
=\chi_{(\lambda,\bar k)}((h\sigma^u)^n)
=\lambda(n^{-1}hn)\zeta_q^{ku}
=(w\lambda)(h)\zeta_q^{ku}.\]
Therefore the action of \(W^\sigma\) on \(\frk X(S^{(\sigma)})\cong X\oplus\bb{Z}_q\) 
via the section \(s\) is given by \(w\cdot(\lambda,\bar k)=(w\lambda,\bar k)\).
Equation~\eqref{eq:full-gen-weyl-group-action} will follow once we show that the bijection 
\((f,w)\mapsto\iota(f)\,s(w)\) is a group homomorphism. To this end, write \(s(w)=nS\), as above,
and \(\iota(f)=[t]:=tS\) for some \(t\in T\), as in the proof of \Cref{lem:translation-action-onestep}. Then \(s(w)\,\iota(f)\,s(w)^{-1}=[ntn^{-1}]=[t^{w^{-1}}]\). For any \(\ell\in L_\sigma\), we have
\[
\chi_{[t^{w^{-1}}]}(\ell)=\ell(t^{w^{-1}})=(w^{-1}\ell)(t)=\chi_{[t]}(w^{-1}\ell)=\zeta_q^{f(w^{-1}\ell)}=\zeta_q^{(w\cdot f)(\ell)},
\]
which means that \([t^{w^{-1}}]=\iota(w\cdot f)\). We have shown that 
\(s(w)\,\iota(f)\,s(w)^{-1}=\iota(w\cdot f)\), as required.
%
% Since \(\bar{\nu}\) is \(W^\sigma\)-equivariant, it follows that \(R_wT_fR_w^{-1}=T_{wf}\). Consequently, the assignment
% \((f,w)\mapsto T_fR_w\) defines an action of the semidirect product
% \(W^{(\sigma)}_{0,\mathrm{lat}}\rtimes W^\sigma\), whose multiplication is
% \((f,w)(f',w')=(f+wf',ww')\). Indeed,
% \(T_fR_wT_{f'}R_{w'}=T_fT_{wf'}R_{ww'}=T_{f+wf'}R_{ww'}\). Applying \(T_fR_w\) to \((\lambda,\bar k)\) gives \(T_fR_w(\lambda,\bar k)
% =T_f(w\lambda,\bar k)=\bigl(w\lambda,\bar k+\langle\overline{w\lambda},\bar\nu(f)\rangle_q\bigr)\),
% which is the asserted formula.
\end{proof}

\begin{remark}\label{rem:connectedTsigma}
The action of the \(\sigma\)-Weyl group on \(\frk{X}(S^{(\sigma)})\) is not necessarily faithful. Since \(C_G(S^{(\sigma)})=C_G(S)\cap G^\sigma=T^\sigma\), the kernel of this action in \(W^{(\sigma)}\) is \(T^\sigma/S\), which corresponds in the lattice model 
\(W^{(\sigma)}_{\mathrm{lat}}=W^{(\sigma)}_{0,\mathrm{lat}}\rtimes W^\sigma \) to the kernel of \(\bar\nu\). In particular, the action is faithful if and only if \(T^\sigma\) is connected. This is the case for \(G\) of simply connected or adjoint type, or more generally if \(\sigma\) acts on the lattice \(\Lambda=\frk{X}(T)\) by permuting the elements of a basis.
\end{remark}

\begin{corollary}
The classification of semisimple elements of finite order in \(G\rtimes\Gamma_0\) 
up to \(G\)-conjugation is independent of the algebraically closed field \(K\), 
except for the condition that the order is not divisible by \(p\) when \(\op{char}K=p\).
\end{corollary}

\begin{proof}
By \Cref{thm:sem-conj-outer-reorganized}, every such conjugacy class has a representative in \(M^{(\sigma)}=\sigma S\) for some \(\sigma\in\Gamma_0\), with two elements representing the same class if and only if they are in the same \(W^{(\sigma)}\)-orbit.
By \Cref{prop:full-lattice-sigma-weyl-action}, the action of \(W^{(\sigma)}\) on \(\frk X(S^{(\sigma)})\) is represented, after fixing a primitive \(q\)-th root of unity, by the action of the group \(W^{(\sigma)}_{\mathrm{lat}}\) on \(X\oplus\bb{Z}_q\). 
But the elements of finite order in \(M^{(\sigma)}\) correspond to the set \(\Hom(X,\mathrm{t}(K^\times))\). The result follows.
\end{proof}

Our next goal is to include a semisimple analog of elements of order $p$, namely, subgroupschemes $\boldsymbol{\mu}_{p^k}$.

\section{Classification of conjugacy classes of morphisms}
\subsection{Scheme-theoretic reductions}\label{sec:lem-on-group-schemes}

We start this section with auxiliary results needed in the case $\op{char}K=p>0\) for the classification of conjugacy classes of morphisms \(\boldsymbol{\mu}_m \to \mathbf A\) that we will obtain in~\Cref{sec:mu-embeddings}.

We will use the following convention. We write in boldface for affine group schemes over \(K\). For an algebraic affine group scheme \(\mathbf A\), we write \(A:=\mathbf A(K)\) for the
associated affine algebraic group. If $\mathbf{A}$ is smooth, then it is determined by $A$, since \(K\) is assumed to be algebraically closed. 

Write \(m=m'p^k\) with \((m',p)=1\). Then we have $\Z_m\cong\Z_{m'}\times\Z_{p^k}$, 
and Cartier duality yields an isomorphism 
\(\bmu_m \cong \boldsymbol{\mu}_{m'}\times \boldsymbol{\mu}_{p^k}\)
of diagonalizable group schemes.

\begin{lemma}\label{lem:component-image-prime-to-p}
Let \(\eta\cl\boldsymbol{\mu}_m\to\mathbf A\) be a morphism of group schemes.
%to an algebraic affine group scheme. %Write \(m=m'p^k\), with \((m',p)=1\). 
Then the induced morphism \(\boldsymbol{\mu}_m\to\mathbf A/\mathbf A^\circ\) factors through the quotient map \(\boldsymbol{\mu}_m \to \boldsymbol{\mu}_{m'}\). In particular, its image has order prime to \(p\).
\end{lemma}

\begin{proof}
Let \(\pi\cl\mathbf A\longrightarrow\mathbf A/\mathbf A^\circ\) be the
quotient morphism, and set \(\overline\eta:=\pi\circ\eta\). Under the
decomposition \(\boldsymbol{\mu}_m\cong\boldsymbol{\mu}_{m'}\times\boldsymbol{\mu}_{p^k}\), 
let \(\iota_p\cl\boldsymbol{\mu}_{p^k}\longrightarrow\boldsymbol{\mu}_m\)
denote the inclusion of the second factor. Since
\(\boldsymbol{\mu}_{p^k}\) is connected,
we have \(\operatorname{\mathbf{Im}}(\overline\eta\circ\iota_p)
\subseteq(\mathbf A/\mathbf A^\circ)^\circ=\mathbf 1\).
Thus \(\overline\eta\) is trivial on \(\boldsymbol{\mu}_{p^k}\), 
and hence factors through the projection onto \(\boldsymbol{\mu}_{m'}\).
Since \(\boldsymbol{\mu}_{m'}\) has order prime to \(p\), so do its quotients.
\end{proof}

% \begin{remark}\label{rem:pnmidq-suffices}
% Thus, in the classification theorem of~\Cref{sec:mu-embeddings}, the outer component of a morphism \(\boldsymbol{\mu}_m\to\mathbf A\) is controlled by the prime-to-\(p\) factor \(\boldsymbol{\mu}_{m'}\). Consequently, the results of~\Cref{subsec:pinned-semidirect-product-case} apply to the resulting diagram automorphism \(\sigma\), and the only genuinely new positive-characteristic contribution comes from the factor \(\boldsymbol{\mu}_{p^k}\).
% \end{remark}

\begin{lemma}\label{lem:Q-in-torus-scheme}
Let \(\eta\cl\mathbf D\to \mathbf A\) be a morphism of group schemes where \(\mathbf{A}\) is smooth and \(\mathbf D\) is connected and diagonalizable. Then there exists a maximal torus \(\mathbf T\) of \(\mathbf A\) such that \(\op{\bf{Im}}\eta\subseteq \mathbf{T}\).
\end{lemma}

\begin{proof}
Set \(\mathbf Q:=\op{\bf{Im}}\eta\subseteq \mathbf A\). Since \(\mathbf D\) is diagonalizable, so is \(\mathbf Q\). Since \(\mathbf D\) is connected, we have \(\mathbf Q\subseteq \mathbf G:=\mathbf{A}^\circ\).

Consider the centralizer \(\mathbf H:=C_{\mathbf G}(\mathbf Q)\). 
Then \(\mathbf H\) is a smooth closed subgroupscheme of \(\mathbf G\). 
Let \(\mathbf T_0\) be a maximal torus of \(\mathbf H^\circ\), 
and set \(\mathbf C:=C_{\mathbf H^\circ}(\mathbf T_0)\). 
Since \(\mathbf H^\circ\) is connected and \(\mathbf T_0\) is a maximal torus, \(\mathbf C\) is connected and nilpotent (for example, apply~\cite[\S21.4]{Hum} to the groups of \(K\)-points), and we have a decomposition \(\mathbf C\cong \mathbf T_0\times \mathbf U\), where \(\mathbf U\) is the unipotent radical of \(\mathbf C\) (see e.g.~\cite[\S10.4]{Wat}). 
By construction, \(\mathbf Q\subseteq C_{\mathbf H}(\mathbf T_0)\). Since \(\mathbf Q\) is connected, it follows that \(\mathbf Q\subseteq C_{\mathbf H}(\mathbf T_0)^\circ=\mathbf C\). Since \(\mathbf U\) does not have nontrivial diagonalizable subgroupschemes (see e.g.~\cite[\S8.3,~Cor.~1]{Wat}), the composite \(\mathbf Q\to \mathbf C\to \mathbf U\) must be trivial, so \(\mathbf Q\subseteq \mathbf T_0\).

Finally, every torus of \(\mathbf A\) is contained in a maximal torus. Choose a maximal torus \(\mathbf T\) with \(\mathbf T_0\subseteq \mathbf T\). 
Then \(\op{\mathbf{Im}}\eta=\mathbf{Q}\subseteq\mathbf{T}\).
\end{proof}

\begin{lemma}\label{lem:conj-of-p-part}
Let \(\mathbf T\) be a maximal torus of $\mathbf{A}$, and let \(\eta,\tilde\eta\cl\mathbf{D}\to\mathbf T\) be morphisms of affine group schemes where 
\(\mathbf D\) is diagonalizable. 
If \(\eta\) and \(\tilde\eta\) are \(A\)-conjugate, then they are \(N_A(T)\)-conjugate.\footnote{In the case \(\mathbf{A}=\mathbf{Aut}\mathcal{A}\), this is a reformulation of \cite[Prop.~4.22]{Alb}.}
\end{lemma}

\begin{proof}
Let \(\mathbf Q:=\op{\bf{Im}}\eta\subseteq \mathbf T\) and
\(\widetilde{\mathbf{Q}}:=\op{\bf{Im}}\tilde{\eta}\subseteq \mathbf T\).
By hypothesis, there is \(g\in A\) such that \(\widetilde{\mathbf{Q}}=\Int(g)(\mathbf Q)\). 
Set \(\mathbf H:=C_{\mathbf A}(\mathbf Q)\) and \(\widetilde{\mathbf H}:=C_{\mathbf A}(\widetilde{\mathbf Q})\).  
Because \(\mathbf Q,\widetilde{\mathbf{Q}}\subseteq \mathbf T\) and \(\mathbf T\) is 
commutative, we have \(\mathbf T\subseteq \mathbf H\) and \(\mathbf T\subseteq \widetilde{\mathbf H}\). 
Since \(\mathbf T\) is maximal in \(\mathbf A\), it is also maximal in \(\mathbf H\) and  $\widetilde{\mathbf{H}}$.

Now \(\Int(g)\) is an automorphism of \(\mathbf A\), so it  
induces an isomorphism \(\mathbf H\xrightarrow{\sim}\widetilde{\mathbf H}\).
Since \(T\) is a maximal torus of \(H\), it follows that
\(\Int(g)(T)\) is a maximal torus of \(\widetilde{H}\).
Since any two maximal tori in \(\widetilde{H}\) are \( \widetilde H\)-conjugate, there exists \(h\in \widetilde H\) such that \(h(gTg^{-1})h^{-1}=T\). 
Set \(n:=hg\in A\), then \(n\in N_A(T)\). 
Since \(h\in \widetilde H\), conjugation by \(h\) acts trivially on \(\widetilde{\mathbf{Q}}\). 
Therefore, \(\Int(n)\circ\eta=\Int(h)\circ\Int(g)\circ\eta=\Int(h)\circ\tilde\eta=\tilde\eta\).
Thus, \(\eta\) and \(\tilde\eta\) are \(N_A(T)\)-conjugate.
\end{proof}

\subsection{Classification of conjugacy classes of morphisms}
\label{sec:mu-embeddings}

We first work in the general setup of \Cref{subsec:the-general-setup}. 
Thus \(A\) is a reductive algebraic group with identity component \(G\) and component group \(\Gamma_0=A/G\). 
%\(G:=A^\circ\) is semisimple, and \(C_A(G)\subseteq G\). 
%Recall that \(\Gamma_A=\operatorname{Im}(\bar\rho)\cong A/G\) and let \(\pi\colon A\to\Gamma_A\) be the quotient morphism introduced there.
Let \(\mathbf A\) be the smooth affine group scheme associated to \(A\),
and let \(\boldsymbol{\Gamma}_0\) be the constant group scheme associated to \(\Gamma_0\). 
Let \(\pi\colon\mathbf A\rightarrow\boldsymbol{\Gamma}_0\) be the quotient map.
The objective of this section is to classify, up to \(G\)-conjugacy, the morphisms of affine group schemes \(\eta\cl \boldsymbol{\mu}_m\to \mathbf A\), for any algebraically closed base field \(K\) of characteristic \(p\ge 0\). 

Write \(m=m'p^k\), with \((m',p)=1\) if \(p>0\), and put
\(m'=m\) if \(p=0\). Fix a primitive \(m'\)-th root of unity
\(\zeta_{m'}\in\boldsymbol{\mu}_{m'}(K)=\boldsymbol{\mu}_m(K)\). For each
\(\sigma\in\Gamma_0\), we classify those morphisms $\eta$ for which $\pi_K\circ\eta_K(\zeta_{m'})=\sigma$.

For $\eta\in \Hom(\boldsymbol{\mu}_m, \mathbf{A})$, we write \(\eta_{p'}\) and \(\eta_p\) for the restrictions of \(\eta\) to
\(\boldsymbol{\mu}_{m'}\) and \(\boldsymbol{\mu}_{p^k}\), respectively.
By \Cref{lem:component-image-prime-to-p}, the component morphism
\(\pi\circ\eta\) factors through \(\boldsymbol{\mu}_{m'}\). 
Hence \(\pi\circ\eta\) is determined by the image in \(\Gamma_0\) under \((\pi\circ\eta_{p'})_K\) of the fixed generator \(\zeta_{m'}\). Furthermore, this image has order dividing \(m'\).
% Set \(s:=\eta_K(\zeta_{m'})\in A\), \(\sigma:=\pi_K(s)\in\Gamma_0\), and \(q:=\ord(\sigma)\).
% Since \(s\) has finite order not divisible by \(p\), it is semisimple. 
% %(e.g.\ \cite[\S15]{Hum}). 
% Furthermore, \(q=\ord(\sigma)\mid \ord(s)\), and hence \(p\nmid q\). 

For \(\sigma\in\Gamma_0\), we define the \emph{\(\sigma\)-block} by
\[
\mathcal{H}_\sigma:=
\Bigl\{\eta\in \Hom(\boldsymbol{\mu}_m,\mathbf A)\ \Big|\ \pi_K(\eta_K(\zeta_{m'}))=\sigma\Bigr\}.
\]
Let \(\Gamma_{0,p'}:=\{\sigma\in \Gamma_0 \mid p\nmid \ord(\sigma)\}\subseteq\Gamma_0\) (not necessarily a subgroup).
Every \(\sigma\) arising from a morphism 
\(\boldsymbol{\mu}_m\to \mathbf A\) lies in \(\Gamma_{0,p'}\). 
Moreover, for each \(\sigma\in\Gamma_{0,p'}\) of order \(q\), the corresponding block  
can be nonempty only if \(q\mid m'\). 

It is clear from the definition that
\[
\Hom(\boldsymbol{\mu}_m,\mathbf{A})
\;=\;\bigsqcup_{\sigma\in\Gamma_{0,p'}} \mathcal{H}_\sigma,
\]
and each $\mathcal{H}_\sigma$ is stable under conjugation by \(G\), because \(\pi_K(G)=1\).

% \begin{proof}
% Let $\eta \in \Hom(\boldsymbol{\mu}_m, \mathbf{A})$, then
% $\eta\in \mathcal{H}_\sigma$ for a unique $\sigma\in\Gamma_{A,p'}$, proving the disjoint union.
% If $g\in G$, then $\pi_K(g)=1$, so for any $\eta$ we have
% \[\pi_K\bigl((\Int(g)\circ \eta)_K(\zeta_{m'})\bigr)
% =\pi_K\bigl(g\,\eta_K(\zeta_{m'})\,g^{-1}\bigr)
% =\pi_K\bigl(\eta_K(\zeta_{m'})\bigr).\]
% Thus $\mathcal{H}_\sigma$ is \(G\)-stable.
% \end{proof}

\begin{remark}\label{rem:xi-choice}
The block decomposition depends on the choice of the generator \(\zeta_{m'}\).  
Replacing \(\zeta_{m'}\) by \(\zeta_{m'}^u\) (with \(\gcd(u,m')=1\)) sends \(\mathcal{H}_\sigma\) to \(\mathcal H_{\sigma^u}\).  
This relabelling does not change the classification: the collections \(\{\mathcal{H}_\sigma\mid 
\sigma\in\Gamma_{0}\}\) and \(\{\mathcal H_{\sigma^u}\mid \sigma\in\Gamma_{0}\}\) are in bijective correspondence by precomposing with the power map $[u]:\boldsymbol{\mu}_m\to \boldsymbol{\mu}_m$, given on $R$-points by $x \mapsto x^u$.
\end{remark}

Recall the quasitorus \(S^{(\sigma)}\) and its component \(M^{(\sigma)}\) from \Cref{thm:sem-conj-outer-reorganized} and let \(\mathbf{S}^{(\sigma)}\) be the corresponding smooth closed subgroupscheme of \(\mathbf{A}\).

\begin{theorem}\label{thm:block-reduction-Weyl-orbits}
Fix \(\sigma\in\Gamma_{0,p'}\). Every \(\eta\in\mathcal H_\sigma\) is \(G\)-conjugate to a morphism \(\eta':\boldsymbol{\mu}_m\to\mathbf S^{(\sigma)}\) belonging to
\[
\Hom(\boldsymbol{\mu}_m,\mathbf S^{(\sigma)})_\sigma:=
\left\{\theta\in\Hom(\boldsymbol{\mu}_m,\mathbf S^{(\sigma)})
\mid\theta_K(\zeta_{m'})\in M^{(\sigma)}\right\}.
\]
Moreover, this induces a bijection between \(G\)-orbits on \(\mathcal H_\sigma\) and \(W^{(\sigma)}\)-orbits on \(\Hom(\boldsymbol{\mu}_m,\mathbf S^{(\sigma)})_\sigma\).
\end{theorem}

\begin{proof}
Let \(\eta\in\mathcal H_\sigma\), and set
\(s:=(\eta_{p'})_K(\zeta_{m'})\in A\).
By definition of \(\mathcal H_\sigma\), we have
\(\pi_K(s)=\sigma\), and hence \(s\in A^{(\sigma)}\). Moreover, the
order of \(s\) divides \(m'\), so \(s\) is semisimple. By
\Cref{thm:sem-conj-outer-reorganized}(4), after conjugating \(\eta\)
by an element of \(G\), we may assume that
\(s\in M^{(\sigma)}\subseteq S^{(\sigma)}\).
By \Cref{lem:component-image-prime-to-p}, the composition
\(\pi\circ\eta_p\) is trivial, and hence
\(\op{\mathbf{Im}}(\eta_p)\subseteq\mathbf G\).
Since
\(\boldsymbol{\mu}_m
\cong \boldsymbol{\mu}_{m'}\times\boldsymbol{\mu}_{p^k}\)
is commutative, the images of \(\eta_{p'}\) and \(\eta_p\) commute.
Therefore \(\op{\mathbf{Im}}(\eta_p)\) centralizes \(s\), and hence
\(\op{\mathbf{Im}}(\eta_p)\subseteq\mathbf G^s\).
Again by~\Cref{lem:component-image-prime-to-p}, \(\op{\mathbf{Im}}(\eta_{p})\) is contained in \((\mathbf G^s)^\circ\). Since the element \(s\) is semisimple, the subgroupscheme \(\mathbf G^s\) is smooth. By
\Cref{lem:Q-in-torus-scheme}, there exists a maximal torus
\(\mathbf R\subseteq\mathbf G^s\) such that
\(\op{\mathbf{Im}}(\eta_p)\subseteq\mathbf R\).
By \Cref{thm:sem-conj-outer-reorganized}(1), we may assume without loss of generality that 
\(s\) is the semisimple element used to define the objects attached to the component \(A^{(\sigma)}\).
Since \(T^{(\sigma)}\) and \(R\) are maximal tori of \(G^s\), conjugating by a suitable element of \(G^s\), we may further assume that \(R=T^{(\sigma)}\), and hence \(\op{\mathbf{Im}}(\eta_p)\subseteq\mathbf{S}^{(\sigma)}\). 
% Hence \(\mathbf R\) and \(\mathbf T^{(\sigma)}\) are maximal tori of
% \(\mathbf G^s\), and are therefore conjugate by an element of
% \((G^s)^\circ\): \(cRc^{-1}=T^{(\sigma)}\) for some \(c\in(G^s)^\circ\).
% Conjugating \(\eta\) by \(c\) fixes \(s\) and sends
% \(\op{\mathbf{Im}}(\eta_p)\) into \(\mathbf T^{(\sigma)}\). 
Since \(s\in S^{(\sigma)}\) and \(\op{\mathbf{Im}}(\eta_{p'})\) is the smooth group 
scheme corresponding to  \(\langle s\rangle\subseteq S^{(\sigma)}\), it follows that \(\op{\mathbf{Im}}(\eta_{p'}) \subset\mathbf S^{(\sigma)}\), too. 
Therefore, after conjugation, $\eta\cl \boldsymbol{\mu}_m\to \mathbf{S}^{(\sigma)}$.
Since \(\eta_K(\zeta_{m'})=s\in M^{(\sigma)}\), the resulting morphism belongs to
\(\Hom(\boldsymbol{\mu}_m,\mathbf S^{(\sigma)})_\sigma\).

It remains to determine the \(G\)-conjugacy relation on
\(\Hom(\boldsymbol{\mu}_m,\mathbf S^{(\sigma)})_\sigma\).
Let \(\eta,\tilde\eta \in \Hom(\boldsymbol{\mu}_m,\mathbf S^{(\sigma)})_\sigma\),
and suppose that \(\tilde\eta=\Int(g)\circ\eta\)
for some \(g\in G\). Set
\(s:=(\eta_{p'})_K(\zeta_{m'})\) and
\(\tilde s:=(\tilde\eta_{p'})_K(\zeta_{m'})\).
Then \(\tilde s=gsg^{-1}\).
Since both morphisms lie in \(\Hom(\boldsymbol{\mu}_m,\mathbf S^{(\sigma)})_\sigma\), we have \(s,\tilde s\in M^{(\sigma)}\).
By \Cref{thm:sem-conj-outer-reorganized}(5), there exists \(n\in N_G(M^{(\sigma)})\)
such that \(n\tilde s n^{-1}=s\).
Replacing \(\tilde\eta\) by \(\tilde{\eta}^\prime:=\Int(n)\circ\tilde\eta\) and \(g\) by \(g_1:=ng\), we have \(\tilde{\eta}^\prime=\Int(g_1)\circ\eta\), \(g_1\in G^s\). 
Moreover, \(\tilde\eta'_{p'}=\eta_{p'}\),
because the two morphisms from \(\boldsymbol{\mu}_{m'}\) have the same
value \(s\) at the generator \(\zeta_{m'}\).
Both \(\eta_p\) and \(\tilde\eta'_p\) have image in
\(\mathbf T^{(\sigma)}\subseteq\mathbf G^s\), and they are conjugate by
\(g_1\in G^s\). Applying \Cref{lem:conj-of-p-part} to the (smooth) affine
group scheme \(\mathbf G^s\), with maximal torus
\(\mathbf T^{(\sigma)}\), there exists
\(\tilde{n}\in N_{G^s}(T^{(\sigma)})\)
such that \(\tilde\eta'_p=\Int(\tilde{n})\circ\eta_p\).
Since also \(\tilde\eta'_{p'}=\eta_{p'}=\Int(\tilde{n})\circ\eta_{p'}\),
the product decomposition of \(\boldsymbol{\mu}_m\) gives \(\tilde\eta'=\Int(\tilde{n})\circ\eta\).
The element \(\tilde{n}\) centralizes \(s\) and normalizes \(T^{(\sigma)}\).
Consequently, it normalizes
\(sT^{(\sigma)}=M^{(\sigma)}\),
so \(\tilde{n}\in N_G(M^{(\sigma)})\). Therefore
\(\tilde\eta=\Int(n^{-1}\tilde{n})\circ\eta,
\ n^{-1}\tilde{n}\in N_G(M^{(\sigma)})\).
Since \(W^{(\sigma)}=N_G(M^{(\sigma)})/T^{(\sigma)}\), the result follows.
\end{proof}

\subsection{The pinned semidirect-product case}\label{sec:mu-embeddings-semidirect}
We now specialize to the setup of
\Cref{subsec:pinned-semidirect-product-case}. Thus \(G\) is semisimple and \(A=G\rtimes\Gamma_0\), where \(\Gamma_0\le\Gamma_G\) acts through the fixed pinning. 
For \(\sigma\in\Gamma_0\) of order \(q\), put \(S:=(T^\sigma)^\circ\),
\(S^{(\sigma)}=S \times \langle \sigma\rangle\), \(M^{(\sigma)}=\sigma S\), and \(W^{(\sigma)}=N_G(S^{(\sigma)})/S\).
We also retain the identification
\(\frk X(S^{(\sigma)})\cong X\oplus\mathbb Z_q\), where  \(X:=\frk X(S)\) and \(\frk X(\langle\sigma\rangle)\) is identified with \(\Z_q\) via \(\zeta_q:=\zeta_{m'}^{m/q}\) (recall that \(\mathcal{H}_\sigma\) is empty unless \(q\mid m'\)).
We call \(\sigma\) the \emph{outer type} of \(\eta\in\mathcal{H}_\sigma\). 
If \(\sigma=1\), then \(\eta\) is said to be of \emph{inner type}. 

\Cref{thm:block-reduction-Weyl-orbits} reduces the classification problem for morphisms in \(\mathcal{H}_\sigma\) to the action of
\(W^{(\sigma)}\) on morphisms \(\theta\cl\boldsymbol{\mu}_m\to\mathbf{S}^{(\sigma)}\), where  
\(\mathbf{S}^{(\sigma)}\) is the smooth subgroupscheme of \(\mathbf{A}\) corresponding to \(S^{(\sigma)}\). Since in our case \(S^{(\sigma)}\cap A^{(\sigma)}=M^{(\sigma)}\), the condition \(\theta_K(\zeta_{m'})\in M^{(\sigma)}\) is equivalent to \(\theta\) having outer 
type \(\sigma\). It remains to translate this orbit problem into the character
group of \(S^{(\sigma)}\) to make it independent of \(K\). 
The passage from \(W^{(\sigma)}\) to the ``lattice model'' \(W^{(\sigma)}_{\mathrm{lat}}\) acting on \(X\oplus\mathbb Z_q\) is exactly the content of~\Cref{prop:full-lattice-sigma-weyl-action}.

Since \(\boldsymbol{\mu}_m\) and \(\mathbf S^{(\sigma)}\) are diagonalizable
group schemes, Cartier duality gives a canonical bijection
\(\Hom(\boldsymbol{\mu}_m,\mathbf S^{(\sigma)})
\xleftrightarrow{\ \sim\ }
\Hom(\frk X(S^{(\sigma)}),\bb{Z}_m)\). To write this in terms of \(X\oplus\mathbb Z_q\), let \(\varepsilon:=(0,\bar 1)\in X\oplus\bb{Z}_q\)
and define
\[
\Hom(X\oplus\bb{Z}_q,\bb{Z}_m)_\sigma:=\left\{\varrho\in\Hom(X\oplus\bb{Z}_q,\bb{Z}_m)\ \middle|\
\zeta_{m'}^{\,\varrho(\varepsilon)}=\zeta_q
\right\}.
\]
Equivalently, writing \(\varrho(\varepsilon)=\frac{m}{q}\bar\varrho\) with \(\bar{\varrho}\in\Z_q\), 
the condition becomes \(\bar{\varrho}=\bar 1\) in \(\Z_q\).
%\(\varrho(\varepsilon)\) lies in the coset \(\overline{m/q}+m'\bb{Z}_m\subset\bb{Z}_m\).
We will see that this is precisely the subset corresponding to morphisms
\(\boldsymbol{\mu}_m\to\mathbf S^{(\sigma)}\) whose outer type is \(\sigma\).
Note that the element \(\varepsilon\) is fixed by the \(W^{(\sigma)}_{\mathrm{lat}}\)-action on
\(X\oplus\bb{Z}_q\) (see \Cref{prop:full-lattice-sigma-weyl-action}), so this subset is
\(W^{(\sigma)}_{\mathrm{lat}}\)-stable.

\begin{theorem}\label{thm:final-classification-mu-embeddings}
Fix \(\sigma\in\Gamma_{0,p'}\) such that \(q:=\ord(\sigma)\) divides $m$ (and hence $m'$). After the fixed choice of \(\zeta_{m'}\in\boldsymbol{\mu}_{m}(K)\) and hence the primitive $q$-th root of unity \(\zeta_q=\zeta_{m'}^{m/q}\), the set of \(G\)-conjugacy classes in \(\mathcal H_\sigma\) is in bijection with the set of \(W^{(\sigma)}_{\mathrm{lat}}\)-orbits on
\(\Hom(X\oplus\bb{Z}_q,\bb{Z}_m)_\sigma \)
with respect to the action defined in~\Cref{prop:full-lattice-sigma-weyl-action}.
\end{theorem}

\begin{proof}
By \Cref{thm:block-reduction-Weyl-orbits}, the set of \(G\)-conjugacy classes in \(\mathcal H_\sigma\) is identified with the set of \(W^{(\sigma)}\)-orbits on
\(\Hom(\boldsymbol{\mu}_m,\mathbf S^{(\sigma)})_\sigma\). 
Let \(\theta\cl\boldsymbol{\mu}_m\to\mathbf S^{(\sigma)}\) correspond to \(\varrho\in\Hom(\frk{X}(\mathbf{S}^{(\sigma)}),\Z_m)\). This means that, for any \(\chi\in\frk{X}(\mathbf{S}^{(\sigma)})\), any commutative associative unital \(K\)-algebra \(\mathcal R\), and any \(r\in\boldsymbol{\mu}_m(\mathcal R)\), we have \(\chi_{\mathcal R}(\theta_{\mathcal R}(r))=r^{\,\varrho(\chi)}\). 
Taking \(\chi=\varepsilon\), \(\mathcal R=K\), and \(r=\zeta_{m'}\), we obtain 
\(\varepsilon_K(\theta_K(\zeta_{m'}))=\zeta_{m'}^{\,\varrho(\varepsilon)}\). 
Since \(\varepsilon_K\) is trivial on \(S\) and sends \(\sigma\) to \(\zeta_q\), the condition \(\theta_K(\zeta_{m'})\in M^{(\sigma)}\) is equivalent to
\(\zeta_{m'}^{\,\varrho(\varepsilon)}=\zeta_q\). 
Hence Cartier duality restricts to a bijection
\[
\Hom(\boldsymbol{\mu}_m,\mathbf S^{(\sigma)})_\sigma
\xleftrightarrow{\ \sim\ }
\Hom(X\oplus\bb{Z}_q,\bb{Z}_m)_\sigma .\]
This bijection is compatible with the $W^{(\sigma)}$-actions. 
Indeed, the left action of \(N_G(S^{(\sigma)})\) on morphisms \(\theta\in\Hom(\boldsymbol{\mu}_m,\mathbf S^{(\sigma)})\) is by conjugation:
\( n \cdot \theta:=\Int(n)\circ\theta\) for all $n\in N_G(S^{(\sigma)})$. On the character group
\(\frk X(S^{(\sigma)})\cong \frk X(\mathbf S^{(\sigma)})\), the left action corresponding to the conjugation action on $\mathbf S^{(\sigma)}$ is
the usual action given by 
\((n \chi)_{\mathcal R}(t)=\chi_{\mathcal R}(n^{-1}tn)\) for all $\chi \in \frk X(\mathbf{S}^{(\sigma)})$, \(n \in N_G(S^{(\sigma)})\), and \(t \in\mathbf{S}^{(\sigma)}(\mathcal R)\).
Hence, if \(\theta\) corresponds to
\(\varrho\in\Hom(\frk X(S^{(\sigma)}),\bb{Z}_m)\), then
\(\Int(n)\circ\theta\) corresponds to
\(\chi\mapsto \varrho(n^{-1}\chi)\).
Since \(S\) centralizes \(S^{(\sigma)}\), the action factors through
\(W^{(\sigma)}\).

By~\Cref{prop:full-lattice-sigma-weyl-action}, after identifying
\(\frk X(S^{(\sigma)})\) with \(X\oplus\bb{Z}_q\), and \(W^{(\sigma)}\) with \(W^{(\sigma)}_{\mathrm{lat}}\), the
\(W^{(\sigma)}\)-action on $\frk X(S^{(\sigma)})$ corresponds to the
\(W^{(\sigma)}_{\mathrm{lat}}\)-action on $X \oplus\Z_q$. Therefore the
\(W^{(\sigma)}\)-orbits in the block
\(\Hom(\boldsymbol{\mu}_m,\mathbf S^{(\sigma)})_\sigma\) correspond exactly to
the \(W^{(\sigma)}_{\mathrm{lat}}\)-orbits in
\(\Hom(X\oplus\bb{Z}_q,\bb{Z}_m)_\sigma \).
\end{proof}

\begin{corollary}\label{rem:field-independence-mu}
For \(A=G\rtimes\Gamma_0\), the classification of morphisms \(\boldsymbol{\mu}_m\to\mathbf{A}\) up to G-conjugation is independent of the algebraically closed field \(K\), except for the condition that only morphisms of outer type \(\sigma\) whose order is not divisible by \(p\) can occur when \(\operatorname{char}K=p\).
\end{corollary}

\begin{proof}
After fixing a generator of \(\boldsymbol{\mu}_m(K)\), \Cref{thm:final-classification-mu-embeddings} identifies the classification of \(G\)-orbits in each block \(\mathcal H_\sigma\) with that of $W^{(\sigma)}_{\mathrm{lat}}$-orbits in the set \(\Hom(X\oplus\bb{Z}_q,\bb{Z}_m)_\sigma\), which depends only on the root datum of \(G\) and the diagram automorphism \(\sigma\in\Gamma_{0,p'}\). If \(p \mid \operatorname{ord}(\sigma)\) then 
\(\mathcal H_\sigma=\varnothing\).    
\end{proof}

\begin{remark}\label{rem:action_real_torus}
It is convenient to embed all groups \(\Z_n\) into the circle \(\mathbb{R}/\Z\) via \(k+n\Z\mapsto\frac{k}{n}+\Z\). Writing \(\varrho\in\Hom(X\oplus\Z_q,\Z_m)\) as \((\rho,\bar{\varrho})\), 
where \(\rho\) is the restriction of \(\varrho\) to \(X\) and \(\bar{\varrho}\in\Z_q\) is defined by 
\(\varrho(\varepsilon)=\frac{m}{q}\bar{\varrho}\), we then embed \(\Hom(X\oplus\Z_q,\Z_m)_\sigma\) into  
the set \(\Hom(X,\mathbb{R}/\Z)\times\{\frac{1}{q}+\Z\}\). This latter is canonically identified with the 
real torus \(E/Y\) where \(Y:=\frk{Y}(S)\) and \(E:=Y\otimes_\Z\mathbb{R}\). Under these identifications, 
the action of \(W^{(\sigma)}_{\mathrm{lat}}\) given by Equation~\eqref{eq:full-gen-weyl-group-action} 
corresponds to the action \((f,w)\cdot x=wx-\frac{1}{q}\bar{\nu}(f)\) for \(x\in E/Y\).
\end{remark}

The next result relates \(G\)-conjugacy to \(A\)-conjugacy.

\begin{proposition}\label{prop:A-vs-G-fixed-outer-block}
Let \(\sigma\in\Gamma_{0,p'}\). If \(a\in A\) has image
\(\gamma\in\Gamma_0\), then
\(\Int(a)(\mathcal H_\sigma)=\mathcal H_{\gamma\sigma\gamma^{-1}}\).
Consequently, the stabilizer of \(\mathcal H_\sigma\) in \(\Gamma_0\) is
\(C_{\Gamma_0}(\sigma)\). The induced action of \(C_{\Gamma_0}(\sigma)\) on
the set of \(G\)-conjugacy classes in \(\mathcal H_\sigma\) factors through
\(C_{\Gamma_0}(\sigma)/\langle\sigma\rangle\).
In particular, if \(C_{\Gamma_0}(\sigma)=\langle\sigma\rangle\), then two
morphisms in \(\mathcal H_\sigma\) are \(A\)-conjugate if and only if they are
\(G\)-conjugate.
\end{proposition}

\begin{proof}
Let \(a\in A\), and write \(\gamma:=\pi_K(a)\). If \(\eta\in\mathcal H_\sigma\), then
\[\pi_K\bigl((\Int(a)\circ\eta)_K(\zeta_{m'})\bigr)=\gamma\,\pi_K(\eta_K(\zeta_{m'}))\,\gamma^{-1}=\gamma\sigma\gamma^{-1}.\]
Thus \(\Int(a)(\mathcal H_\sigma)=\mathcal H_{\gamma\sigma\gamma^{-1}}\).
It follows that an element of \(A\) preserving the block \(\mathcal H_\sigma\)
must have image in \(C_{\Gamma_0}(\sigma)\).

Let \(\eta\in\mathcal H_\sigma\). By~\Cref{thm:block-reduction-Weyl-orbits}, there exist \(g\in G\) and
\(\theta\in\Hom(\boldsymbol{\mu}_m,\mathbf S^{(\sigma)})_\sigma\) such that
\(\eta=\Int(g)\circ\theta\). Since
\(S^{(\sigma)}=\langle S,\sigma\rangle\), the element \(\sigma\)
centralizes \(S^{(\sigma)}\). Hence \(\Int(\sigma)\circ\theta=\theta\). Therefore
\(\Int(\sigma)\circ\eta=
\Int(\sigma g\sigma^{-1})\circ\theta\),
which is \(G\)-conjugate to \(\theta\), and hence \(G\)-conjugate to \(\eta\).
% Thus the residual action of \(C_{\Gamma_0}(\sigma)\) on the set of \(G\)-orbits
% factors through \(C_{\Gamma_0}(\sigma)/\langle\sigma\rangle\).
% 
% If \(C_{\Gamma_0}(\sigma)=\langle\sigma\rangle\), then every \(A\)-conjugacy between
% two \(G\)-orbits in \(\mathcal H_\sigma\) is induced by an element
% of \(\langle\sigma\rangle\), which acts trivially by the preceding paragraph.
% Hence \(A\)-conjugacy and \(G\)-conjugacy coincide inside \(\mathcal H_\sigma\).
\end{proof}

\begin{remark}\label{rem:irreducible-centralizer-condition}
Assume that the root system \(\Phi\) of \(G\) is irreducible. Then, for every
nontrivial \(\sigma\in\Gamma_0\), one has \(C_{\Gamma_0}(\sigma)=\langle\sigma\rangle\).
% Indeed, the group \(\Gamma_0\) embeds into
% \(\Gamma_G\subseteq\Gamma:=\Aut\Pi\), and
% \(C_{\Gamma_0}(\sigma)=\Gamma_0\cap C_{\Gamma}(\sigma)\).
% For an irreducible Dynkin diagram, the group \(\Gamma\) is either
% trivial, cyclic of order \(2\), or isomorphic to \(S_3\) in type
% \(D_4\). In the latter case, the centralizer in \(S_3\) of a
% transposition is the subgroup generated by that transposition, while
% the centralizer of a \(3\)-cycle is the subgroup generated by that
% \(3\)-cycle. Hence
% \(C_{\Gamma}(\sigma)=\langle\sigma\rangle\)
% for every nontrivial \(\sigma\in\Gamma\). Since
% \(\langle\sigma\rangle\subseteq\Gamma_0\), it follows that 
% \(C_{\Gamma_0}(\sigma)=\langle\sigma\rangle\).
Consequently, for every nontrivial outer type in the irreducible case, the \(G\)-conjugacy classification within the block \(\mathcal H_\sigma\) already coincides with the
\(A\)-conjugacy classification within that block.
\end{remark}

\section{Applications to classical simple Lie algebras}\label{sec:classical-simple-lie-algebras}

We now apply the general classification results of~\Cref{sec:mu-embeddings-semidirect} to the automorphism groups of classical simple Lie algebras in arbitrary characteristic, which will yield a classification of gradings by cyclic groups on these algebras. The additional input needed for this is the smoothness of the automorphism group schemes.

\subsection{Classical simple Lie algebras in positive characteristic}\label{subsec:classical-small-char-applications}
In characteristic \(0\), the passage from an irreducible root system to the corresponding simple Lie algebra is uniform. 
In positive characteristic, however, several exceptional phenomena may occur. 
The Chevalley Lie algebra attached to a classical root system may fail to be simple, and even when the associated simple Lie algebra \(\g\) exists, the automorphism group scheme \(\op{\mathbf{Aut}}\g\) can have type different from \(\g\) or may a priori fail to be smooth (see \cite{Rob}, also \cite[\S0.13]{Hu} and \cite[Chap. 3, \S3]{Alb}). Thus, before applying the general classification results, one must first identify the correct simple Lie algebra and verify the smoothness of its automorphism group scheme.

The ideal structure of the tangent Lie algebras of simple algebraic groups depends on the isogeny type, but is well understood (see \cite{hogeweij,hummod,hiss}). 
For our purposes, it is sufficient to consider the Chevalley Lie algebras \(\tilde{\g}\) and their central quotients $\g=\tilde{\g}/Z(\tilde{\g})$ introduced in~\Cref{sec:classical}. 
The former arise from the simply connected groups: \(\tilde{\g}=\Lie(G_{\mathrm{sc}})\), 
and the central quotients are simple with the following exceptions: in $\op{char}K=3$, type $G_2$ and, in $\op{char}K=2$, \(A_1\) and all non-simply-laced types are excluded (type $G_2$ in characteristic $2$ gives a simple Lie algebra, but it is isomorphic to the one coming from \(A_3\)). Outside of these cases, the simple quotient \(\g\) can be alternatively described as the derived algebra of \( \Lie(G_{\mathrm{ad}})\). We include a proof for completeness.

\begin{lemma}\label{lem:der-alg-G-ad}
\([\Lie(G_{\mathrm{ad}}),\Lie(G_{\mathrm{ad}})]
=\ad\tilde\g\).   
\end{lemma}

\begin{proof}
Let \(\widetilde G=G_{\mathrm{sc}}\) and \(G=\Ad \widetilde G \subseteq \Aut \tilde \g\), which is of adjoint type. 
The differential of \(\Ad\cl\widetilde{G}\to G\) is the map
\(\ad\cl\tilde\g=\Lie(\widetilde{G})\to\Lie(G)\subseteq\Der\tilde\g\). Note that the image \(\ad\tilde\g\) is an ideal of \(\Lie(G)\), since it is an ideal of \(\Der\tilde\g\). 
%Indeed, for \(D\in\Lie(G)\subseteq\Der\tilde\g\) and \(x\in\tilde\g\), one has \([D,\ad x]=\ad(D(x))\in\ad\tilde\g\).

Let \(\widetilde T\) be a maximal torus of \(\widetilde G\) and let \(T=\Ad\widetilde{T}\). 
Consider the root space decomposition of \(\tilde \g\) and \(\Lie(G)\) with respect to these tori. 
For any \(x \in \tilde{\g}_\alpha\) and \(t \in \widetilde T\), we have \((\Ad t) (\ad x) (\Ad t)^{-1}=\ad((\Ad t)x)=\alpha(t)\ad x\), so \(\ad x \in \Lie(G)_\alpha\).  Therefore, every root space of
\(\Lie(G)\) is contained in
\(\ad\tilde\g\). Hence
\(\Lie(G)/\ad\tilde\g\) is a quotient of the abelian Lie algebra
\(\Lie(T)\), and is therefore abelian. This gives \([\Lie(G),\Lie(G)]\subseteq \ad\tilde\g\).
On the other hand, since
\(\ad\tilde\g\cong\g\) is simple, we have
\(\ad\tilde\g=[\ad\tilde\g,\ad\tilde\g]
\subseteq [\Lie(G),\Lie(G)]\).
\end{proof}

\begin{remark}\label{rem:gtilde_perfect}
\(\tilde{\g}\) is perfect. Indeed, \([\g,\g]=\g\) implies \(\tilde\g=[\tilde\g,\tilde\g]+Z(\tilde\g)\), but the relation \([x_\alpha,x_{-\alpha}]=h_\alpha\) implies that \(\tilde\h\subset[\tilde\g,\tilde\g]\).
\end{remark}

It remains to verify smoothness of the automorphism group scheme of these classical simple Lie algebras. We use the standard differential criterion: an algebraic affine group scheme \(\mathbf G\) over \(K\) is smooth if and only if \(\dim\Lie(\mathbf G)=\dim\mathbf G\). For \(\mathbf G=\op{\mathbf{Aut}}\g\), this amounts to checking that \(\dim\Der\g=\dim\Aut\g\).  

\begin{lemma}\label{lem:faithful-restriction-to-ideal}
Let \(\frk l\) be a finite-dimensional Lie algebra over \(K\), and let
\(\frk m\subseteq\frk l\) be an ideal such that
\(C_{\frk l}(\frk m)=0\). Assume that \(\frk m\) is \((\op{\mathbf{Aut}}\frk l)\)-invariant, i.e., for every commutative associative unital \(K\)-algebra \(\mathcal R\), the ideal
\(\frk m_{\mathcal R}:=\frk m\otimes_K\mathcal R\) is preserved by every automorphism of
\(\frk l_R:=\frk l\otimes_K\mathcal R\). Then restriction defines a closed embedding of affine
group schemes \(\operatorname{res}_{\frk m}\colon
\op{\mathbf{Aut}}\frk l\rightarrow\op{\mathbf{Aut}}\frk m\).
% This morphism is a monomorphism, and its differential
% \(d\operatorname{res}_{\frk i}\colon
% \Der\frk l\rightarrow\Der\frk i\),
% \(D\mapsto D|_{\frk i}\), is injective.
\end{lemma}

\begin{proof}
By hypothesis, the adjoint representation \(\frk l\to\End_K(\frk m)\) is injective. 
Since \(K\) is a field, it remains injective after tensoring with any \(\mathcal R\), i.e.,
\(C_{\frk l_{\mathcal R}}(\frk m_{\mathcal R})=0\) for every \(\mathcal R\).

Now let \(\theta\in(\op{\mathbf{Aut}}\frk l)(\mathcal R)\) lie in the kernel of 
\((\operatorname{res}_{\frk i})_{\mathcal R}\), i.e.,  \(\theta\) acts as the identity on 
\(\frk m_{\mathcal R}\). 
Then, for \(x\in\frk l_{\mathcal R}\) and \(y\in\frk m_{\mathcal R}\), we have 
\([\theta(x)-x,y]=[\theta(x),\theta(y)]-[x,y]=\theta([x,y])-[x,y]=0\),
because \([x,y]\in\frk m_{\mathcal R}\).
Therefore \(\theta(x)-x\in C_{\frk l_{\mathcal R}}(\frk m_{\mathcal R})=0\), and hence
\(\theta=\operatorname{id}_{\frk l_{\mathcal R}}\). 
%
% The hypothesis also implies that every derivation of \(\frk l\) preserves
% \(\frk i\). Indeed, for \(D\in\Der\frk l\), the automorphism \(1+\varepsilon D\) of 
% \(\frk l\otimes_KK[\varepsilon]/(\varepsilon^2)\) preserves
% \(\frk i\otimes_KK[\varepsilon]/(\varepsilon^2)\), and hence 
% \(D(\frk i)\subseteq\frk i\).
% Now let \(D\in\Der\frk l\) satisfy \(D|_{\frk i}=0\). For \(x\in\frk l\) and
% \(y\in\frk i\), the derivation identity gives
% \([D(x),y]=D([x,y])-[x,D(y)]=0\), since \([x,y]\in\frk i\). Hence \(D(x)\in C_{\frk l}(\frk i)=0\) for every \(x\in\frk l\), so \(D=0\). Therefore \(d\operatorname{res}_{\frk i}\) is injective.
\end{proof}

\begin{lemma}\label{lem:faithful-pi}
Let \(\frk l\) be a finite-dimensional Lie algebra over \(K\). Then \(\frk z=Z(\frk l)\) is \((\op{\mathbf{Aut}}\frk l)\)-invariant. If \(\frk l\) is perfect, then the passage to the quotient defines a closed embedding of affine group schemes \(\pi\colon
\op{\mathbf{Aut}}\frk l\rightarrow\op{\mathbf{Aut}}(\frk l/\frk z)\). 
\end{lemma}

\begin{proof}
The invariance follows from the fact that \(\frk z\) is the kernel of the adjoint representation \(\frk l\to\End_K(\frk l)\). If \(\theta\in(\op{\mathbf{Aut}}\frk l)(\mathcal R)\) lies in the kernel of \(\pi_{\mathcal R}\), i.e.,  \(\theta(x)-x\in\frk z_{\mathcal R}\) for all \(x\in\frk l_{\mathcal R}\), then, for any \(x,y\in\frk l_{\mathcal R}\), we have 
\(\theta([x,y])=[\theta(x),\theta(y)]=[x,y]\). Since \(\frk l_{\mathcal R}\) is the \(\mathcal R\)-span of commutators, we conclude that \(\theta=\operatorname{id}_{\frk l_{\mathcal R}}\).
\end{proof}

\subsubsection{Characteristic \(p\neq2\).}
In this case, smoothness of the automorphism group schemes \(\op{\mathbf{Aut}}\g\) is discussed in~\cite[Chap.~3]{Alb}, where it is proved that, under two additional hypotheses, every derivation is inner (see \cite[Thm.~3.2]{Alb}). One of these hypotheses fails for the Lie algebra of type \(E_6\) in characteristic \(3\) as well as for $A_n$ if $p$ divides $n+1$, and the argument does not apply in characteristic \(2\). We therefore retain only one of their two hypotheses, namely, that distinct roots induce distinct \(\h\)-weights. This gives a uniform description of \(\Der\g\) sufficient for the smoothness argument. The exceptional case \(A_2\) in characteristic \(3\) does not satisfy this hypothesis and is therefore excluded from this uniform argument. However, it is shown in~\cite{Alb} that the automorphism group scheme is smooth in this case as well (of type $G_2$).

Let $\widetilde G$ be a simply connected semisimple group with $\mathscr{L}(\widetilde G)=\tilde\g$, and let \(\widetilde{T}\subset \widetilde{G}\) be the maximal torus that induces the root space decomposition in the definition of $\tilde\g$ (see~\Cref{sec:classical}), so $\tilde\h=\Lie(\widetilde T)$. We will consider the corresponding (smooth) affine group schemes $\widetilde{\mathbf G}$ and $\widetilde{\mathbf T}$, and  
denote by $\mathbf{G}=\Ad \widetilde{\mathbf{G}} \subseteq \op{\mathbf{Aut}}\tilde\g$ and \(\mathbf T:=\Ad\widetilde{\mathbf{T}}\subset \op{\mathbf{Aut}}\tilde\g\) their image under the adjoint representation. 
Let \(\pi\cl\op{\mathbf{Aut}}\tilde\g\rightarrow\op{\mathbf{Aut}}\g\) be the closed embedding as in \Cref{lem:faithful-pi} (see also \Cref{rem:gtilde_perfect}). 
% and put \(\overline{\mathbf{T}}:=\pi(\mathbf T)\subset \op{\mathbf{Aut}}\g\). 
% Since \(\widetilde{\mathbf T}\) is a torus and both \(\Ad\) and \(\pi\) are morphisms of affine group schemes,  \(\mathbf T\) and \(\overline{\mathbf T}\) are also tori. 
We denote \(\tilde{\mathfrak t}:=\mathscr L( \mathbf T)\subset\Der\tilde{\g}\) and  
\(\ft:=d\pi(\tilde{\mathfrak t})\subseteq 
% \mathscr{L}(\overline{\mathbf T})\subseteq 
% \mathscr{L}(\op{\mathbf{Aut}}\g)= 
\Der\g\).

The action of \(\ft\) on \(\g\) is easy to describe explicitly. The adjoint representation of $\widetilde T$ gives the root-space decomposition
$\tilde{\g}=\tilde{\mathfrak{h}} \oplus \bigoplus_{\alpha \in \Phi} \tilde{\g}_\alpha$,
where $\tilde{\mathfrak{h}}= \mathscr{L}(\widetilde{T})$ and $\tilde{\g}_\alpha=K x_\alpha$ is the weight space of weight $\alpha$. 
The character group of \(T\) is naturally identified with the root lattice \(\Lambda_r\). 
Now, for any abelian group \(M\), we have \(\mathscr{L}(M^D)\cong \Hom_{\mathbb Z}(M,K)\). In particular, we get \(\tilde{\ft}\cong\Hom_{\mathbb Z}(\Lambda_r,K)\). 
Under this identification, an element \(u\in \tilde{\ft}\) acts trivially on \(\tilde\h\) and acts on each root space \(\tilde\g_\alpha=Kx_\alpha\) by the scalar \(u(\alpha)\). 
Being an action by derivations, it preserves the center and therefore descends to the quotient \(\g=\tilde\g/Z(\tilde\g)\): an element $u$ acts trivially on \(\h:=\tilde{\mathfrak{h}} / Z(\tilde{\g})\) and acts on each root space \(\g_{\alpha}\) by the scalar \(u(\alpha)\).

We can now follow the approach of \cite{Alb} to describe the full derivation algebra of \(\g\) under the hypothesis that distinct roots give distinct \(\h\)-weights.

\begin{theorem}\label{thm:derivations-of-g}
Let \(\g=\tilde\g/Z(\tilde\g)\) be a classical simple Lie algebra over a field \(K\). Assume that
\(\bar\alpha\neq\bar\beta\) for all distinct roots
\(\alpha,\beta\in\Phi\) (hence $\op{char}K \neq 2$). Let \(\ft\) be the image of \(\Hom_{\mathbb Z}(\Lambda_r,K)\) in \(\Der\g\) given by its natural action. Then
\[\Der\g=\ad\g+\ft.\]
\end{theorem}

\begin{proof}
% By the hypothesis \(\bar\alpha\neq\bar\beta\) for \(\alpha\neq\beta\), the
% \(\h\)-eigenspace decomposition in Equation~\eqref{4} coincides with the
% root-space decomposition \(\g=\h\oplus\bigoplus_{\alpha\in\Phi}\g_\alpha\), where
% \(\g_\alpha=Kx_\alpha\). Thus \(\g\) is graded by the root lattice
% \(\Lambda_r\), with \(\h\) in degree \(0\). 
The root-space decomposition of \(\g\) is a grading by \(\Lambda_r\), which 
induces a \(\Lambda_r\)-grading on \(\End\g\). Since \(\Der\g\) is a graded subspace
of \(\End\g\), it suffices to consider a homogeneous derivation
\(D\in(\Der\g)_\gamma\), with \(\gamma\in\Lambda_r\).

Assume first that \(\gamma\neq0\). If \(\gamma\notin\Phi\), then
\(D(\h)\subseteq\g_\gamma=0\), so \(D(\h)=0\). For \(h\in\h\) and
\(x_\alpha\in\g_\alpha\), the derivation identity gives
\([h,D(x_\alpha)]=D([h,x_\alpha])=\bar\alpha(h)D(x_\alpha)\). Hence
\(D(x_\alpha)\) has \(\h\)-weight \(\bar\alpha\). On the other hand,
homogeneity gives \(D(x_\alpha)\in\g_{\alpha+\gamma}\), so \(\bar\alpha=\overline{\alpha+\gamma}\) unless \(D(x_\alpha)=0\). If \(\alpha+\gamma \notin \Phi \cup\{0\}\), then \(\g_{\alpha+\gamma}=0\), hence \(D(x_\alpha)=0\). If
\(\alpha+\gamma\in\Phi\), then since \(\alpha+\gamma\neq\alpha\), by hypothesis, \(\overline{\alpha+\gamma}\neq\bar\alpha\), so we must have \(D(x_\alpha)=0\). If \(\alpha+\gamma=0\), then
\(D(x_\alpha)\in\h\), whose \(\h\)-weight is \(0\), whereas
\(\bar\alpha\neq0\). In all cases \(D(x_\alpha)=0\). Thus \(D=0\).

Now suppose \(0 \neq \gamma=\alpha\in\Phi\). Then \(D(\h)\subseteq\g_\alpha\), so
there is a linear form \(f\cl\h\to K\) such that
\(D(h)=f(h)x_\alpha\). Applying \(D\) to the relation \([h',h'']=0\) gives
\(f(h')\bar\alpha(h'')=f(h'')\bar\alpha(h')\). Since \(\bar\alpha\neq0\), we
have \(f=\mu\bar\alpha\) for some \(\mu\in K\). Hence \(\widetilde D:=D+\mu\,\ad(x_\alpha)\) annihilates \(\h\). The preceding
argument, applied to \(\widetilde D\), shows that \(\widetilde D(x_\beta)=0\)
for every \(\beta\in\Phi\). Hence \(\widetilde D=0\), and \(D=-\mu\,\ad(x_\alpha)\in\ad\g\). Therefore every homogeneous derivation of
nonzero degree is inner.

It remains to consider the case \(\gamma=0\). Then \(D\) preserves \(\h\) and each
root space \(\g_\alpha\). We claim that \(D(\h)=0\). Indeed, for \(h\in\h\) and
\(x_\alpha\in \g_\alpha\),
\[[D(h),x_\alpha]=D([h,x_\alpha])-[h,D(x_\alpha)]=\bar\alpha(h)D(x_\alpha)-\bar\alpha(h)D(x_\alpha)=0.\]
Also \(D(h)\in\h\), and \(\h\) is abelian. Hence \(D(h)\) commutes with
\(\h\oplus\bigoplus_{\alpha\in\Phi}\g_\alpha=\g\). Since \(\g\) is simple, its center
is zero, so \(D(h)=0\). Let \(\Pi=\{\alpha_1,\dots,\alpha_\ell\}\) be the fixed base of \(\Phi\). Since each
\(\g_{\alpha_i}\) is one-dimensional, there exist scalars \(\mu_i,\nu_i\in K\) such
that \(D(x_{\alpha_i})=\mu_i x_{\alpha_i},\ D(x_{-\alpha_i})=\nu_i x_{-\alpha_i}
\ (1\le i\le \ell)\).
Applying \(D\) to the relation \([x_{\alpha_i},x_{-\alpha_i}]=h_i\) and using \(D(h_i)=0\), we obtain \(0=\left(\mu_i+\nu_i\right) h_i\). Since \(h_i \neq 0\), it follows that
\(D(x_{-\alpha_i})=-\mu_i x_{-\alpha_i}\).
Since \(\{\alpha_1,\dots,\alpha_\ell\}\) is a \(\mathbb Z\)-basis of \(\Lambda_r\), there exists
\(u\in \Hom_{\mathbb Z}(\Lambda_r,K)\) such that
\(u(\alpha_i)=\mu_i\ (1\le i\le \ell)\).
Let \(\tilde \delta\in \tilde{\ft}\) correspond to \(u\) under the identification \(\tilde{\ft}\cong\Hom_{\mathbb Z}(\Lambda_r,K)\), and let \(\delta:=d\pi(\tilde\delta)\in \ft\).
By construction, \(\delta\) acts trivially on \(\h\), and for every root \(\alpha\) it acts
on \(\g_\alpha\) as multiplication by \(u(\alpha)\). In particular,
\(\delta(x_{\alpha_i})=\mu_i x_{\alpha_i}\) and
\(\delta(x_{-\alpha_i})=-\mu_i x_{-\alpha_i}\).
Therefore, \(D-\delta\) annihilates the elements \(x_{\pm\alpha_i}\) of \(\g\). Since
these elements generate \(\g\) as a Lie algebra, it follows that \(D=\delta\in \ft\).
%
% We have shown that every homogeneous derivation of degree \(0\) lies in \(\ft\), and
% every homogeneous derivation of nonzero degree lies in \(\ad\g\). Hence
% \(\Der\g\subseteq\ad\g+\ft\).
% The reverse inclusion is clear, since
% \(\ad\g\subseteq \Der\g\) and
% \(\ft\subseteq\mathscr L(\overline{\mathbf T})\subseteq \Der\g\).
\end{proof}

\begin{remark}\label{rem:comparison-eld-kochetov}
\Cref{thm:derivations-of-g} recovers~\cite[Thm.~3.2]{Alb} under the additional hypothesis \(Z(\tilde\g)=0\) (in other words, \(p\) does not divide the determiant of the Cartan matrix). Then the map \(\tilde\h\to\ft\), \(h\mapsto\ad h\) is injective. Since
\(\dim\tilde\h=\operatorname{rk}\Phi=\dim\ft\), it is an isomorphism. Hence
\(\ft=\ad\tilde\h\subseteq\ad\g\), and therefore \(\Der\g=\ad\g+\ft=\ad\g\).
\end{remark}

We now explain how~\Cref{thm:derivations-of-g} implies the smoothness of
\(\op{\mathbf{Aut}}\g\). 
Recall that \(\mathbf{G}:=\Ad{\widetilde{\mathbf G}}\subseteq\op{\mathbf{Aut}}\tilde\g\) 
and let \(\Gamma=\Aut\Pi\) be the group of diagram automorphisms. 
Let \(\Gamma\) act on \(\widetilde G\) and \(G\) through the pinning given by \(\Pi\) and 
\((x_{\alpha_1},\ldots,x_{\alpha_\ell})\). The corresponding \(\Gamma\)-action on \(\tilde\g=\Lie(\widetilde G)\) permutes \(x_{\alpha_i}\) and \(x_{-\alpha_i}\), and the semidirect product \(A:=G\rtimes\Gamma\) acts on \(\tilde\g\). 
Let \(\mathbf A\) be the (smooth) affine group
scheme defined by \(A\), and let \(\iota\cl \mathbf A\to\op{\mathbf{Aut}}\tilde\g\) be the morphism of affine group schemes defined by the action of \(A\) on \(\tilde\g\) (see~\Cref{rem:reduced-homomorphism-to-group-scheme}). 
Set \(\psi=\pi\circ\iota\cl\mathbf{A}\to\op{\mathbf{Aut}}\g\).

\begin{proposition}\label{prop:aut-g-smooth-char-not2}
Let \(\g\) be a classical simple Lie algebra over \(K\). Assume \(\op{char} K\neq 2\), and assume moreover that if $\op{char}K=3$, then \(\g\) is not of type \(A_2\). Then the morphism
\(\psi\cl \mathbf A\to \op{\mathbf{Aut}}\g\)
is an isomorphism. In particular, \(\op{\mathbf{Aut}}\g\) is smooth.
\end{proposition}

\begin{proof}
For \(p\neq2\), outside the exceptional case \(A_2\) in characteristic \(3\),
Steinberg's description of automorphisms shows that
\(\psi_K\cl A\to\Aut\g\) is bijective (see~\cite[\S4.2,~\S\S4.5--4.7]{Rob}). 
% The map \(\iota_K\cl A\to \Aut\tilde{\g}\) is injective by construction, 
% and \(\pi_K\) is injective by~\cite[\S3.1]{Rob}; hence \(\psi_K\) is bijective.

We next show that the differential of \(\psi\) is also bijective. 
By construction, the restriction of \(\iota\) to \(\mathbf G\) is the inclusion 
\(\bG=\Ad\tbG\subseteq\op{\mathbf{Aut}}\tilde\g\),
hence \(d\iota\) identifies \(\Lie(\mathbf G)\) with a Lie subalgebra of
\(\Lie(\op{\mathbf{Aut}}\tilde\g)=\Der\tilde\g\). 
Since \(\pi\) is a closed embedding, \(d\pi\) is injective, and hence 
\(d\psi=d\pi\circ d\iota\) is injective. 
Now recall that \(\tilde\h=\mathscr{L}(\widetilde{\mathbf T})\subseteq \tilde\g\) and 
\(\frk t=d\pi(\tilde\ft)\subseteq \Der\g\), where $\tilde\ft=\mathscr{L}(\mathbf T)\subseteq \mathscr{L}(\mathbf G)$ and $ \mathbf T=\op{Ad}\widetilde{\mathbf{T}}$. 
Since the differential of \(\Ad\cl\widetilde{\mathbf G}\to\mathbf G\) is 
\(\ad\cl\tilde\g\to\Lie(\mathbf G)\), we have \(\ad\tilde\g \subseteq \Lie(\mathbf G)\). 
Note that \(d\pi(\ad x)=\ad(x+\frk{z})\) for all \(x \in \tilde\g\), so \(d\pi(\ad\tilde\g)=\ad\g\). 
Therefore \(d\pi(\mathscr{L}(\mathbf G))\supseteq \ad\g+\frk t\).
By~\Cref{thm:derivations-of-g}, we have \(\Der\g=\ad\g+\frk t\), and so 
\(d\psi=d\pi\circ d\iota\cl \Lie(\mathbf A)=\Lie(\mathbf G)\to \Der\g\) is surjective. 

We have shown above that the map on \(K\)-points
\(\psi_K\cl A\to \Aut\g\)
is bijective, and that the differential
\(d\psi\cl \Lie(\mathbf A)\to \Lie(\op{\mathbf{Aut}}\g)=\Der\g\)
is bijective. Since \(\mathbf A\) is smooth,~\cite[Thm.~A.50]{Alb} applies and shows that \(\psi\) is an isomorphism of
affine algebraic group schemes. In particular, \(\op{\mathbf{Aut}}\g\) is smooth.
\end{proof}

The hypothesis of~\Cref{thm:derivations-of-g} fails in characteristic \(2\) and also 
for type \(A_2\) in characteristic \(3\).
In this latter case, \(\g \cong \mathfrak{psl}_3(K)\), the closed embedding \(\psi\) is not
an isomorphism, and in fact the automorphism group has type \(G_2\), see \cite[\S 7.2]{Rob}. 
A proof based on octonions is given in \cite[Cor.~4.24,~Thm.~4.26]{Alb}, which can also be used to  establish the smoothness of \(\op{\mathbf{Aut}}\g\) in this case.

\subsubsection{Characteristic \(2\).}
As discussed in \cite[\S2.6]{Rob}, there are no simple Lie algebras of type \(A_1\), \(B_\ell\), \(C_\ell\), \(F_4\), or \(G_2\) in characteristic \(2\). For the types that remain, we verify smoothness directly by checking the equivalent condition \(\dim \Der\g=\dim \Aut\g\). The derivation side is obtained from the dimensions of the outer derivation algebras computed in \cite{outerderchar2}, while the automorphism side is obtained from Steinberg's description of the automorphism groups in characteristic \(2\) in~\cite[\S 4.5, \S 7.2]{Rob}.
 
Since \(\dim\ad\g=\dim\tilde\g-\dim Z(\tilde\g)\) and \(\dim Z(\tilde\g)\) is the nullity of the reduction of the Cartan matrix modulo \(2\), we can compute the dimension of \(\Der\g\) from \Cref{tab:dimensions}, which records \(\dim(\Der\g/\op{ad}\g)=\dim H^1(\g,\g)\) for
the simple Lie algebras \(\g\) (and also the corresponding dimension for \(\tilde\g\)).

\begin{table}[htbp]
\centering
\renewcommand{\arraystretch}{0.75}
\setlength{\tabcolsep}{8pt}
\begin{tabular}{>{\centering\arraybackslash}m{0.20\textwidth}
                >{\centering\arraybackslash}m{0.30\textwidth}
                >{\centering\arraybackslash}m{0.15\textwidth}
                >{\centering\arraybackslash}m{0.14\textwidth}}
\toprule
Type of root system \(\Phi\) & \(\ell\) & \(\dim H^1(\g,\g)\) & \(\dim H^1(\tilde\g,\tilde\g)\) \\
\midrule
\(A_{\ell}\) & \(3\) & \(7\) & \(1\) \\[4pt]
\(A_{\ell}\) & \(\ell>3,\ \ell\equiv 1 \pmod{2}\) & \(1\) & \(1\) \\[4pt]
\(A_{\ell}\) & \(\ell\equiv 0 \pmod{2}\) & \(0\) & \(0\) \\[4pt]
\(D_{\ell}\) & \(4\) & \(26\) & \(2\) \\[4pt]
\(D_{\ell}\) & \(\ell>4,\ \ell\equiv 0 \pmod{2}\) & \(2\ell+2\) & \(2\) \\[4pt]
\(D_{\ell}\) & \(\ell\equiv 1 \pmod{2}\) & \(2\ell+1\) & \(1\) \\[4pt]
\(E_6,E_8\) & --- & \(0\) & \(0\) \\[4pt]
\(E_7\) & --- & \(1\) & \(1\) \\
\bottomrule
\end{tabular}
\vspace{6pt}
\caption{Dimensions of spaces of outer derivations of classical simple Lie algebras in characteristic \(2\)}
\label{tab:dimensions}
\end{table}

\subsubsection{Smoothness and its consequences}

\begin{theorem}\label{thm:aut-g-smooth}
Let \(\g\) be a classical simple Lie algebra over any field \(K\) of any
characteristic. Then the automorphism group scheme \(\op{\mathbf{Aut}}\g\) is smooth.
\end{theorem}

\begin{proof}
Since smoothness is not affected by field extension, we may assume that \(K\) is algebraically closed.
For \(\op{char}K \neq 2\), the result follows from~\Cref{prop:aut-g-smooth-char-not2} together with the treatment of
\(A_2\) in characteristic $3$ in~\cite{Alb}. Let \(K\) be an algebraically closed field of characteristic \(2\), so \(\g\) has one of the following types:
\[
A_{\ell}\ (\ell\ge 2),\qquad
D_{\ell}\ (\ell\ge 3),\qquad
E_6,\ E_7,\ E_8.
\]
Except for \(\g\) of type \(D_\ell\), \(\Aut\g\) has the same type as \(\g\), and for \(D_\ell\) including \(D_3=A_3\), the type of \(\Aut\g\) is recorded in \Cref{tab:exceptional-aut-groups}. 
In all cases, one verifies from \Cref{tab:dimensions} that \(\dim \Der\g=\dim \Aut\g\).
Hence the automorphism group scheme \(\op{\mathbf{Aut}}\g\) is smooth.
%For \(\op{char}K=2\), it follows from~\Cref{prop:aut-g-smooth-char2}.
\end{proof}

\begin{table}[htbp]
\centering
\renewcommand{\arraystretch}{0.75}
\setlength{\tabcolsep}{8pt}
\begin{tabular}{>{\centering\arraybackslash}m{0.20\textwidth}
                >{\centering\arraybackslash}m{0.25\textwidth}
                >{\centering\arraybackslash}m{0.20\textwidth}}
\toprule
\(p=\op{char}K\) & Type of \(\g\) & Type of \(\Aut\g\) \\
\midrule
\(3\) & \(A_2\) & \(G_2\) \\[4pt]
\(2\) & \(D_4\) & \(F_4\) \\[4pt]
\(2\) & \(D_{\ell}\) (\(\ell=3\) or \(\ell \geq 5\)) & \(C_{\ell}\) \\
\bottomrule
\end{tabular}
\vspace{6pt}
\caption{Type of \(\Aut\g\) in the exceptional cases}
\label{tab:exceptional-aut-groups}
\end{table}

As before, let \(G=\Ad\widetilde{G}\) and let 
\(\mathbf{A}\) be the (smooth) affine group scheme defined by \(A=G\rtimes\Gamma\), so we have an action \(\iota\cl\mathbf{A}\to\op{\mathbf{Aut}}\tilde\g\).

\begin{corollary}\label{cor:pi-isomorphism}
In the nonexceptional cases, i.e., those not listed in \Cref{tab:exceptional-aut-groups}, the morphisms 
\(\mathbf A\xrightarrow{\iota}
\op{\mathbf{Aut}}\tilde\g
\xrightarrow{\pi}
\op{\mathbf{Aut}}\g\)
are isomorphisms. 
\end{corollary}

\begin{proof}
As in \Cref{prop:aut-g-smooth-char-not2}, consider \(\psi=\pi\circ\iota\).
Steinberg's description of automorphisms shows
that \(\psi_K\) is bijective, and hence \(\pi_K\) is surjective. Since \(d\pi\) is injective by \Cref{lem:faithful-pi} and \(\op{\mathbf{Aut}}\g\) is smooth by
\Cref{thm:aut-g-smooth}, one has
\(\dim\op{\mathbf{Aut}}\tilde\g
\le \dim\Der\tilde\g
\le \dim\Der\g
= \dim\op{\mathbf{Aut}}\g\).
Since \(\pi_K\) is surjective, we have
\(\dim\op{\mathbf{Aut}}\tilde\g
\geq \dim\op{\mathbf{Aut}}\g\).
Thus all the displayed inequalities are equalities,
so \(\op{\mathbf{Aut}}\tilde\g\) is smooth and \(d\pi\) is bijective.
Therefore \(\pi\) is an isomorphism by~\cite[Thm.~A.50]{Alb}. It follows that \(\iota_K\) is bijective. Since \(d\iota\) is the inclusion \(\Lie(G) \to \Der\tg\), it is bijective by dimension count, hence \(\iota\) is an isomorphism.  
\end{proof}

Recall from \Cref{lem:der-alg-G-ad} that the simple Lie algebra \(\g\cong\ad\tilde\g\) can be obtained as the derived algebra of \(\Lie(G)\). 
Denote \(\frk a=\Lie(G)=\Lie(A)\) and consider the morphism \(\pi'\cl \op{\bf{Aut}}\frk{a} \to \op{\bf{Aut}}\frk a^{(1)}\)  defined by restriction. Since \(\frk a=\Der\tilde\g\) by \Cref{cor:pi-isomorphism}, conjugation defines a morphism
\(\Int\colon\bAut\tilde\g\to\op{\mathbf{Aut}}\frk a\). 
%where \(\Int_R(\varphi)(D):=\varphi D\varphi^{-1}\).

\begin{corollary}\label{cor:all-automorphism-maps-isomorphisms}
In the nonexceptional cases, i.e., those not listed in \Cref{tab:exceptional-aut-groups}, all the morphisms in the sequence
\(\mathbf A\xrightarrow{\iota}
\op{\mathbf{Aut}}\tilde\g
\xrightarrow{\Int}
\op{\mathbf{Aut}}\frk a
\xrightarrow{\pi'}
\op{\mathbf{Aut}}\frk a^{(1)}\)
are isomorphisms.
\end{corollary}

\begin{proof}
Consider the isomorphism \(\vartheta\cl \g \to \frk a^{(1)}\) that sends \(x+\frk z \to \ad(x)\), where \(\frk z=Z(\tg)\), and the corresponding isomorphism \(\theta\cl\bAut\g\to\bAut\frk a^{(1)}\).  We claim that
\(\theta\circ\pi=\pi'\circ\Int\). Indeed, let \(\mathcal R\) be a commutative associate unital \(K\)-algebra and let
\(\varphi\in(\op{\mathbf{Aut}}\tilde\g)(\mathcal R)\). Since \(\frk a^{(1)}=\ad\tilde\g\), let \(x\in\tilde\g_{\mathcal R}\) and consider 
\(\Int_{\mathcal R}(\varphi)(\ad x)=\varphi(\ad x)\varphi^{-1}=\ad(\varphi(x))\).
Consequently,
\begin{align*}
(\pi'_{\mathcal R}\circ\Int_{\mathcal R})(\varphi) \bigl[\vartheta_{\mathcal R}\bigl(x+\frk z_{\mathcal R}\bigr)\bigr]
&=\Int_{\mathcal R}(\varphi)(\ad x)
 =\ad(\varphi(x))
\\ &=\vartheta_{\mathcal R}(\varphi(x)+\frk z_{\mathcal R})
=\vartheta_{\mathcal R}\bigl[\pi_{\mathcal R}(\varphi)(x+\frk z_{\mathcal R})\bigr].
\end{align*}
Thus \((\pi'_{\mathcal R}\circ\Int_{\mathcal R})(\varphi)=\vartheta_{\mathcal R}\pi_{\mathcal R}\vartheta_{\mathcal R}^{-1}\), as claimed.  

By \Cref{cor:pi-isomorphism}, \(\pi\) and \(\iota\) are isomorphisms. Next we show that \(C_{\frk a}(\frk a^{(1)})=0\). Let
\(D\in C_{\frk a}(\frk a^{(1)})\). Since \(\frk a^{(1)}=\ad\tilde\g\), we have
\(0=[D,\ad x]=\ad(D(x))\) for every \(x\in\tilde\g\). Hence
\(D(\tilde\g)\subseteq Z(\tilde\g)\), so \(d\pi(D)=0\). Since \(d\pi\) is injective, we get \(D=0\), proving the claim. Now
\Cref{lem:faithful-restriction-to-ideal} shows that \(\pi'\) is a closed embedding. Since \(\pi' \circ \Int\) is an isomorphism, it follows that \(\pi'\) is an isomorphism. Consequently,
\(\Int=(\pi')^{-1}\circ\pi\)
is also an isomorphism. 
\end{proof}

\begin{theorem}\label{thm:aut-group-scheme-classical}
Let \(\g\) be a classical simple Lie algebra over an algebraically closed field \(K\) of any characteristic. Let \(G\) be the adjoint simple algebraic group whose type
agrees with that of \(\g\) in the nonexceptional cases and is given
in~\Cref{tab:exceptional-aut-groups} in the exceptional cases. Let \(\Gamma\)
be the group of diagram automorphisms of \(G\), \(A=G\rtimes\Gamma\), and 
\(\mathbf A\) the corresponding smooth affine group scheme. Set
\(\frk a:=\Lie(\mathbf A)=\Lie(\mathbf G)\) and let \(\frk m\) be the following ideal: \(\frk m=\frk a^{(1)}\) in the nonexceptional cases and \(\frk m\) is spanned by \(x_\alpha\) and \(h_\alpha\) for all short roots of \(\frk a\).
Then \(\frk m\) is \(\mathbf A\)-invariant and isomorphic to \(\g\), and restriction defines an isomorphism of affine group schemes
\(\psi\colon\mathbf A\rightarrow\op{\mathbf{Aut}}\g\). 
\end{theorem}

\begin{proof}
The nonexceptional cases are covered by \Cref{cor:all-automorphism-maps-isomorphisms}, 
so we consider the exceptional cases: \(\mathbf G\) has type \(G_2\) in characteristic \(3\), or type \(F_4\) or \(C_\ell\) in characteristic \(2\), so \(\Gamma\) is trivial.
The fact that \(\frk m\) is an ideal isomorphic to \(\g\) is proved in \cite[\S 2]{Rob}, and the fact that \(\frk m\) is \(G\)-stable in \cite[\S 7]{Rob}, which implies that it is \(\mathbf{G}\)-stable since \(\mathbf{G}\) is smooth. Moreover, Steinberg's description of automorphisms shows that 
\(\psi_K\colon A\to\Aut\g\) is a bijection. 
The kernel of the map \(d\psi\) is \(C_{\frk a}(\frk m)\), which is zero according to \cite{Rob},  
% Write
% \(\frk a=
% \frk t\oplus\bigoplus_{\beta\in\Phi}Kx_\beta\)
% for the root-space decomposition with respect to the maximal torus \(T\subset G\). Let
% \(x\in C_{\frk a}(\frk m)\), and write
% \(x=h+\sum_{\beta\in\Phi}c_\beta x_\beta,
% \ h\in\frk t\).
% Since \(\frk m\) contains \(x_\alpha\) for every short root
% \(\alpha\), we have \([x,x_\alpha]=0\) for every such \(\alpha\).
% For each \(\beta\in\Phi\), a direct inspection of the root systems
% \(G_2\), \(F_4\), and \(C_\ell\) shows that there exists a short root
% \(\alpha\) such that
% \([x_\beta,x_\alpha]\neq0\).
% Projecting the equality
% \([x,x_\alpha]=0\) onto the root space
% \(\frk a_{\beta+\alpha}\), or onto \(\frk t\) when
% \(\beta=-\alpha\), gives
% \(c_\beta[x_\beta,x_\alpha]=0\).
% Hence \(c_\beta=0\) for every \(\beta\in\Phi\), and therefore
% \(x=h\in\frk t\).
% Now, for every short root \(\alpha\), we have 
% \(0=[h,x_\alpha]=\bar\alpha(h)x_\alpha\),
% and hence \(\bar\alpha(h)=0\). Since the short roots generate the root lattice, it follows that \(h=0\) because \(G\) is of adjoint type and hence \(\frk{X}(T)\) is the root lattice. Thus
% \(C_{\frk a}(\frk m)=0\),
so \(d\psi\) is injective.
\end{proof}

\subsection{Cyclic gradings on classical simple Lie algebras}\label{subsec-cyclic-gradings-on-simple}

We now specialize the classification result obtained in~\Cref{sec:mu-embeddings-semidirect}, namely \Cref{thm:final-classification-mu-embeddings}, to the
automorphism group scheme of a classical simple Lie algebra \(\g\). Recall from \Cref{sec:duality} that a \(\bb{Z}_m\)-grading of \(\g\) is equivalent to a morphism of affine group schemes
\(\eta\cl\boldsymbol{\mu}_m\to\op{\mathbf{Aut}}\g\), and two such gradings are
isomorphic if and only if the corresponding morphisms are conjugate under
\(\Aut\g\). By \Cref{thm:aut-group-scheme-classical}, \(\bAut\g\) is the smooth affine group scheme associated to \(A=G\rtimes\Gamma\), where $G$ is an adjoint  simple algebraic group, of the same type as \(\g\) except in some cases in characteristic \(2\) and \(3\).   

A \(\Z_m\)-grading of \(\g\) is called   \emph{inner} if its associated morphism
\(\boldsymbol{\mu}_m\to\op{\mathbf{Aut}}\g\) is of inner type, equivalently if
its image is contained in \(\mathbf{G}\). It is said to be of \emph{outer type \(\sigma\)} if the associated morphism has outer type \(\sigma\neq 1\), in which case the order \(q\) of \(\sigma\) is \(2\) or \(3\). 
%Accordingly, an inner grading is said to be of \emph{Type I}, and an outer grading is said to be of \emph{Type II} if \(q=2\) and of \emph{Type III} if \(q=3\).

\begin{theorem}\label{thm:cyclic-gradings-final-theorem}
Let \(\g\) be a classical simple Lie algebra over an algebraically closed field \(K\). Let \(T\) be a maximal torus of \(\Aut\g=G \rtimes \Gamma\) and \(\Phi\) the associated root system with 
diagram automorphism group \(\Gamma\). For \(\sigma \in \Gamma\), let \(\Phi_\sigma\) be the folded root system, let \(X=\frk X(T^\sigma)=\Z\Phi_\sigma\), and let 
\(W^{(\sigma)}_{\mathrm{lat}}=W^{(\sigma)}_{0,\mathrm{lat}} \rtimes W^\sigma\)  
be the \(\sigma\)-Weyl group as in Equation~\eqref{eq:sigma-weyl-group-lattice}. 
\begin{enumerate}
\item[(1)] The set of isomorphism classes of inner $\Z_m$-gradings on \(\g\) is in bijection with the $(W\rtimes \Gamma)$-orbits on $\Hom_{\Z}(\frk X(T),\Z_m)$. 
\item[(2)] Let \(\sigma\in\Gamma\) be of order \(q \neq 1\). If \(\operatorname{char}K=p>0\) and \(p\mid q\), then there are no \(\mathbb Z_m\)-gradings of outer type \(\sigma\). Otherwise, the set of isomorphism classes of \(\mathbb Z_m\)-gradings on \(\g\) of outer type \(\sigma\) is in bijection with the set of \(W^{(\sigma)}_{\mathrm{lat}}\)-orbits on
\(\Hom_{\mathbb Z}(X\oplus\mathbb Z_q,\mathbb Z_m)_\sigma\). 
\end{enumerate}
\end{theorem}

\begin{proof}
(1) The case \(\sigma=\mathrm{id}\) of~\Cref{thm:final-classification-mu-embeddings} 
parametrizes morphisms \(\boldsymbol{\mu}_m\to\mathbf{G}\) up to \(G\)-conjugation by 
the \(W^{(\mathrm{id})}\)-orbits on \(\Hom_{\mathbb Z}\bigl(\frk X(T),\mathbb Z_m\bigr)\).
Passing from \(G\)-conjugacy to \((\Aut\g)\)-conjugacy introduces the diagram-automorphism 
action of \(\Gamma\). 
Hence the isomorphism classes of inner gradings are parametrized by the \((W\rtimes\Gamma)\)-orbits.

(2) Let \(\sigma\neq\mathrm{id}\). If \(p\mid q\), then the block \(\mathcal H_\sigma\) of \(\Hom(\boldsymbol{\mu}_m,\mathbf{G}\rtimes\mathbf{\Gamma})\) is empty, so there are no gradings of outer type \(\sigma\).
Assume therefore that \(p\nmid q\). By~\Cref{thm:final-classification-mu-embeddings}, the \(G\)-conjugacy classes in the block \(\mathcal{H}_\sigma\) are
parametrized by the \(W_{\mathrm{lat}}^{(\sigma)}\)-orbits on
\(\Hom_{\mathbb Z}\bigl(X\oplus\mathbb Z_q,\mathbb Z_m\bigr)_\sigma\)
where \(X=\frk{X}(S)\) and \(S=T^\sigma\) (see \Cref{rem:connectedTsigma}).
By \Cref{rem:irreducible-centralizer-condition}, in this case \(G\)-conjugacy and \((\Aut\g)\)-conjugacy coincide.
\end{proof}

The classification in~\Cref{thm:cyclic-gradings-final-theorem} depends only on the root datum 
of \(G=(\Aut\g)^\circ\) and not on the field \(K\), except for the condition
\(\operatorname{char}K\nmid\operatorname{ord}(\sigma)\) and the fact that the type of \(G\) 
may be different from that of \(\g\) (see \Cref{tab:exceptional-aut-groups}). 
We may therefore use the usual Kac-coordinate description from characteristic \(0\),
which we very briefly recall for completeness (see e.g. \cite[\S 3.6,\S 3.10]{Vin} or \cite[\S2,\S3]{reeder}) for details. Let \(Y=\frk{Y}(T^\sigma)\) and \(E=Y\otimes_\Z\mathbb{R}\) (cf. \Cref{rem:action_real_torus}).

In the inner case, \(X\) is the root lattice of \(\Phi\) and \(Y\) is the weight lattice of \(\Phi^\vee\), so homomorphisms \(\varrho\cl X\to\Z_m\) are represented by points in \(E/Y\), 
or equivalently orbits in \(E\) under \(Y\)-translations. The orbit corresponding to \(\varrho\) is \((x_1,\ldots,x_\ell)+Y\) where \(x_i\in\frac{1}{m}\varrho(\alpha_i)+\Z\). 
The group \(\widetilde{W}:=Y\rtimes W\) acts by affine transformations of \(E\), 
and its normal subgroup \(\widetilde{W}_r:=(\Z\Phi^\vee)\rtimes W\), known as the affine Weyl 
group, is generated by reflections and acts simply transitively on the so-called Weyl alcoves. 
It follows that \(\widetilde{W}=\widetilde{W}_r\rtimes\Omega\) where \(\Omega\) is the 
stabilizer of one alcove in \(\widetilde{W}\). (Note that \(\Omega\cong Y/\Z\Phi^\vee\), 
the fundamental group of \(G\).) The standard alcove is defined by \(x_i\ge 0\) for \(0\le i\le\ell\) where \(x_0:= 1-\sum_{i=1}^\ell n_i x_i\) and \(n_i\) are the positive integers expressing the highest root \(\delta\) as \(\sum_{i=1}^\ell n_i\alpha_i\); they are known as the 
labels of the (untwisted) affine Dynkin diagram of \(\Phi\), with the additional 
``affine node'' \(\alpha_0:=-\delta\) assigned the label \(n_0:=1\). 
Hence each \((W\rtimes\Gamma)\)-orbit in \(E/Y\) has a representative in the alcove, 
with two points representing the same orbit if and only if they 
are conjugate by the action of \(\Omega\rtimes\Gamma\), which turns out to be the full symmetry 
group of the alcove. In particular, each isomorphism class of \(\Z_m\)-gradings is represented
by nonnegative integers \((s_0,\ldots,s_\ell)\), called the \emph{inner Kac coordinates}, 
satisfying \(\sum_{i=0}^{\ell}n_i s_i=m\). 

For an outer type \(\sigma\ne\mathrm{id}\), a similar analysis yields the 
\emph{outer Kac coordinates}, which are nonnegative integers \((p_0,\ldots,p_{\ell(\sigma)})\) 
satisfying \(\sum_{i=0}^{\ell(\sigma)}n_i p_i=m\), where \(n_i\) are the labels of the
corresponding twisted affine diagram, which can be defined as follows. 
Recall from Equation~\eqref{eq:sigma-eigenspac-decom} the grading 
\(\g=\bigoplus_{\bar k\in\Z_q}\g_{\bar k}\) defined by \(\sigma\) on the 
complex simple Lie algebra with root system \(\Phi\). Then the root system of \(\g_{\bar 0}\) 
is the reduced system \(\Psi\) corresponding to \(\Phi_\sigma\), 
\(n_i:=q n'_i\) where the positive integers \(n'_i\) are the coefficients expressing 
the highest weight \(\delta_\sigma\) of the \(\g_{\bar 0}\)-module \(\g_{\bar 1}\) 
as a linear combination of a base \(\{\beta_1,\ldots,\beta_{\ell(\sigma)}\}\) of \(\Psi\), 
and the ``affine node'' \(\beta_0:=(-\delta_\sigma,\bar{1})\in X\oplus\Z_q\) 
is assigned the label \(n_0:=q\).

In either case, the Kac coordinates are defined up to permutations 
that correspond to automorphisms of the appropriate affine Dynkin diagram.  

We finish by recording how the Kac coordinates determine the homogeneous components
of the corresponding \(\mathbb Z_m\)-grading of \(\g\) over \(K\), which is obtained from
the \(\sigma\)-root decomposition of \(\frk{a}:=\Lie(G)\) by restricting to \(\g\) and 
coarsening via the homomorphism \(\varrho\cl X\oplus\Z_q\to\Z_m\) 
defined by the Kac coordinates. 
Note that the \(\Z_q\)-grading defined by \(\sigma\) is a coarsening of this \(\Z_m\)-grading 
(cf. \Cref{fig:coarsening-refinement-diagram}) and also its special case (corresponding to
the \(0\)-th coordinate equal to \(1\) and all others \(0\)).

\begin{remark}\label{rem:almost_fine}
Since \(\Aut\g=G\rtimes\Gamma\), \(T^\sigma\) is a maximal torus of \(G^\sigma\), and 
\(S^{(\sigma)}=\langle T^\sigma,\sigma\rangle\), the \(\sigma\)-root decomposition of 
\(\g\) is, in the terminology of \cite{almostfinegrading}, the canonical almost fine refinement associated to 
either the \(\Z_m\)- or \(\Z_q\)-grading. Moreover, since the normalizer of \(S^{(\sigma)}\) 
in \(\Aut\g\) is \(N= N_G(S^{(\sigma)})\rtimes N_\Gamma(\langle\sigma\rangle)\), 
and hence the centralizer is \(C=T^\sigma\times\langle\sigma\rangle=S^{(\sigma)}\), 
the quasitorus \(S^{(\sigma)}\) is maximal in \(\Aut\g\), 
so the \(\sigma\)-root decomposition of \(\g\) is actually a fine grading 
(i.e., does not admit proper refinements), and \(\frk{X}(S^{(\sigma)})=X\oplus\Z_q\)
is its universal group (see \cite[\S 1.2, \S 1.4]{Alb}). 
Moreover, the Weyl group of this fine grading is, by definition, \(N/C\), which is isomorphic to
\(W\rtimes\Gamma\) in the inner case and \(W^{(\sigma)}\cong W^{(\sigma)}_\mathrm{lat}\) 
in the outer case (cf. Theorem 4.3 in \cite{almostfinegrading}).
\end{remark}

First consider an inner grading, and let \((s_0,s_1,\ldots,s_\ell)\) be its
Kac coordinates, so that \(\sum_{i=0}^{\ell}n_i s_i=m\). Then \(\varrho\cl\Z\Phi\to\Z_m\) 
is defined by \(\varrho(\alpha_i)=\bar s_i\) for \(0\leq i\leq\ell\). 
%Since \(\alpha_0=-\sum_{i=1}^{\ell}n_i\alpha_i\), the relation \(\sum_{i=0}^{\ell}n_i s_i=m\) gives \(\varrho(\alpha_0)=\bar s_0\).
Let \(\Phi_{\g}:=\Phi\) in the nonexceptional cases and \(\Phi_{\g}:=\Phi_{\mathrm{sh}}\) 
in the exceptional cases (see \Cref{tab:exceptional-aut-groups}), where
\(\Phi_{\mathrm{sh}}\) denotes the set of short roots in \(\Phi\). 
Then the \(T\)-weight decomposition of \(\g\) is
\(
\g=\g_0\oplus\bigoplus_{\alpha\in\Phi_{\g}}\g_\alpha
\) 
where \(\g_0=\g\cap\frk{a}_0\) and \(\g_\alpha=\frk{a}_\alpha\) for \(\alpha\in\Phi_{\g}\).
Then the homogeneous component of degree \(\bar\jmath\in\mathbb Z_m\) is given by
\[
\g_{\bar\jmath}=\bigoplus_{\substack{\alpha=\sum_{i=1}^{\ell}m_i\alpha_i \in\Phi_{\g}\cup\{0\}\\
\sum_{i=1}^{\ell}m_i s_i\equiv j\;(\operatorname{mod}m)}}\g_\alpha.
\]
% In the nonexceptional cases this reduces to
% \[\g_{\bar\jmath}=\bigoplus_{\substack{\alpha=\sum_{i=1}^{\ell}m_i\alpha_i\in\Phi\cup\{0\}\\
% \sum_{i=1}^{\ell}m_i s_i\equiv j\;(\operatorname{mod}m)}}\g_\alpha,\]
% whereas in the exceptional cases the same formula holds with \(\Phi\) replaced by the set \(\Phi_{\mathrm{sh}}\) of short roots of \(G\).

Now consider an outer grading of type \(\sigma\). In the exceptional cases, the group \(G\) 
has type \(G_2\), \(F_4\), or \(C_\ell\), hence all gradings in these cases are inner. 
The same is true in the nonexceptional cases if \(\operatorname{char}K=2\).
So, in the outer case, \(\operatorname{char}K\ne 2\) and \(\g\) is the derived algebra of 
\(\frk{a}=\Lie(G)\). 
Recall the \(\sigma\)-root decomposition of \(\frk{a}\):
\[
\frk{a}=\bigoplus_{(\beta,\bar k)\in\Phi^{(\sigma)}\cup\{(0,\bar0)\}}
\mathfrak a_{(\beta,\bar k)}.
\]
Since \(\operatorname{char}K\ne 2\), we have \(Z(\tilde\g)=0\) except for types \(A_\ell\) 
if \(\operatorname{char}K\) divides \(\ell+1\) and \(E_6\) 
if \(\operatorname{char}K=3\), and in these cases \(\sigma\) has order \(2\) and acts as the 
negative identity on \(Z(\tilde\g)\). It follows that 
\(\g_\chi=\frk{a}_\chi\) except in the above cases when \(\chi=(0,\bar{1})\). 
Finally, let \((p_0,p_1,\ldots,p_{\ell(\sigma)})\) be the Kac coordinates.
Then \(\varrho\colon X \oplus \Z_q \to\mathbb Z_m\) is defined by 
\(\varrho(\beta_i)=\bar p_i\), for \(0\le i\le \ell(\sigma)\).
% The relation \(\sum_{i=0}^{\ell(\sigma)}n_i p_i=m\) shows that this assignment is compatible with the relation among the affine simple \(\sigma\)-roots.
Let \(\beta_0\) be the ``affine node''. Since we defined 
\(\beta_0=(-\delta_\sigma,\bar1) \in X \oplus \Z_q\), we have
\((\beta,\bar k)=k\beta_0+\left(\beta+k\delta_\sigma,\bar0\right)\) for all \(\beta\in X\).
Writing \(\beta+k\delta_\sigma=\sum_{i=1}^{\ell(\sigma)}m_i\beta_i\), 
we get \((\beta,\bar k)=k\beta_0+ \sum_{i=1}^{\ell(\sigma)}m_i\beta_i\), and hence
\(\varrho(\beta,\bar k)=k\bar p_0+\sum_{i=1}^{\ell(\sigma)}m_i\bar p_i\) in \(\mathbb Z_m\).
Hence the components of the \(\mathbb Z_m\)-grading on \(\g\) are given by 
\[
\g_{\bar{\jmath}}= \bigoplus_{\substack{(\beta,\bar k)\in\Phi^{(\sigma)}\cup\{(0,\bar0)\}\\
k p_0+\sum_{i=1}^{\ell(\sigma)}m_ip_i\equiv j (\op{mod} m)}}\mathfrak g_{(\beta,\bar k)}.
\]
This completes the description of the homogeneous components in terms of the inner and outer 
Kac coordinates.

\bibliographystyle{plain} % We choose the "plain" reference style
\bibliography{bib} % Entries are in the bib.bib file

\bigskip

\end{document}